\documentclass{amsart}

\usepackage[utf8]{inputenc}
\usepackage{amssymb}
\usepackage{amsmath,tabu,amsthm}
\usepackage{mathtools}
\usepackage{color}
\usepackage{lscape}
\usepackage{hyperref}

\usepackage{tikz}
\usepackage{tikz-cd}
\usetikzlibrary{decorations.markings}
\usetikzlibrary{decorations.pathreplacing}
\usepackage{enumitem}
\usepackage{mathtools}
\usepackage[margin=1in]{geometry}
\usepackage{amsrefs}
\setenumerate[1]{label=\arabic*.}
\usepackage{bbm}

\DeclareFontFamily{U}{mathx}{\hyphenchar\font45}
\DeclareFontShape{U}{mathx}{m}{n}{
      <5> <6> <7> <8> <9> <10>
      <10.95> <12> <14.4> <17.28> <20.74> <24.88>
      mathx10
      }{}
\DeclareSymbolFont{mathx}{U}{mathx}{m}{n}
\DeclareFontSubstitution{U}{mathx}{m}{n}
\DeclareMathAccent{\widecheck}{0}{mathx}{"71}
\DeclareMathAccent{\wideparen}{0}{mathx}{"75}

\newcommand{\wh}{\widehat}                              	 
\newcommand{\MC}[1]{\mathcal{#1}}                          
\newcommand{\MB}{\mathbb}                            		 
\newcommand{\p}{\varphi}                                 		 
\DeclareMathOperator{\Hom}{Hom}                    	 
\DeclareMathOperator{\ind}{Ind}                          	 
\DeclareMathOperator{\res}{Res}                        	 
\DeclareMathOperator{\End}{End}                      	 
\DeclareMathOperator{\UnitModule}{\mathbbm{1}}         
\newcommand{\partitions}{\MC{P}}                               
\newcommand{\partitionsn}[1]{\MC{P}_{#1}}                 
\newcommand{\ydcell}{\;\framebox(7,7){}\;}	           

\newcommand{\Sy}[1]{\ensuremath{S_{#1}}}                                                      
\newcommand{\charrep}[1]{\chi^{#1}}                                                                  
\newcommand{\norcharrep}[1]{\widetilde{\chi}^{#1}}                                           
\newcommand{\primaryiso}{\p}									        
\newcommand{\loworderterms}{\text{l.o.t.}}						        
\newcommand{\Hcenter}{\End_{\MC{H}'}(\UnitModule)}  			                   
\newcommand{\partsofpartition}[2]{m_{#1}(#2)} 						  
\newcommand{\content}[1]{\text{cont}(#1)}						           
\newcommand{\JM}[1]{J_{#1}}										     
\newcommand{\pr}[1]{\mathrm{pr}_{#1}}								             
\newcommand{\conj}{\mathrm{Conj}}						     			     
\newcommand{\conjOdd}{\mathrm{Conj}_{odd}}						      

\newcommand{\ds}{\displaystyle}									  
\newcommand{\Heis}{\mathcal{H}_{tw}}								  
\newcommand{\Tr}{\operatorname{Tr}}							           
\newcommand{\TrH}{\Tr(\Heis)}										   
\newcommand{\Endid}{\text{End}_{\Heis}(\UnitModule)}							   
\newcommand{\Ser}{\mathbb{S}_n}								            
\newcommand{\Clif}{\mathcal{C\ell}_n}								    

\newcommand\numberthis{\addtocounter{equation}{1}\tag{\theequation}}              
\newcommand{\TrHEv}{\TrH_{\overline{0}}}							    

\newcommand{\Ind}{\operatorname{Ind}}								    
\newcommand{\regrep}{\tau}								                     
\newcommand{\Sern}[1]{\MB{S}_{#1}}								     
\newcommand{\supersym}{\Gamma}							                      
\newcommand{\oddpartitions}[1][]{\mathcal{OP}_{#1}}				             
\newcommand{\strictpartitions}[1][]{\mathcal{SP}_{#1}}					     
\newcommand{\stanshiftedstrict}[1]{g_{#1}}							      
\newcommand{\Wminus}{W^{-}}									     
\newcommand{\Winfty}{W_{1+ \infty}}								     
\newcommand{\length}[1]{\ell(#1)}           								     
\newcommand{\Schurgraph}{\mathbb{G}}								     
\newcommand{\pathsfromzero}[1]{h(#1)}	                                                                 
\newcommand{\shifted}[1]{S(#1)}									      
\newcommand{\lengthparity}{\delta}								              
\newcommand{\downtransition}[2]{p^{\downarrow}(#1,#2)}			              
\newcommand{\uptransition}[2]{p^{\uparrow}(#1,#2)}			              
\newcommand{\edgemulti}[2]{\kappa(#1,#2)}							      
\newcommand{\Plancherel}[1]{Pl_{#1}}								      
\newcommand{\SimpleSer}[1]{L^{#1}}								      
\newcommand{\shiftedpowersum}{\mathfrak{p}}
\newcommand{\KerovCoorUp}[1]{X(#1)}							               
\newcommand{\KerovCoorDown}[1]{Y(#1)}							      
\newcommand{\KerovCoorUpNoZero}[1]{X'(#1)}						      
\newcommand{\JMeig}[1]{s(#1)}									      
\newcommand{\upmoment}[1]{\mathbf{g}^{\uparrow}_{#1}}					      
\newcommand{\downmoment}[1]{\mathbf{g}^{\downarrow}_{#1}}				      
\newcommand{\typeBWeyl}[1]{B_{#1}}								       
\newcommand{\hyperoctahedral}[1]{\widehat{B}_{#1}}					      
\newcommand{\cliffordgroup}[1]{\Pi_{#1}}							              
\newcommand{\evencenter}[1]{Z(\Sern{#1})_{\overline{0}}}						              
\newcommand{\distinguishedperm}[2]{\sigma_{#1;#2}}                                                  
\newcommand{\octoclasssum}[2]{\wh{A}_{#1;#2}}						       
\newcommand{\Sergeevclasssum}[2]{A_{#1;#2}}						        
\newcommand{\centralidem}[1]{e_{#1}}                                                                       
\newcommand{\symtoHom}[1]{\MC{T}_{#1}}						                
\newcommand{\FockSpaceCat}[1][]{\mathfrak{S}_{#1}}					       
\newcommand{\FockSpaceFunctor}[1]{F^{\Heis}_{#1}}					        
\newcommand{\hyperSergeevProj}[1]{\pi_{#1}}           					        
\newcommand{\LcosHyper}[2]{\wh{\MC{LC}}^{#1}_{#2}}						
\newcommand{\LcosSer}[2]{\MC{LC}^{#1}_{#2}}							
\newcommand{\HyperChar}[1]{\wh{\chi}^{#1}}								
\newcommand{\SerChar}[1]{\chi^{#1}}								         
\newcommand{\SimpleHyper}[1]{\wh{L}^{\lambda}}					                  
\newcommand{\normSerChar}[1]{\widetilde{\chi}^{#1}}					          
\newcommand{\primaryisom}{\p}										  
\newcommand{\funonYD}{\text{Fun}(\strictpartitions,\MB{C})}				          
\newcommand{\supersymhom}[1]{F^{\supersym}_{#1}}					         

\newcommand{\refequal}[1]{\xy {\ar@{=}^{#1}
(-1,0)*{};(1,0)*{}};
\endxy}

\newcommand{\arrowlines}{%
\begin{tikzpicture}
  \draw[thick] (0,0) -- (-.08,-.08);
  \draw[thick] (0,0) -- (.08,-.08);
\end{tikzpicture}%
}

\newcommand{\redarrowlines}{%
	\begin{tikzpicture}
	\draw[thick,red] (0,0) -- (-.08,-.08);
	\draw[thick,red] (0,0) -- (.08,-.08);
	\end{tikzpicture}%
}

\newcommand{\smallydcell}{
\begin{tikzpicture}
    \draw (0,0) rectangle (.2,.2);
\end{tikzpicture}
}

\newcommand{\trup}{
\begin{tikzpicture}[baseline=(current bounding box).center]
\draw[thick] (2,2) circle (.5cm);
\node at (1.5,2) {\arrowlines};
\node at (2,2) {$*$};
\end{tikzpicture}}

\newcommand{\trdown}{
	\begin{tikzpicture}[baseline=(current bounding box).center]
	\draw[thick] (2,2) circle (.5cm);
	\node[rotate = 180] at (1.5,2) {\arrowlines};
	\node at (2,2) {$*$};
	\end{tikzpicture}}

\newcommand{\ckpicture}{
\begin{tikzpicture}
\draw[thick] (2,2) circle (.5cm);
\node at (1.5,2) {\arrowlines};
\draw[fill=black] (1.62,2.33) circle (.08cm);
\node at (1.3,2.7) {$2k$};
\end{tikzpicture}}

\newcommand{\cktildepicture}{
\begin{tikzpicture}
\draw[thick] (2,2) circle (.5cm);
\node[rotate = 180] at (1.5,2) {\arrowlines};
\draw[fill=black] (1.62,2.33) circle (.08cm);
\node at (1.3,2.7) {$2k$};
\end{tikzpicture}}

\newcommand{\op}{\omega_{+}}

\newcommand{\omitt}[1]{}

\newtheorem{theorem}{Theorem}[section]
\newtheorem{lemma}[theorem]{Lemma}
\newtheorem{proposition}[theorem]{Proposition}
\newtheorem{corollary}[theorem]{Corollary}
\newtheorem{remark}[theorem]{Remark}

\newenvironment{definition}[1][Definition]{\begin{trivlist}
\item[\hskip \labelsep {\bfseries #1}]}{\end{trivlist}}
\newenvironment{example}[1][Example]{\begin{trivlist}
\item[\hskip \labelsep {\bfseries #1}]}{\end{trivlist}}

\title[The center of the twisted Heisenberg category]{The center of the twisted Heisenberg category, factorial Schur $Q$-functions, and transition functions on the Schur graph}
\author{Henry Kvinge, Can Ozan O{\u g}uz, and Michael Reeks}
\date{\today}

\usepackage{graphicx}
\usepackage{color}

\begin{document}

\maketitle

\begin{abstract}
We establish an isomorphism between the center of the twisted Heisenberg category and the subalgebra $\Gamma$ of the symmetric functions generated by odd power sums. We give a graphical description of the factorial Schur $Q$-functions and inhomogeneous power sums as closed diagrams in the twisted Heisenberg category, and show that the bubble generators of the center correspond to two sets of generators of $\Gamma$ which encode data related to up/down transition functions on the Schur graph. Finally, we describe an action of the trace of the twisted Heisenberg category, the $W$-algebra $W^-\subset W_{1+\infty}$, on $\Gamma$.
\end{abstract}

\setcounter{tocdepth}{1}
\tableofcontents

\section{Introduction}

In~\cite{Kho14}, Khovanov describes a linear monoidal category $\mathcal{H}$ which conjecturally categorifies the Heisenberg algebra. The morphisms of $\mathcal{H}$ are governed by a graphical calculus of planar diagrams. This category has connections to many interesting areas of representation theory and combinatorics. The trace of $\mathcal{H}$, which can be defined diagrammatically as the algebra of diagrams on the annulus, was shown in~\cite{CLLS15} to be isomorphic to the $W$-algebra $W_{1+\infty}$ at level one. 
 The center of $\mathcal{H}$, which is the algebra $\End_{\mathcal{H}}(\UnitModule)$ of endomorphisms of the monoidal identity, was shown in~\cite{KLM16} to be isomorphic to the algebra of shifted symmetric functions $\Lambda^*$ of Okounkov and Olshanski~\cite{OO97}. 

A twisted version of Khovanov's Heisenberg category was introduced by Cautis and Sussan in ~\cite{CS15}. The twisted Heisenberg category $\Heis$ is a $\mathbb{C}$-linear additive monoidal category, with an additional $\mathbb{Z}/2\mathbb{Z}$-grading. It conjecturally categorifies the twisted Heisenberg algebra. The center of $\Heis$, $\Endid$, was studied in~\cite{OR17} where it was shown that as a commutative $\MB{C}$-algebra, 
\begin{equation*}
\Endid \cong \MB{C}[d_0,d_2,d_4, \dots ] \cong \MB{C}[\bar{d}_2,\bar{d}_{4},\bar{d}_{6}, \dots],
\end{equation*}
where $d_{2k}$ and $\bar{d}_{2k}$ are certain clockwise and counterclockwise bubble generators respectively. While symmetric groups play a central role for $\mathcal{H}$ in~\cite{Kho14}, finite Sergeev superalgebras $\{\Sern{n}\}_{n \geq 0}$ (also known as finite Hecke--Clifford algebras of type $A$) play the central role for $\Heis$. In particular, Cautis and Sussan construct a family of functors $\{\FockSpaceFunctor{n}\}_{n \geq 0}$ from $\Heis$ to bimodule categories of Sergeev algebras in order to categorify the Fock space representation of the twisted Heisenberg algebra. When restricted to $\Endid$, each functor gives a surjective algebra homomorphism $\FockSpaceFunctor{n}: \Endid \twoheadrightarrow \evencenter{n}$ where $\evencenter{n}$ is the even center of $\Sern{n}$. 

In this paper, we study the combinatorial and representation theoretic properties of $\Endid$. Our main result, Theorem~\ref{thm-main}, establishes an isomorphism 
$ \label{eqn-intro-main-iso} 
\primaryiso: \Endid \overset{\sim}{\longrightarrow} \supersym
$, 
where $\supersym$ is a subalgebra of the algebra of symmetric functions $\supersym = \MB{C}[p_1,p_3,p_5, \dots]$ ($\supersym$ is sometimes known as the algebra of supersymmetric~\cite{Iv01} or doubly symmetric~\cite{Pet09} functions). The construction of $\primaryiso$ relies on the fact that there are embeddings of both $\Endid$ and $\supersym$ into the algebra of functions on strict partitions, $\funonYD$. In our proof of Theorem~\ref{thm-main} we identify the images of certain algebraically independent generators of these algebras in $\funonYD$:  the closures of $n$-cycles from $\Endid$ and inhomogeneous analogues of odd power sums $\shiftedpowersum_{n}$ in $\supersym$. The latter were first investigated by Ivanov in his study of the asymptotic behavior of characters of projective representations of symmetric groups~\cite{Iv01}. We go on to identify the closure of idempotents of $\Sern{n}$ with scalar multiples of Ivanov's factorial Schur $Q$-functions. Intriguingly, the coefficients that appear on the image of idempotent closures when written in terms of factorial Schur $Q$-functions count the number of paths between specific vertices in the graph of all strict partitions (also known as the Schur graph). A similar phenomenon was observed in \cite{KLM16}. A dictionary between $\supersym$ and $\Endid$ is found in Table \ref{dictionary-table}.

In parallel to the surjective homomorphisms $\{\FockSpaceFunctor{n}\}_{n \geq 0}$ from $\Endid$ to $\{\evencenter{n}\}_{n \geq 0}$, for all $n \geq 0$ one can also construct surjective homomorphisms $\supersymhom{n}: \supersym \twoheadrightarrow \evencenter{n}$~\cite{Iv01}. Our isomorphism $\primaryiso$ is canonical in the sense that it intertwines the pair $\FockSpaceFunctor{n}$ and $\supersymhom{n}$ for each $n \geq 0$.
\begin{center}
\begin{tikzpicture}

\node at (-3,0) {$\Endid$};

\node at (3,0) {$\supersym$};

\node at (0,-2) {$\evencenter{n}$};

\draw[<->] (-1.9,0) --(2.5,0);
\draw[->>] (-2.1,-.3) -- (-.5,-1.7);
\draw[->>] (2.5,-.3) -- (.5,-1.7);

\node at (0,.3) {$\primaryiso$};
\node at (-1.8,-1.2) {$\FockSpaceFunctor{n}$};
\node at (1.8,-1.2) {$\supersymhom{n}$};

\end{tikzpicture}
\end{center}

One interesting feature of the center of the non-twisted Heisenberg category $\mathcal{H}$ is that, as shifted symmetric functions, the curl generators are best understood in terms of moments of Kerov's transition and co-transition measures on Young diagrams; fundamental tools used to answer probabilistic questions related to the asymptotic representation theory of symmetric groups~\cite{Ker93}. In this paper we show that this connection to asymptotic representation theory extends to the twisted Heisenberg category. Specifically, we identify the clockwise bubble generators $\{d_{2k}\}_{k \geq 0}$ and counterclockwise bubble generators $\{\bar{d}_{2k}\}_{k \geq 1}$ with two sets of algebraically independent generators for $\supersym$ discovered by Petrov~\cite{Pet09}, $\{\downmoment{k}\}_{k \geq 0}$ and $\{\upmoment{k}\}_{k \geq 0}$ respectively. The functions $\{\downmoment{k}\}_{k \geq 0}$ (respectively $\{\upmoment{k}\}_{k \geq 0}$) encode down (resp. up) Markov transition kernels on the Schur graph. In particular, the difference between up and down transition functions manifests itself graphically in $\Endid$ as a difference in orientation of diagrams. This seems to be yet another indication of the ``planar nature'' of structures arising from noncommutative probability theory.

There is a natural action of the trace of a category on its center, which can be diagrammatically defined as gluing annular diagrams (elements of the trace) around planar ones (elements of the center). In the case of Khovanov's Heisenberg category, the results of \cite{CLLS15} and \cite{KLM16} give rise to an action of $\Winfty$ on the algebra of shifted symmetric functions. This representation of $\Winfty$ was described in terms of symmetric group representation theory by Lascoux and Thibon in \cite{LT01}. The trace of $\Heis$ was shown in \cite{OR17} to be isomorphic to a classical-type subalgebra of $\Winfty$ discovered by Kac, Wang, and Yan called $\Wminus$ \cite{KWY98}. This result along with Theorem \ref{thm-main} gives a representation of $\Wminus$ on $\supersym$. In this paper we describe this representation which is a twisted version of the representation described in \cite{LT01}. 

The paper is structured as follows. In Section \ref{section-schur-sergeev} we describe necessary background material on Schur's graph and the representation theory of Sergeev algebras. In Section \ref{sect-gamma} we describe the subalgebra $\supersym$ of the symmetric functions and several of its bases. In Sections \ref{sect-twisted-Heisenberg} we recall the definition of the twisted Heisenberg category $\mathcal{H}_{tw}$ and review the functors $\{\FockSpaceFunctor{n}\}_{n \geq 0}$. In Section \ref{Section-main} we first establish the isomorphism between $\Endid$ and $\supersym$, and then describe the $W$-algebra $\Wminus$ and its induced action on $\supersym$.

\textbf{Acknowledgements:} The authors would like to thank Weiqiang Wang and Nick Ercolani for many helpful discussions. The second author was partially supported by grants DMS-1664240 and DMS-1255334.

\section{The Schur graph and Sergeev algebras} \label{section-schur-sergeev}
  
\subsection{Transition functions on the Schur graph}\label{section-shifted-YD}
Let $\partitionsn{n}$ be the set of all partitions of $n$ and set
\begin{equation*}
\partitionsn{} := \bigcup_{n \geq 0} \partitionsn{n}.
\end{equation*}
We freely identify a partition $\rho$ with its corresponding Young diagram. If $\rho \in \partitionsn{n}$ then we write $|\rho| = n$. If $\rho = (\rho_1,\rho_2,\dots,\rho_r)$ and $\eta = (\eta_1, \eta_2, \dots, \eta_t) \in \partitions$ then we write $\eta \subset \rho$ when $\eta_i \leq \rho_i$ for all $i \geq 1$. We denote the number of parts (or length) of a partition $\rho$ by $\length{\rho}$. A partition $\mu = (\mu_1, \dots, \mu_r) \in \partitionsn{n}$ is called an \emph{odd partition} if $\mu_i$ is odd for all $1 \leq i \leq r$. We denote the collection of odd partitions of $n$ by $\oddpartitions[n]$ and set $\oddpartitions{} := \bigcup_{n \geq 0} \oddpartitions[n]$. 

We call a partition $\lambda \in \partitionsn{n}$ \emph{strict} if all its nonzero parts are distinct. Let $\strictpartitions[n]$ be the set of all strict partitions of $n$ and set $\strictpartitions{} := \bigcup_{n \geq 0} \strictpartitions[n]$. To a strict partition $\lambda$ we can associate its \emph{shifted Young diagram} $\shifted{\lambda}$ which is obtained from the Young diagram (using English notation) by shifting all rows so that the $i$th row is shifted rightward by $(i-1)$ cells. 

\begin{example}
Let $\lambda = (6,5,2,1) \in \strictpartitions[14]$, then 

\vspace{4mm}

\begin{center}
\begin{tikzpicture}

\node at (-1,1.25) {$\lambda \quad = $};

\draw (2.5,1.5) rectangle (3,2);
\draw (2,1.5) rectangle (2.5,2);
\draw (1.5,1.5) rectangle (2,2);
\draw (1,1.5) rectangle (1.5,2);
\draw (.5,1.5) rectangle (1,2);
\draw (0,1.5) rectangle (.5,2);
\draw (2,1) rectangle (2.5,1.5);
\draw (1.5,1) rectangle (2,1.5);
\draw (1,1) rectangle (1.5,1.5);
\draw (.5,1) rectangle (1,1.5);
\draw (0,1) rectangle (.5,1.5);
\draw (.5,.5) rectangle (1,1);
\draw (0,.5) rectangle (.5,1);
\draw (0,0) rectangle (.5,.5);

\node at (3.1,.2) {$,$};

\node at (6,1.25) {$S(\lambda) \quad = $};

\draw (9.5,1.5) rectangle (10,2);
\draw (9,1.5) rectangle (9.5,2);
\draw (8.5,1.5) rectangle (9,2);
\draw (8,1.5) rectangle (8.5,2);
\draw (7.5,1.5) rectangle (8,2);
\draw (7,1.5) rectangle (7.5,2);
\draw (9.5,1) rectangle (10,1.5);
\draw (9,1) rectangle (9.5,1.5);
\draw (8.5,1) rectangle (9,1.5);
\draw (8,1) rectangle (8.5,1.5);
\draw (7.5,1) rectangle (8,1.5);
\draw (8.5,.5) rectangle (9,1);
\draw (8,.5) rectangle (8.5,1);
\draw (8.5,0) rectangle (9,.5);

\node at (10.1,.2) {$.$};

\end{tikzpicture}
\end{center}

\vspace{4mm}

\end{example}
Henceforth we reserve the variables $\lambda$ and $\nu$ for strict partitions and the variables $\mu$ and $\gamma$ for odd partitions.

For $\nu, \lambda \in \strictpartitions$, we write $\nu \nearrow \lambda$ (respectively $\nu \searrow \lambda$) when we can obtain $\lambda$ from $\nu$ by adding (resp. removing)  a single cell $\ydcell$. Set 
\begin{equation*}
\edgemulti{\nu}{\lambda} := \begin{cases}
2 & \text{if $\nu \nearrow \lambda$, $\length{\lambda} = \length{\nu}$,} \\
1 & \text{if $\nu \nearrow \lambda$ and $\length{\lambda} = \length{\nu} + 1$,}\\
0 & \text{otherwise.}
\end{cases}
\end{equation*}

\begin{definition}
The \emph{Schur graph} $\Schurgraph$ is the graded graph such that:

\begin{itemize}
\item the vertex set of $\Schurgraph$ corresponds to $\strictpartitions{}$ and the $n$th graded component is $\strictpartitions[n]$,
\item the number of edges from $\nu$ to $\lambda$ is given by $\edgemulti{\nu}{\lambda}$.
\end{itemize}

\end{definition}

The version of $\Schurgraph$ that we consider here is the same as that studied in \cite{Pet09}. Another version of the Schur graph without edge multiplicity was investigated in \cite{Bor97}. Both graphs have the same down transition functions (see \eqref{eqn-downtransition-def} below) so in principle we could have chosen to use either.

A {\emph{standard shifted Young tableau}} of shape $\lambda \in \strictpartitions[n]$ is a bijective labeling of the cells of $\shifted{\lambda}$ by the integers $\{1, \dots, n\}$ such that entries increase from left to right across rows and down columns. Let $\stanshiftedstrict{\lambda}$ be the number of standard shifted Young tableaux of shape $\lambda$. We can compute $\stanshiftedstrict{\lambda}$ explicitly as
\begin{equation*}
\stanshiftedstrict{\lambda} = \frac{n!}{\lambda_1!\lambda_2! \dots \lambda_r!} \prod_{1 \leq i < j \leq \length{\lambda}} \frac{\lambda_i - \lambda_j}{\lambda_i + \lambda_j}.
\end{equation*}
Following \cite{Pet09} we denote the number of paths from $\emptyset$ to $\lambda$ in $\Schurgraph$ by $\pathsfromzero{\lambda}$. Then 
\begin{equation*}
\pathsfromzero{\lambda} = 2^{|\lambda| - \length{\lambda}} \stanshiftedstrict{\lambda}.
\end{equation*}

In \cite{BO09} Borodin and Olshanski used coherent families of measures on partitions to construct infinite-dimensional diffusion processes. Petrov studied analogous processes on the Schur graph \cite{Pet09}. We review some basic definitions related to the latter of these works below. 

The {\emph{down transition function}} $p^{\downarrow}: \Schurgraph \times \Schurgraph \rightarrow \MB{Q}$  on $\Schurgraph$ is defined so that for $ \nu,\lambda \in \strictpartitions{}$, 
\begin{equation} \label{eqn-downtransition-def}
\downtransition{\lambda}{\nu} := \frac{\pathsfromzero{\nu}}{\pathsfromzero{\lambda}}\edgemulti{\nu}{\lambda}.
\end{equation}
In particular, when the first argument of $p^{\downarrow}$ is restricted to $\strictpartitions[n]$ and the second to $\strictpartitions[n-1]$ the function $p^{\downarrow}$ gives a Markov transition kernel from $\Schurgraph_{n}$ to $\Schurgraph_{n-1}$. 

A {\emph{coherent system}} on $\Schurgraph$ with respect to down transition function $p^{\downarrow}$ is a collection of probability measures $\{M_{n}\}_{n \geq 0}$, with $M_n$ a probability measure on $\Schurgraph_n$, such that if $\nu \in \strictpartitions[n-1]$, then
\begin{equation*}
M_{n-1}(\nu) = \sum_{\lambda \searrow \nu} \downtransition{\lambda}{\nu} M_{n}(\lambda).
\end{equation*}
By abuse of notation we write $M_n(\lambda)$ for $M_n(\{\lambda\})$. One choice of coherent system with respect to the $p^{\downarrow}$ defined by \eqref{eqn-downtransition-def} is the collection of {\emph{Plancherel measures}} $\{\Plancherel{n}\}_{n \geq 0}$ where for $\lambda \in \strictpartitions[n]$
\begin{equation*}
\Plancherel{n}(\lambda) := \frac{2^{\length{\lambda}-n}\pathsfromzero{\lambda}^2}{n!}.
\end{equation*}

Given the coherent system $\{\Plancherel{n}\}_{n \geq 0}$ and the down transition function $p^{\downarrow}$, the corresponding {\emph{up transition function}} $p^\uparrow: \Schurgraph \times \Schurgraph \rightarrow \MB{Q}$ is defined as
\begin{equation*}
p^{\uparrow}(\nu,\lambda) := \frac{\Plancherel{n+1}(\lambda)}{\Plancherel{n}(\nu)}\downtransition{\lambda}{\nu} = \frac{\pathsfromzero{\lambda}}{\pathsfromzero{\nu}(|\nu| + 1)}
\end{equation*}
(in \cite{Pet09} this is denoted by $p^{\uparrow}_{\infty}$).

In the next section we will make a connection between induction and restriction of simple Sergeev supermodules and $p^{\uparrow}$, $p^\downarrow$.

  
  \subsection{The Sergeev algebra and the twisted hyperoctahedral group} \label{sect-Sergeev-algebra}
  
Let $\Sy{n}$ be the symmetric group on $n$ elements with $s_1, s_2, \dots, s_{n-1}$ the Coxeter generators of $\Sy{n}$. We will work with the 
Clifford algebra $\Clif$ which is the unital associative algebra with $n$ generators: 
\begin{equation} \label{eqn-clif-def}
    \Clif:=\mathbb{C}\langle  c_1,\dots,c_n \;| \; c_i^2=-1, c_ic_j=-c_jc_i \text{ for } i\neq j \rangle.
\end{equation}

\begin{remark}
There is another common presentation of $\Clif$ in which $\Clif$ is generated by $c_1', \dots, c_n'$ subject to the relations $c_i'^2 = 1$ and $c_i'c_j' = -c_j'c_i'$ for $i \neq j$ (see \cite{WW11} and \cite{K05} for example). An equivalence between this presentation and \eqref{eqn-clif-def} can be obtained by setting $c_i = \sqrt{-1}c_i'$. 
\end{remark}

\begin{definition}
The {\emph{finite Sergeev algebra}} (also known as the {\emph{finite Hecke-Clifford algebra of type A}}) is
\begin{equation}
 \Ser\simeq \Clif \rtimes \mathbb{C}[S_n]
\end{equation}
where the action of $S_n$ on the Clifford generators is by permuting indices, i.e. 
\begin{equation*}
s_ic_i = c_{i+1}s_i, \quad s_ic_{i+1} = c_is_i, \quad \text{and} \quad s_ic_j = c_js_i  \quad \text{for $j \neq i,\;i+1$}.
\end{equation*}
\end{definition}
 
There is a $\MB{Z}/2\MB{Z}$-grading making $\Sern{n}$ into a superalgebra in which the $\Sy{n}$ generators $\{s_i, 1\leq i \leq n-1\}$ are even and the $\Clif$ generators $\{c_j, 1\leq j \leq n\}$ are odd. For homogeneous element $x \in \Sern{n}$ we write $|x|$ for the degree of $x$. The Sergeev algebras form a tower of superalgebras via the embedding $\Sern{n-1} \hookrightarrow \Sern{n}$ which sends $s_i \mapsto s_i$ for $1 \leq i \leq n-2$ and $c_i \mapsto c_i$ for $1\leq i \leq n-1$. We call this the {\emph{standard embedding}} and use it implicitly throughout this paper. We set $\Sern{0} = \MB{C}$.

It will be convenient to realize $\Ser$ as the quotient of a group algebra. Let $C_2$ denote the cyclic group of order two. Define the group
\begin{equation*}
\cliffordgroup{n} :=\langle z, a_1,\dots,a_n \;| \;a_i^2=z, a_ia_j=za_ja_i, za_i = a_iz, z^2 = 1\rangle.
\end{equation*}
The group $\cliffordgroup{n}$ is a double cover of $C_2^n$ via the short exact sequence:
\begin{equation} \label{eqn-short-exact-clifford}
1 \longrightarrow C_2 \longrightarrow \cliffordgroup{n} \longrightarrow C_2^n \longrightarrow 1
\end{equation}
in which $C_2$ is mapped to the subgroup $\{1,z\} \subset \cliffordgroup{n}$ and $z \in \cliffordgroup{n}$ is mapped to $1$.

\begin{definition}
The {\emph{twisted hyperoctahedral group}} is defined as $\widehat{B_n} := \cliffordgroup{n} \rtimes  \Sy{n}$ where $\Sy{n}$ acts on $\{a_i\}$ by permuting their indices, and acts trivially on $z$. 
\end{definition}

The algebra $\MB{C}[\widehat{B_n}]$ is also a superalgebra via the $\MB{Z}/2\MB{Z}$ grading which sets $\deg(a_j) = 1$ for $1\leq j\leq n$ and $\deg(z) = \deg(s_i) = 0$ for $1 \leq i \leq n-1$. Using \eqref{eqn-short-exact-clifford} one can show that $\hyperoctahedral{n}$ is a double cover of the hyperoctahedral group $\typeBWeyl{n} = C_2^n\rtimes  \Sy{n}$ (i.e. the type $B$ Weyl group) via the short exact sequence:

\begin{equation*}
1 \longrightarrow C_2 \longrightarrow \hyperoctahedral{n} \overset{f}{\longrightarrow} \typeBWeyl{n} \longrightarrow 1
\end{equation*}
where $f$ sends $z$ to $1$. On the other hand from a comparison of generators and relations it is clear that
\begin{equation}
\Ser\simeq \mathbb{C}[\widehat{B_n}]/\langle z +1\rangle.
\end{equation}
We denote the corresponding projection by $\hyperSergeevProj{n}: \MB{C}[\widehat{B_n}] \rightarrow \Sern{n}$.

Since $z$ is central and $z^2 = 1$, for any $\MB{C}[\hyperoctahedral{n}]$-supermodule $L$, we have that $z$ must act by multiplication by either $1$ or $-1$. Hence studying $\Sern{n}$-supermodules is equivalent to studying $\MB{C}[\hyperoctahedral{n}]$-supermodules in which $z$ acts as multiplication by $-1$ (these are commonly referred to as spin representations of $\hyperoctahedral{n}$). Furthermore, via the super Wedderburn Theorem it follows that 
\begin{equation} \label{eqn-Sergeev-hyperocto-connection}
\MB{C}[\hyperoctahedral{n}] \;\; \cong \;\; \MB{C}[\hyperoctahedral{n}] /\langle z-1 \rangle \oplus \MB{C}[{\hyperoctahedral{n}}]/\langle z+1 \rangle \;\; \cong \;\;\MB{C}[\typeBWeyl{n}] \oplus \Sern{n}.
\end{equation}

The group algebras $\MB{C}[{\hyperoctahedral{n}}]$ also form a tower of algebras with the embedding $\MB{C}[{\hyperoctahedral{n-1}}] \hookrightarrow \MB{C}[{\hyperoctahedral{n}}]$ which sends $s_i \mapsto s_i$, $a_i \mapsto a_i$, and $z \mapsto z$. Note that this maps the subalgebra $\MB{C}[\Sy{n-1}]$ into the subalgebra $\MB{C}[\Sy{n}]$ in the usual way, and projected down to $\Sern{n-1}$ and $\Sern{n}$ this becomes the standard embedding described above. We set $\mathbb{C}[\hyperoctahedral{0}]$ to be the subalgebra generated by $z$.

\begin{lemma} \label{lemma-cosets}
For $n \geq 2$, 
\begin{equation} \label{eqn-first-left-coset}
\{ \; s_i \dots s_{n-1}a_{n}^\epsilon \; | \: 1 \leq i \leq n, \epsilon \in \{0,1\}\}
\end{equation}
is a collection of left coset representatives of $\hyperoctahedral{n-1}$ in $\hyperoctahedral{n}$.
\end{lemma}

Note that we follow the convention that the elements corresponding to $i = n$ are $a_n^\epsilon$ for $\epsilon \in \{0,1\}$ in Lemma \ref{lemma-cosets}.

\begin{proof}
The set $\{\; s_i \dots s_{n-1} \; | \; 1 \leq i \leq n \;\}$ forms a collection of minimal length left coset representatives of $\Sy{n-1}$ in $\Sy{n}$. It follows from this and the fact that $\hyperoctahedral{n} := \cliffordgroup{n} \rtimes \Sy{n}$ that any element $g \in \hyperoctahedral{n}$ can be written as $g = s_i \dots s_{n-1} \omega a_{n}^\epsilon a_J z^\beta$ where $1 \leq i \leq n$, $\omega \in \Sy{n-1}$, $a_J = a_{j_1}\dots a_{j_t}$ for some $J = \{j_1, \dots, j_t\} \subseteq \{1,2, \dots, n-1\}$, and $\epsilon, \beta \in \{0,1\}$. Since $a_{n}$ commutes with $\Sy{n-1}$ we have $x = s_i \dots s_{n-1} a_n^{\epsilon} \omega  a_J z^\beta$. Since $\omega  a_J z^\beta \in \hyperoctahedral{n-1}$, the set \eqref{eqn-first-left-coset} contains a set of left coset representatives. The result then follows from the observation that the size of \eqref{eqn-first-left-coset} is $2n$ while $|\hyperoctahedral{n}| = 2^{n+1}n!$ and $|\hyperoctahedral{n-1}| = 2^n(n-1)!$.
\end{proof}

\begin{remark}
If $g = s_i \dots s_{n-1}a_n^\epsilon$ for $\epsilon \in \{0,1\}$ then $g^{-1} = a_n^\epsilon z^{\epsilon} s_{n-1} \dots s_i$ and consequently while $\hyperSergeevProj{n}(g) = s_i \dots s_{n-1}c_n^\epsilon$, we have
\begin{equation*}
\hyperSergeevProj{n}(g^{-1}) = (-1)^{\epsilon} c_n^\epsilon s_{n-1} \dots s_i = (-1)^{|g|} c_n^\epsilon s_{n-1} \dots s_i.
\end{equation*}
\end{remark}

We use the inclusions $ \hyperoctahedral{1} \subset \hyperoctahedral{2} \subset \dots \subset \hyperoctahedral{n-1} \subset \hyperoctahedral{n} \subset \dots$ to iterate Lemma \ref{lemma-cosets} to get that for all $1 \leq k < n$, 
\begin{equation*}
\LcosHyper{n}{k} := \{ \: (s_{i_{n}} \dots s_{n-1} a_{n}^{\epsilon_{n}})( s_{i_{n-1}} \dots s_{n-2} a_{n-1}^{\epsilon_{n-1}}) \dots (s_{i_{k+1}} \dots s_k a_{k+1}^{\epsilon_{k+1}})\} \; 
\end{equation*}
\begin{equation*}
| \; 1 \leq i_j \leq j, \; \epsilon_j \in \{0,1\} \; \}
\end{equation*}
is a collection of left coset representatives of $\hyperoctahedral{k}$ in $\hyperoctahedral{n}$. Note in particular that
\begin{equation} \label{eqn-size-left-cosets}
|\LcosHyper{n}{k}| = \ds \frac{|\hyperoctahedral{n}|}{|\hyperoctahedral{k}|}=n^{\downarrow k}2^{n-k}
\end{equation}
where $n^{\downarrow k}$ is the falling factorial
\begin{equation*}
n^{\downarrow k} :=\frac{n!}{(n-k)!} = n(n-1)\dots(n-k+1)
\end{equation*}
for $1 \leq k < n$. The projection $\hyperSergeevProj{n}: \MB{C}[\hyperoctahedral{n}] \rightarrow \Sern{n}$ sends the elements of $\LcosHyper{n}{k}$ to distinct non-zero elements of $\Sern{n}$ and we set
\begin{equation*}
\LcosSer{n}{k} := \hyperSergeevProj{n}(\LcosHyper{n}{k}).
\end{equation*}

The set of conjugacy classes of $\hyperoctahedral{n}$ is indexed by pairs of partitions $(\rho_+,\rho_-)$ such that $|\rho_+| + |\rho_-| = n$, plus an additional parameter $\epsilon \in \{0,1\}$ when either $(\rho_+,\rho_-) = (\mu, \emptyset)$ with $\mu \in \oddpartitions[n]$ or $(\rho_+,\rho_-) = (\emptyset,\lambda)$ with $\lambda \in \strictpartitions[n]$. We denote this indexing set by $\conj$. A detailed description of the conjugacy class structure of $\hyperoctahedral{n}$ can be obtained by analyzing the conjugacy class structure of $\typeBWeyl{n}$ (which follows from the basic theory for the conjugacy class structure of wreath products \cite[Appendix B]{Mac15}) and investigating how the inverse image of these sets under the map $\gamma: \hyperoctahedral{n} \twoheadrightarrow \typeBWeyl{n}$ split into new conjugacy classes \cite{Re76}. The additional parameter $\epsilon \in \{0,1\}$ appears precisely when a conjugacy class in $\typeBWeyl{n}$ splits into two conjugacy classes in $\hyperoctahedral{n}$. Since we are ultimately interested in the center of $\Sern{n}$ which can be described using information about the conjugacy classes indexed by $(\mu, \emptyset, \epsilon)$ for $\mu \in \oddpartitions[n]$ and $\epsilon \in \{0,1\}$, we limit ourselves to considering these classes. We call this set of conjugacy classes $\conjOdd \subset \conj$. For $\beta \in \conj$, we write $\conj(\beta)$ for the corresponding conjugacy class.

We introduce a family of elements of $\Sern{n}$ which will be useful for constructing representatives for the conjugacy classes from $\conjOdd$. For  $\mu = (\mu_1,\dots,\mu_r) \in \oddpartitions[k]$, set $\pi_\mu = 1$ if $\mu = (1^k)$ and otherwise
\begin{equation*}
\pi_\mu := (s_{k-1}\dots s_{k-\mu_r+1})\dots(s_{\mu_1+\mu_2-1} \dots s_{\mu_1 + 1}) (s_{\mu_1 -1} \dots s_2s_1)
\end{equation*}
\begin{equation} \label{eqn-perm-decomp}
= (k,k-1,\dots,k-\mu_r+1)\dots(\mu_1+\mu_2,\dots,\mu_1+1) (\mu_1,\dots,2,1) \in \Sy{k}.
\end{equation}
For $n \geq k$ we define $\distinguishedperm{\mu}{n} := \tau_0 \pi_\mu \tau_0^{-1}$, where $\tau_0$ is the longest element of $\Sy{n}$ by Coxeter length. Notice that $\distinguishedperm{\mu}{n}$ has cycle type $(\mu,1^{n-k}) \in \oddpartitions[n]$ and fixes $1,2,\dots,n-k$ pointwise. 


\begin{proposition} 
The elements $\{\; \distinguishedperm{\mu}{n},\; z\distinguishedperm{\mu}{n} \; | \; \mu \in \oddpartitions[n]\}$ form a complete set of conjugacy class representatives for the conjugacy classes $\conjOdd$ in $\hyperoctahedral{n}$ with $\distinguishedperm{\mu}{n}$ corresponding to $(\mu,\emptyset,0) \in \conjOdd$ and $z\distinguishedperm{\mu}{n}$ corresponding to $(\mu,\emptyset,1) \in \conjOdd$.
\end{proposition}

\begin{proof}
This follows from the description of the conjugacy classes of $\typeBWeyl{n}$ and results on conjugacy class splitting in $\hyperoctahedral{n}$ \cite{Re76} (see \cite[Section 2.5]{WW11} for an overview).
\end{proof}

Note that under the projection map $\hyperSergeevProj{n}: \MB{C}[\hyperoctahedral{n}] \rightarrow \Sern{n}$, the two sets of conjugacy classes $\{\distinguishedperm{\mu}{n}\}_{\mu \in \strictpartitions[n]}$ and $\{z\distinguishedperm{\mu}{n}\}_{\mu \in \strictpartitions[n]}$ are identified since $\hyperSergeevProj{n}(z)=-1$.

The size of the conjugacy classes $\conj(\mu,\emptyset,\epsilon)$ will be important to us later. For $\rho \in \partitionsn{n}$, we denote by $z_\rho$ the size of the stabilizer of an element of $\Sy{n}$ of cycle type $\rho$ under the conjugation action. Recall that 
\begin{equation*}
z_\rho = \prod_{i \in \MB{Z}_{\geq 0}} i^{\partsofpartition{i}{\rho}} \partsofpartition{i}{\rho}!
\end{equation*}
where $\partsofpartition{i}{\rho}$ is the number of parts of size $i$ in $\rho$. 

\begin{lemma} \cite{Iv01} \label{lemma-size-of-orbit}
For $\mu \in \oddpartitions[n]$, $\epsilon \in \{0,1\}$
\begin{equation*}
|\conj(\mu,\emptyset,\epsilon)| = \frac{n!}{z_{\mu}} 2^{n - \length{\mu}}.
\end{equation*}
\end{lemma}

$\Sern{n}$ has analogues to the classical Jucys-Murphy elements of $\MB{C}[S_{n}]$. These elements $\{\JM{i}\}^{n}_{i =1}$, which we also call Jucys-Murphy elements, are defined by
\begin{equation*}
\JM{1} := 0, \quad\quad \JM{k} := \sum_{j =1}^{k-1} (1+c_jc_k)(j,k).
\end{equation*}
They generate a commutative subalgebra of $\Sern{n}$ and their spectra have a combinatorial interpretation analogous to that of the classical Jucys-Murphy elements (\cite{Naz97}, \cite{VS08}, \cite{HKS11} \cite{Wan10}).


\subsection{The super representation theory of $\Sern{n}$ and $\widehat{B_n}$}

In this subsection we will review basic facts about the super representation theory of $\Sern{n}$. Recall that any $\Sern{n}$-supermodule is by definition a spin representation of $\hyperoctahedral{n}$, so all statements about $\Sern{n}$-supermodules also hold for $\hyperoctahedral{n}$ spin representations. We refer the reader to \cite{K05} and \cite{WW11} for thorough accounts of these topics as well as a review of super representation theory. 

Let $\lengthparity: \strictpartitions \rightarrow \{0,1\}$ be defined by
\begin{equation*}
\lengthparity(\lambda) := \begin{cases} 0 & \length{\lambda} \text{ is even}\\ 1 & \length{\lambda} \text{ is odd.}
\end{cases}
\end{equation*}
The function $\lengthparity$ will be useful for describing quantities related to the representation theory of $\Sern{n}$.

 \begin{theorem} \cite{Ser84} \label{thm-dim-simple}
  The set of simple $\Sern{n}$-supermodules are indexed by $\strictpartitions[n]$, and
for any $\lambda \in \strictpartitions[n]$, the simple $\Sern{n}$-supermodule indexed by $\lambda$ has dimension
\begin{equation*}
\dim(\SimpleSer{\lambda}) = 2^{n-\frac{\length{\lambda}-\lengthparity(\lambda)}{2}}\stanshiftedstrict{\lambda}.
\end{equation*}

\end{theorem}

The algebras $\{\Sern{n}\}_{n \geq 0}$ are semisimple. When $N$ and $M$ are $\Sern{n}$-supermodules we write
\begin{equation*}
[M : N] := \dim(\Hom_{\Sern{n}}(M,N)).
\end{equation*}

The theorem below describes the branching for $\{\Sern{n}\}_{n \geq 0}$.

\begin{theorem} \cite{K05} \label{thm-Ser-branching}
Let $\lambda \in \strictpartitions[n]$ and $\nu \in \strictpartitions[n-1]$, then:
\begin{enumerate}
\item \label{eqn-branching-res}
$\ds [\SimpleSer{\nu}:\res^{\Sern{n}}_{\Sern{n-1}} \SimpleSer{\lambda}] = \begin{cases}
2^{\frac{2 + \length{\nu} - \lengthparity(\nu) - \length{\lambda} + \lengthparity(\lambda)}{2}} & \lambda \searrow \nu \\
0 & \text{otherwise.}
\end{cases}$
\item \label{eqn-branching-ind}
$ \ds [\ind^{\Sern{n}}_{\Sern{n-1}} \SimpleSer{\nu}:\SimpleSer{\lambda}] = \begin{cases}
2^{\frac{2 + \length{\nu} - \lengthparity(\nu) - \length{\lambda} + \lengthparity(\lambda)}{2}} & \nu \nearrow \lambda, \\
0 & \text{otherwise.}
\end{cases} $
\end{enumerate}
\end{theorem}

The next theorem describes how $\Sern{n}$ as a left $\Sern{n}$-supermodule decomposes into a direct sum of simple $\Sern{n}$-supermodules.

\begin{proposition} \label{prop-reg-rep-decomp}
$\Sern{n}$ as a left $\Sern{n}$-supermodule decomposes as
\begin{equation*}
\Sern{n} \cong \bigoplus_{\lambda \in \strictpartitions[n]} (\SimpleSer{\lambda})^{\oplus \dim(\SimpleSer{\lambda})/2^{\lengthparity(\lambda)}}.
\end{equation*}
\end{proposition}

\begin{proof}
For $\lambda \in \strictpartitions[n]$, the multiplicity of $\SimpleSer{\lambda}$ in $\Sern{n}$ is equal to the multiplicity of the corresponding spin representation $\wh{\SimpleSer{\lambda}}$ of $\hyperoctahedral{n}$ in the left regular super representation. If $\wh{\SimpleSer{\lambda}}$ is of type $\mathsf{M}$ (i.e. $\lengthparity(\lambda) = 0$), then $\wh{\SimpleSer{\lambda}}$ is simple as an ungraded $\hyperoctahedral{n}$-module \cite[Section 12.2]{K05}.
When $\wh{\SimpleSer{\lambda}}$ is of type $\mathsf{Q}$ (i.e. $\lengthparity(\lambda) = 1$) then considered as an ungraded $\MB{C}[\hyperoctahedral{n}]$-module, $\wh{\SimpleSer{\lambda}}$ splits into a direct sum of two simple $\MB{C}[\hyperoctahedral{n}]$-modules, each with dimension $\dim(\wh{\SimpleSer{\lambda}})/2$. The result then follows from the ungraded representation theory of $\hyperoctahedral{n}$. 
\end{proof}

We can now relate $\downtransition{\cdot}{\cdot}$ and $\uptransition{\cdot}{\cdot}$ to the representation theory of the algebras $\{\Sern{n}\}_{n \geq 0}$.

\begin{proposition} \label{lemma-transition-prob-rep}
Let $\lambda \in \strictpartitions[n]$ and $\nu \in \strictpartitions[n-1]$. Then
\begin{enumerate}
\item \label{enum-lemma-cotransition-rep}
$ \ds \downtransition{\lambda}{\nu} = \frac{[\SimpleSer{\nu}:\res^{\Sern{n}}_{\Sern{n-1}} \SimpleSer{\lambda}]\dim(\SimpleSer{\nu})}{\dim(\SimpleSer{\lambda})},$
\item \label{enum-lemma-transition-rep}
$ \ds 2^{\lengthparity(\lambda) - \lengthparity(\nu)}\uptransition{\nu}{\lambda} = \frac{[\ind^{\Sern{n}}_{\Sern{n-1}} \SimpleSer{\nu}:\SimpleSer{\lambda}]\dim(\SimpleSer{\lambda})}{\dim(\ind^{\Sern{n}}_{\Sern{n-1}} \SimpleSer{\nu})}.
$
\end{enumerate}
\end{proposition}

\begin{proof}
Both \ref{enum-lemma-cotransition-rep} and \ref{enum-lemma-transition-rep} follow from the branching rules for Sergeev algebras in Theorem \ref{thm-Ser-branching} and the dimension formula in Theorem \ref{thm-dim-simple}.
\end{proof}

For $\lambda \in \strictpartitions[n]$, we denote by $\HyperChar{\lambda}$ the character corresponding to simple $\MB{C}[\hyperoctahedral{n}]$-supermodule $\SimpleHyper{\lambda}$. This descends to a character $\SerChar{\lambda}$ for simple $\Sern{n}$-supermodule $\SimpleSer{\lambda}$ with
\begin{equation*}
\SerChar{\lambda}(\hyperSergeevProj{n}(g)) = \HyperChar{\lambda}(g).
\end{equation*}
The {\emph{normalized character}} $\normSerChar{\lambda}$ is defined such that for $x \in \Sern{n}$
\begin{equation*}
\normSerChar{\lambda}(x) := \frac{\SerChar{\lambda}(x)}{\dim(\SimpleSer{\lambda})} = \frac{\SerChar{\lambda}(x)}{\SerChar{\lambda}(1)}.
\end{equation*}

\begin{proposition} \label{prop-char-determined-by-odd-partitions}
For $\lambda \in \strictpartitions[n]$, the character $\SerChar{\lambda}$ is uniquely determined by its value on the elements $\{\distinguishedperm{\mu}{n} \; | \; \mu \in \oddpartitions[n]\}$.
\end{proposition}

\begin{proof}
This follows from a similar statement \cite[Proposition 1.9]{Iv01} where each element $\distinguishedperm{\mu}{n}$ is replaced by an element of $\Sern{n}$ that is conjugate to it. Since characters are constant across conjugacy classes, the result follows.
\end{proof}

Given Proposition \ref{prop-char-determined-by-odd-partitions}, for $\mu \in \oddpartitions[k]$ with $k \leq n$, it is convenient to write $\SerChar{\lambda}(\mu \cup 1^{n-k}) := \SerChar{\lambda}(\distinguishedperm{\mu}{n})$.
 

\subsection{The centers of $\Sern{n}$ and $\MB{C}[\hyperoctahedral{n}]$}

As a superalgebra the center of $\Sern{n}$ breaks up into even and odd components of super-commutative elements. In this paper we will focus on $\evencenter{n}$, which corresponds to the ungraded center of $\Sern{n}$. Note that $\evencenter{n}$ is exactly those elements that act on all simple $\Sern{n}$-modules as multiplication by a scalar. Following \cite{Iv01} we will construct a basis for $\evencenter{n}$ via the surjection $\pi:\MB{C}[{\hyperoctahedral{n}}] \twoheadrightarrow \Sern{n}$.

Recall that the set $\conj$ indexes the conjugacy classes of $\hyperoctahedral{n}$. For $\beta \in \conj$ set
\begin{equation*}
\wh{C}_{\beta} := \sum_{g \in \conj(\beta)} g.
\end{equation*}
It is clear that $\{\wh{C}_{\beta}\}_{\beta \in \conj}$ is a basis for the ungraded center of $\MB{C}[\hyperoctahedral{n}]$. In \cite{Iv01}, Ivanov uses the subset of this basis corresponding to elements of $\conjOdd$ of the form $(\mu,\emptyset,0)$ to construct a basis for $\evencenter{n}$. For $\mu \in \oddpartitions[n]$ let
\begin{equation*}
C_\mu := \hyperSergeevProj{n}(\wh{C}_{(\mu,\emptyset,0)}).
\end{equation*}

\begin{proposition} \cite{Iv01} \label{prop-center-basis}
The set $\{C_\mu \; | \; \mu \in \oddpartitions[n]\}$ is a linear basis for $\evencenter{n}$.
\end{proposition}

We now define a scaled version of Ivanov's basis of $\evencenter{n}$ which naturally appears from the center of the twisted Heisenberg category.

\begin{definition} \label{def-class-sums}
For $k \leq n$ and  $\mu \in \oddpartitions[k]$, define
\begin{equation*}
\octoclasssum{\mu}{n} := \sum_{g \in \LcosHyper{n}{n-k}} g\distinguishedperm{\mu}{n} g^{-1}
\end{equation*}
and 
\begin{equation*}
\Sergeevclasssum{\mu}{n} := \pi(\octoclasssum{\mu}{n}). 
\end{equation*}
\end{definition}

\begin{proposition} \label{prop-conj-classes-go-to-zero}
Let $k\leq n$ and $\mu \in \oddpartitions[k]$ then:
\begin{enumerate} 
\item $\octoclasssum{\mu}{n} \in Z(\MB{C}[\hyperoctahedral{n}])$ and $\Sergeevclasssum{\mu}{n} \in \evencenter{n}$.
\item \label{item-scaled-basis} 
$\ds \octoclasssum{\mu}{n} = 2^{k-n+\length{\mu}}\frac{z_{\mu \cup 1^{n-k}}}{(n-k)!} \wh{C}_{(\mu \cup 1^{n-k},\emptyset,0)}$.

\item \label{prop-conj-classes-go-to-zero-2} Let $h \in \hyperoctahedral{n}$ be an element not belonging to the same conjugacy class as $\distinguishedperm{\mu}{n}$ or $z\distinguishedperm{\mu}{n}$ for some $\mu \in \oddpartitions[n]$ (i.e. $h$ does not belong to a conjugacy class indexed by $(\mu,\emptyset,\epsilon)$, $\epsilon \in \{0,1\}$). Then
\begin{equation*}
\hyperSergeevProj{n}\Big(\sum_{g \in\hyperoctahedral{n}} g h g^{-1}\Big) = 0.
\end{equation*}
\end{enumerate}
\end{proposition}

\begin{proof}
\begin{enumerate}
\item Recall that we defined $\distinguishedperm{\mu}{n}$ as a distinguished element from the conjugacy class of $\hyperoctahedral{n}$ indexed by $(\mu\cup1^{n-k}, \emptyset, 0)$. Since $\distinguishedperm{\mu}{n}$ is by definition a product of $s_{n-1},\dots,s_{n-k+1}$ it commutes with $\hyperoctahedral{n-k}$. Since $\LcosHyper{n}{n-k}$ is a collection of left coset representatives of $\hyperoctahedral{n-k}$ in $\hyperoctahedral{n}$ any element $g \in \hyperoctahedral{n}$ can be written uniquely as $g = \sigma h$ for $\sigma \in \LcosHyper{n}{n-k}$ and $h \in \hyperoctahedral{n-k}$. Thus $g\distinguishedperm{\mu}{n}g^{-1} = \sigma h \distinguishedperm{\mu}{n} h^{-1}\sigma^{-1} =  \sigma \distinguishedperm{\mu}{n} \sigma^{-1}$ and hence $g\distinguishedperm{\mu}{n}g^{-1}$ is completely determined by the left coset to which $g$ belongs. It follows that 
\begin{equation} \label{eqn-octo-class-sum-rln}
\sum_{g \in \hyperoctahedral{n}} g \distinguishedperm{\tilde{\mu}}{n} g^{-1} = |\hyperoctahedral{n-k}|\sum_{g \in \LcosHyper{n}{n-k}} g\distinguishedperm{\mu}{n} g^{-1} = |\hyperoctahedral{n-k}|\displaystyle \octoclasssum{\mu}{n}
\end{equation}
and $\octoclasssum{\mu}{n} \in Z(\MB{C}[\hyperoctahedral{n}])$ since $\octoclasssum{\mu}{n}$ is a scalar multiple of a central element. 
  
 Finally, note that $\hyperSergeevProj{n}$ is a degree-preserving homomorphism and $\octoclasssum{\mu}{n}$ is even, so $\pi(\octoclasssum{\mu}{n})=\Sergeevclasssum{\mu}{n} \in \evencenter{n}$.

\item It follows from Lemma \ref{lemma-size-of-orbit} and the orbit stabilizer theorem that
\begin{equation*}
\sum_{g \in \hyperoctahedral{n}} g\distinguishedperm{\mu}{n} g^{-1} = 2^{\length{\mu} + 1}z_{\mu \cup 1^{n-k}}\wh{C}_{(\mu\cup 1^{n-k},\emptyset,0)}.
\end{equation*}
Then \eqref{eqn-octo-class-sum-rln} implies that
\begin{equation*}
|\hyperoctahedral{n-k}|\octoclasssum{\mu}{n} = 2^{\length{\mu} + 1}z_{\mu \cup 1^{n-k}}\wh{C}_{(\mu\cup 1^{n-k},\emptyset,0)}.
\end{equation*}
The result follows.

\item It follows from the proofs of  \cite[Proposition 3.4]{R17} and \cite[Proposition 3.9]{R17} that the image in $\Sern{n}$, $\hyperSergeevProj{n}(h)$, of any such element is conjugate to its negative. Hence each term in the image of the conjugacy class sum of any such $h$ appears in a pair with its negative. The result follows. \qedhere
\end{enumerate}
\end{proof}

It follows from Proposition \ref{prop-conj-classes-go-to-zero}.\ref{item-scaled-basis} that $\{\Sergeevclasssum{\mu}{n} \; | \; \mu \in \oddpartitions[n]\}$ is also a linear basis of $\evencenter{n}$. 

For a spin representation  $\SimpleHyper{\lambda}$ of $\hyperoctahedral{n}$, the corresponding character $\HyperChar{\lambda}$ is a homomorphism when restricted to $Z(\MB{C}[\hyperoctahedral{n}])_0$ and $\SerChar{\lambda}$ and $\normSerChar{\lambda}$ are homomorphisms on $\evencenter{n}$. 

\begin{proposition} \label{prop-character-classsum}
Let $\lambda \in \strictpartitions[n]$ and $\mu \in \oddpartitions[k]$. Then 
\begin{equation*}
\normSerChar{\lambda}(\Sergeevclasssum{\mu}{n}) = 2^kn^{\downarrow k}\frac{\SerChar{\lambda}(\distinguishedperm{\mu}{n})}{\SerChar{\lambda}(1)}.
\end{equation*}
\end{proposition}

\begin{proof}
This follows from the fact that characters are invariant under conjugation and \eqref{eqn-size-left-cosets}.
\end{proof}

Another basis for $\evencenter{n}$ is given by the set of central idempotents of $\Sern{n}$ corresponding to the simple $\Sern{n}$-supermodules. We denote these central idempotents by $\{ \centralidem{\lambda} \; | \; \lambda \in \strictpartitions[n]\}$. 

\begin{lemma} \label{lemma-central-idempotent}
For $\lambda \in \strictpartitions[n]$, the central idempotent $\centralidem{\lambda} \in \Sern{n}$ corresponding to the simple $\Sern{n}$-supermodule $\SimpleSer{\lambda}$ can be written as
\begin{equation*}
\centralidem{\lambda} = 2^{\frac{-\length{\lambda}-\lengthparity(\lambda)}{2}}\frac{\stanshiftedstrict{\lambda}}{n!} \sum_{\mu \in \oddpartitions[n]} \SerChar{\lambda}(\mu)C_\mu.
\end{equation*}
\end{lemma}

\begin{proof}
The definition of $\Sern{n}$ implies that $\centralidem{\lambda}$ is the image of the corresponding central idempotent $\wh{\centralidem{\lambda}}$ in $\hyperoctahedral{n}$ under the projection map $\hyperSergeevProj{n}$. If $\SimpleSer{\lambda}$ is of type $\mathsf{M}$, then $\SimpleHyper{\lambda}$ is simple as an ungraded $\MB{C}[\hyperoctahedral{n}]$-module; if $\SimpleSer{\lambda}$ is of type $\mathsf{Q}$, then $\SimpleHyper{\lambda}$ viewed as an ungraded $\MB{C}[\hyperoctahedral{n}]$-module breaks into the direct sum of two simples of equal dimension, $\SimpleHyper{\lambda}  = \wh{L}^{\lambda_0} \oplus \wh{L}^{\lambda_{1}}$. Thus the central idempotent corresponding to $\SimpleHyper{\lambda}$ is given by
\begin{equation*} \label{eqn-general-central-idempotent}
\wh{\centralidem{\lambda}} = \frac{\dim(\SimpleHyper{\lambda})}{2^{\delta(\lambda)}|\hyperoctahedral{n}|}\sum_{\beta \in \conj} \HyperChar{\lambda}(g)\wh{C}_\beta,
\end{equation*}
where, if $\SimpleSer{\lambda}$ is of type $\mathsf{Q}$, then $\HyperChar{\lambda}(g) = \HyperChar{\lambda_0}(g) + \HyperChar{\lambda_1}(g)$.

Applying $\hyperSergeevProj{n}$ to $\wh{\centralidem{\lambda}}$, Proposition \ref{prop-conj-classes-go-to-zero}.\ref{prop-conj-classes-go-to-zero-2} implies that most terms go to zero and we are left with
\begin{equation*}
\centralidem{\lambda} =  2^{\frac{- \length{\lambda}-\lengthparity(\lambda)}{2}-1}\frac{\stanshiftedstrict{\lambda}}{n!} \sum_{\beta \in \conjOdd} \SerChar{\lambda}(\beta)\hyperSergeevProj{n}(\wh{C}_\beta).
\end{equation*}
Recall that $\beta \in \conjOdd$ contains pairs $\beta = (\mu,\emptyset,0)$ and $\bar{\beta} = (\mu, \emptyset,1)$ for $\mu \in \oddpartitions[n]$ such that if $x \in \conj(\beta)$ then $zx \in \conj(\bar{\beta})$. It follows that $\hyperSergeevProj{n}(\wh{C}_{\bar{\beta}}) = -\hyperSergeevProj{n}(\wh{C}_\beta)$. At the same time, since $z$ acts as multiplication by $-1$ on $\SimpleHyper{\lambda}$ then $\SerChar{\lambda}(z\distinguishedperm{\mu}{n}) = -\SerChar{\lambda}(\distinguishedperm{\mu}{n})$. It follows that 
\begin{equation*}
2^{\frac{- \length{\lambda}-\lengthparity(\lambda)}{2}-1}\frac{\stanshiftedstrict{\lambda}}{n!} \sum_{\beta \in \conjOdd} \SerChar{\lambda}(\beta)\hyperSergeevProj{n}(\wh{C}_\beta) = 2^{\frac{-\length{\lambda}-\lengthparity(\lambda)}{2}}\frac{\stanshiftedstrict{\lambda}}{n!} \sum_{\mu \in \oddpartitions[n]} \SerChar{\lambda}(\mu)C_\mu.
\end{equation*}
\end{proof}


\subsection{Interlacing coordinates for strict partitions}

In \cite{Ker00} Kerov developed a useful way to parametrize Young diagrams via their interlacing coordinates. Petrov showed that shifted strict Young diagrams can be similarly parametrized with only slight modification \cite{Pet09}. For $\lambda \in \strictpartitions{}$, and $\ydcell \in \shifted{\lambda}$ with coordinates $(i,j)$, the {\emph{content}} of $\ydcell$ is defined to be 
\begin{equation*}
\content{\ydcell} := i - j.
\end{equation*}
Note that when $\ydcell$ comes from a shifted diagram, $\content{\ydcell}$ is always nonnegative. Let
\begin{enumerate}
\item $\KerovCoorUp{\lambda}$ be the set of contents for cells that we can add to $\shifted{\lambda}$ to get another shifted strict partition. 
\item $\KerovCoorDown{\lambda}$ be the set of contents for cells that we can remove from $\shifted{\lambda}$ to get another shifted strict partition. 
\end{enumerate}
The set $(\KerovCoorUp{\lambda}, \KerovCoorDown{\lambda})$ uniquely characterizes $S(\lambda)$ and is called the {\emph{Kerov coordinates}} of $S(\lambda)$. We follow \cite{Pet09} and denote the shifted diagram obtained by adding a cell $\ydcell$ with content $x \in \KerovCoorUp{\lambda}$ to $S(\lambda)$ by $S(\lambda) + \ydcell(x)$ and the shifted diagram obtained by removing a cell $\ydcell$ from $S(\lambda)$ with content $y \in \KerovCoorDown{\lambda}$ by $S(\lambda) - \ydcell(y)$.                

\begin{example}                                   
In the case of $\lambda = (6,5,2,1) \in \strictpartitions[14]$, $\KerovCoorDown{\lambda} = \{{\color{red}{0}},{\color{red}{4}}\}$, $\KerovCoorUp{\lambda} = \{{\color{blue}{2}},{\color{blue}{6}}\}$, and

\vspace{4mm}

\begin{center}
\begin{tikzpicture}

\node at (6,1.25) {$S(\lambda) \quad = $};

\draw (9.5,1.5) rectangle (10,2);
\draw (9,1.5) rectangle (9.5,2);
\draw (8.5,1.5) rectangle (9,2);
\draw (8,1.5) rectangle (8.5,2);
\draw (7.5,1.5) rectangle (8,2);
\draw (7,1.5) rectangle (7.5,2);
\draw (9.5,1) rectangle (10,1.5);
\draw (9,1) rectangle (9.5,1.5);
\draw (8.5,1) rectangle (9,1.5);
\draw (8,1) rectangle (8.5,1.5);
\draw (7.5,1) rectangle (8,1.5);
\draw (8.5,.5) rectangle (9,1);
\draw (8,.5) rectangle (8.5,1);
\draw (8.5,0) rectangle (9,.5);


\node at (8.75,.25) {\color{red}{0}};
\node at (9.75,1.25) {\color{red}{4}};


\node at (9.25,.75) {\color{blue}{2}};
\node at (10.25,1.75) {\color{blue}{6}};

\node at (10.1,.2) {$.$};

\end{tikzpicture}
\end{center}

\vspace{4mm}

\end{example}


\omitt{
As in \cite{Pet09}, we define $\KerovCoorUpNoZero{\lambda} := \KerovCoorUp{\lambda} \setminus \{0\}$. It follows from \cite[Proposition 3.2]{Pet09} that for all $\lambda \in \strictpartitions$, $|\KerovCoorUpNoZero{\lambda}| = |\KerovCoorDown{\lambda}|$.
} 

We set
\begin{equation*}
\JMeig{i} := i(i+1).
\end{equation*}


Petrov defined two families of functions $\upmoment{k}: \strictpartitions \rightarrow \MB{Q}$ and $\downmoment{k}: \strictpartitions \rightarrow \MB{Q}$, such that
\begin{equation*}
\upmoment{k}(\lambda) := \sum_{x \in \KerovCoorUp{\lambda}} \uptransition{\lambda}{\lambda + \ydcell(x)}\JMeig{x}^k 
\end{equation*}
and
\begin{equation*}
\downmoment{k+1}(\lambda) := 2|\lambda|\sum_{y \in \KerovCoorDown{\lambda}} \downtransition{\lambda}{\lambda - \ydcell(y)}\JMeig{y}^k
\end{equation*}
and investigated their properties in \cite{Pet09}. Note that $\upmoment{0} = 1$.

\begin{remark}
$\upmoment{k}(\lambda)$ and $\downmoment{k}(\lambda)$ are the strict partition analog to moments of Kerov's transition and co-transition measure \cite{Ker00}. They will play a similar role to the one they played in \cite{KLM16}.
\end{remark}

\begin{proposition} \cite[Proposition 5.4]{Pet09} \label{prop-recursive-relation-petrov}
For $\lambda \in \strictpartitions[n]$,
\begin{equation*}
\upmoment{k} = \downmoment{k} +  \sum_{\substack{i,j > 0, \\i + j = k}} \upmoment{i}\downmoment{j}.
\end{equation*}
\end{proposition}

We now give algebraic interpretations of $\upmoment{k}(\lambda)$ and $\downmoment{k}(\lambda)$ analogous to those found by Biane for Kerov's transition and co-transition measure on Young diagrams \cite{B98}. Let $\pr{n-1}: \Sern{n} \rightarrow \Sern{n-1}$ be the linear map defined such that for $x \in \Sern{n}$
\begin{equation*}
\pr{n-1}(x) := \begin{cases} 
x & \text{if $x \in \Sern{n-1}$}\\
0 & \text{otherwise}.
\end{cases}
\end{equation*}

\begin{proposition} \label{proposition-algebraic-interp-moments}
Let $\lambda \in \strictpartitions[n]$ for $n \geq 1$ and $k \geq 0$, then
\begin{enumerate}
\item 
$\norcharrep{\lambda}(\pr{n}(\JM{n+1}^{2k})) =  \upmoment{k}(\lambda).$
\item $\ds
\norcharrep{\lambda}\Big( \sum_{x \in \LcosSer{n}{n-1}} x\JM{n}^{r}x^{-1} \Big) = \begin{cases}
\downmoment{k+1}(\lambda) & \text{if $r = 2k$ is even}\\
0 & \text{otherwise}.
\end{cases}$
\end{enumerate}
\end{proposition}

\begin{proof}
\begin{enumerate}

\item Consider the character $\regrep_n: \Sern{n} \rightarrow \MB{C}$ corresponding to $\Sern{n}$ acting on itself by left multiplication. For $x \in \Sern{n}$,
\begin{equation} \label{eqn-reg-rep}
\regrep_n(x) := \begin{cases}
2^{n}n! & \text{if $x =1$}\\
0 & \text{otherwise}.
\end{cases}
\end{equation}
It follows from \eqref{eqn-reg-rep} that 
\begin{equation*}
2(n+1)\regrep_n(\pr{n}(x)) = \regrep_{n+1}(x).
\end{equation*}
Also note that if $y \in \Sern{n}$ and $x \in \Sern{n+1}$ then
\begin{equation*}
\pr{n}(yx) = y\pr{n}(x).
\end{equation*}
Recall that $e_\lambda$ is the central idempotent of $\Sern{n}$ corresponding to simple $\Sern{n}$-supermodule $\SimpleSer{\lambda}$. Then by Lemma \ref{prop-reg-rep-decomp} there are $2^{-\lengthparity(\lambda)}\dim(\SimpleSer{\lambda})$ copies of $\SimpleSer{\lambda}$ in the $\Sern{n}$-supermodule $\Sern{n}$ so that
\begin{align} \label{eqn-one-side-co-moment}
2(n+1)\regrep_n(\pr{n}(e_\lambda \JM{n+1}^{2k})) = 2(n+1)\regrep_n(e_\lambda \pr{n}(\JM{n+1}^{2k})) \\= 2^{1-\lengthparity(\lambda)}(n+1)\dim(\SimpleSer{\lambda})\charrep{\lambda}(\pr{n}(\JM{n+1}^{2k})).
\end{align}
On the other hand, the weight space decomposition for the Jucys-Murphy operators on $\Sern{n}$-supermodules implies that
\begin{equation*}
2(n+1)\regrep_n(\pr{n}(e_\lambda \JM{n+1}^{2k})) = \regrep_{n+1}(e_\lambda \JM{n}^{2k})
\end{equation*}
\begin{equation*}
 = \sum_{x \in \KerovCoorUp{\lambda}} \Big[\SimpleSer{\lambda}:\res^{\Sern{n+1}}_{\Sern{n}}\SimpleSer{\lambda+\smallydcell(x)}\Big]\frac{\dim(\SimpleSer{\lambda})\dim(\SimpleSer{\lambda+\smallydcell(x)})}{2^{\lengthparity(\lambda+\smallydcell(x))}} s(x)^k.
\end{equation*}
Thus, taking the normalized character gives
\begin{align*}
&\norcharrep{\lambda}(\pr{n-1}(\JM{n}^{2k})) \\ &= \sum_{x \in \KerovCoorUp{\lambda}} 2^{\lengthparity(\lambda+\smallydcell(x)) - \lengthparity(\lambda)}\Big[\SimpleSer{\lambda}:\res^{\Sern{n+1}}_{\Sern{n}}\SimpleSer{\lambda+\smallydcell(x)}\Big]\frac{\dim(\SimpleSer{\lambda + \smallydcell(x)}) }{2(n+1)\dim(\SimpleSer{\lambda})}s(x)^{k}.
\end{align*}
As $\dim(\Ind^{\Sern{n+1}}_{\Sern{n}} \SimpleSer{\lambda}) = 2(n+1)\dim(\SimpleSer{\lambda})$ and by Frobenious reciprocity
\begin{equation*}
\Big[\SimpleSer{\lambda}:\res^{\Sern{n+1}}_{\Sern{n}}\SimpleSer{\lambda+\smallydcell(x)}\Big] = \Big[\ind^{\Sern{n+1}}_{\Sern{n}}\SimpleSer{\lambda}:\SimpleSer{\lambda+\smallydcell(x)}\Big],
\end{equation*}
applying Lemma \ref{lemma-transition-prob-rep}.\ref{enum-lemma-transition-rep} gives the desired result.

\item The elements $c_i$ and the Jucys-Murphy elements $\JM{i}$ satisfy $\JM{i}c_i = -c_i\JM{i}$ and for $x = s_i \dots s_{n-1}c_n^{\epsilon}$ we have $x^{-1} = (-1)^{\epsilon}c_n s_{n-1}\dots s_i$. Therefore
\begin{align*}
\sum_{x \in \LcosSer{n}{n-1}} x\JM{n}^{r}x^{-1} 
&= \sum_{i = 1}^{n} s_i \dots s_{n-1}\JM{n}^{r}s_{n-1} \dots s_i \;\;- \;\; s_i \dots s_{n-1}c_{n}\JM{n}^{r}c_{n}s_{n-1}\dots s_i 
\\ \numberthis \label{eqn-JM-simplification}
&= \sum_{i = 1}^{n} s_i \dots s_{n-1}\JM{n}^{r}s_{n-1} \dots s_i \;\;- \;\; (-1)^{r+1}s_i \dots s_{n-1}\JM{n}^{r}s_{n-1}\dots s_i.
\end{align*}
When $r$ is odd, this is then equal to zero. When $r = 2k$, \eqref{eqn-JM-simplification} is equal to
\begin{equation*}
2\sum_{i = 1}^{n} s_i \dots s_{n-1}\JM{n}^{2k}s_{n-1} \dots s_i.
\end{equation*}
Since characters are invariant under conjugation, we have
\begin{equation*}
\norcharrep{\lambda}\Big(2\sum_{i = 1}^{n} s_i \dots s_{n-1}\JM{n}^{2k}s_{n-1} \dots s_i\Big) = 2n\norcharrep{\lambda}(\JM{n}^{2k}).
\end{equation*}
Decomposing $\JM{n}$ into its weight spaces then gives
\begin{align*}
2n\norcharrep{\lambda}(\JM{n}^{2k}) 
&= 2n\sum_{y \in \KerovCoorDown{\lambda}} \frac{\Big[\SimpleSer{\lambda - \smallydcell(y)}:\res^{\Sern{n}}_{\Sern{n-1}} \SimpleSer{\lambda}\Big]\dim(\SimpleSer{\lambda-\smallydcell(y)})\JMeig{y}^k}{\dim(\SimpleSer{\lambda})}\\
&= 2n\sum_{y \in \KerovCoorDown{\lambda}}\downtransition{\lambda}{\lambda - \smallydcell(y)}\JMeig{y}^k
\end{align*}
where the last equality uses Lemma \ref{lemma-transition-prob-rep}.\ref{enum-lemma-cotransition-rep}
\end{enumerate}
\end{proof}


\section{The subalgebra $\supersym$} \label{sect-gamma}

We recall relevant facts about the algebra $\supersym$ following~\cite{Mac15}. Let $p_k$ be the $k$th power sum symmetric function, 
and recall that, for $\rho \in \partitions$,
\begin{equation*}
p_\rho := \prod_{k=1}^{\ell(\rho)} p_{\rho_k}.
\end{equation*}

 $\Gamma$ can be described as the subalgebra of the symmetric functions generated by the odd power sums
\begin{equation*}
\Gamma = \MB{C}[p_1,p_3,p_5,\dots].
\end{equation*}

Elements of $\Gamma$ can be evaluated on partitions in the following way. Let $f \in \supersym$ and $\rho \in \partitions$, and define 
\begin{equation} \label{eqn-eval-map}
f(\rho) := f(\rho_1, \rho_2, \ldots, \rho_{\ell(\rho)}, 0, \ldots).
\end{equation}
Let $\funonYD$ denote the algebra of functions from $\strictpartitions$ to $\MB{C}$ with pointwise multiplication. 

\begin{proposition} \cite[Proposition 6.2]{IK99}) \label{prop-supersym-embed}
The algebra $\supersym$ embeds into $\funonYD$ via the evaluation map \eqref{eqn-eval-map}.
\end{proposition}


 An important linear basis of $\supersym$ is the Schur $Q$-functions $\{Q_\lambda\}$, indexed by $\strictpartitions$ (cf. \cite[Section III.8]{Mac15}). 
 
 Define numbers $X_\mu^\lambda$ for $\lambda \in \strictpartitions[n]$, $\mu \in \oddpartitions[n]$, via 
 \begin{equation} \label{eqn-relation-p-P}
 p_\mu = \sum_{\lambda \in \strictpartitions[n]} 2^{-\length{\lambda}}X_\mu^\lambda Q_\lambda.
 \end{equation}
 
 There is a ``factorial'' version of the Schur $Q$-functions, defined in \cite{Iv04}. For $\lambda \in \strictpartitions$, the {\emph{factorial Schur $Q$-polynomial}} corresponding to $\lambda$ is defined as:
 \begin{equation} 
 Q^*_{\lambda|N}(x_1,\ldots, x_N) := \frac{2^{\length{\lambda}}}{(N-l)!} \sum_{\omega \in S_N} \omega\left(x_1^{\downarrow \lambda_1} x_2^{\downarrow \lambda_2} \ldots x_l^{\downarrow \lambda_l} \prod_{\substack{1\leq i \leq l \\ i < j \leq N}} \frac{x_i+x_j}{x_i-x_j}\right).
 \end{equation}
 If $\ell(\lambda)>N$, then $Q^*_{\lambda|N}$ is defined to be 0. The collection $(Q_{\lambda|N}^*)_{N=1,2,\dots}$ defines an element of $\Gamma$, the {\emph{factorial Schur $Q$-function}} $Q_\lambda^*$. Factorial Schur $Q$-functions have the following useful properties.
 
 \begin{proposition}\cite{Iv01} Let $\lambda, \nu \in \strictpartitions$.  
 \begin{enumerate} 
 \item There exists $g \in \Gamma$ of degree less than $|\lambda|$ such that 
 \begin{equation*}
 Q_\lambda^* = Q_\lambda+g. 
 \end{equation*}
 \item The collection $\{Q_\lambda^*\}_{\lambda \in \strictpartitions}$ is a linear basis of $\Gamma$. 
 \item If $\nu \in \strictpartitions[k]$, $\lambda \in \strictpartitions[n]$ for $k \leq n$ and $\nu \not\subseteq \lambda$, $Q_\lambda^*(\nu) = 0.$  
 \end{enumerate}
 \end{proposition}
 
Let $\psi: \supersym \rightarrow \supersym$ be the linear map that sends $Q_\lambda \mapsto Q_\lambda^*$. For any $\mu \in \oddpartitions$, define the inhomogeneous analogue of the power sum $\shiftedpowersum_\mu := \psi(p_\mu) \in \Gamma$. Applying $\psi$ to both sides of \eqref{eqn-relation-p-P} gives 
\begin{equation*} 
\shiftedpowersum_\mu = \sum_{\lambda \in \strictpartitions[k]} 2^{-\length{\lambda}}X_\mu^\lambda Q_\lambda^*.
\end{equation*}
It also follows from the fact that $X_\mu^\lambda = 2^{-\length{\mu} + \frac{\length{\lambda} -\lengthparity(\lambda)}{2}}\SerChar{\lambda}(\mu)$ \cite[Proposition 3.3]{Iv01} and 
\begin{equation*}
Q_\lambda = \sum_{\mu \in \oddpartitions[n]} \frac{2^{\length{\mu}}}{z_\mu} X^{\lambda}_\mu p_\mu
\end{equation*}
that
\begin{equation} \label{eqn-P-in-terms-of-frakp}
Q_\lambda^* = 2^{\frac{\length{\lambda}-\lengthparity(\lambda)}{2}} \sum_{\mu \in \oddpartitions[n]} \frac{\SerChar{\lambda}(\mu)}{z_\mu} \shiftedpowersum_\mu.
\end{equation}
The elements $\{\shiftedpowersum_\mu\}_{\mu \in \strictpartitions}$ were first studied in \cite{Iv01}, where Ivanov proves that they satisfy the following properties.
 
 \begin{proposition}\cite{Iv01} \label{prop-properties-of-pfrak} 
 Let $\mu \in \oddpartitions[k]$ and $\lambda \in \strictpartitions[n]$. \begin{enumerate}
 \item There exists $g \in \Gamma$ of degree less than $|\mu|$ such that 
    \begin{equation*}
    \shiftedpowersum_\mu = p_\mu + g.
    \end{equation*}
     \item The family $(\shiftedpowersum_\mu)_{\mu \in \oddpartitions}$ is a linear basis of $\Gamma.$
     \item \label{item-evaluation-property}
     \begin{equation*}
     \shiftedpowersum_\mu(\lambda) = \left\{\begin{array}{lr} n^{\downarrow k} \cdot \displaystyle\frac{X_{\mu \cup (1^{n-k})}^\lambda}{g_\lambda} \qquad & \text{if } |\lambda|\geq |\mu|, \\ 0 \qquad & \text{otherwise} \end{array}\right.
     \end{equation*}
     where in particular $g_\lambda = X_{1^{|\lambda|}}^\lambda$. 
     \item Let $\gamma \in \oddpartitions$. Define $\mu \cup \gamma$ to be the partition formed by taking the disjoint union of parts of $\mu$ and $\gamma$ and rearranging them in decreasing order. Then there exists $g \in \supersym$ of degree less than $|\mu \cup \gamma|$ such that $$\shiftedpowersum_\mu \cdot \shiftedpowersum_\gamma = \shiftedpowersum_{\mu \cup \gamma} + g.$$ 
 \end{enumerate}
 \end{proposition}
 
 As a Corollary to part \ref{item-evaluation-property} of the above Proposition, we have another formula for the value of $\shiftedpowersum_\rho$. 
 
 \begin{corollary} \cite{Iv01} \label{cor-value-of-shiftedpower}
 Let $\mu \in \oddpartitions[k]$ and $\lambda \in \strictpartitions[n]$. We have 
 \begin{equation*}
 \shiftedpowersum_\mu(\lambda) = 2^{k-\ell(\mu)} n^{\downarrow k}\frac{\chi^\lambda (\mu \cup 1^{n-k})}{\chi^\lambda(1^n)}.
 \end{equation*}
 \end{corollary}
 
 \begin{corollary} \label{coro-pfrak-is-gen/basis} 
The elements $\{\shiftedpowersum_{2k+1}\}_{k \geq 0}$ are algebraically independent and generate $\supersym$.
\end{corollary}  

It is shown in \cite{Pet09} that viewed as elements of $\funonYD$, $\{\upmoment{k}\}_{k \geq 1}$ and $\{\downmoment{k}\}_{k \geq 1}$ belong to $\supersym$.

\begin{proposition} \cite[Corollary 4.7]{Pet09}
The elements $\{\upmoment{k}\}_{k \geq 1}$ and $\{\downmoment{k}\}_{k \geq 1}$ are each sets of algebraically independent generators of $\supersym$,
 and
\begin{equation*}
\deg(\upmoment{k}) = \deg(\downmoment{k}) = 2k -1.
\end{equation*}
\end{proposition}
 
 
 \section{The twisted Heisenberg category} \label{sect-twisted-Heisenberg}
 
 \subsection{The definition of $\Heis$}
The twisted Heisenberg category $\Heis$ was introduced by Cautis and Sussan in \cite{CS15}. It is a $\mathbb{Z}/2\mathbb{Z}$-graded additive monoidal category whose morphisms are described diagrammatically as oriented compact 1-manifolds immersed in $\MB{R} \times [0,1]$. There is an injective algebra homomorphism from the twisted Heisenberg algebra into the split Grothendieck group of $K_0(\operatorname{Kar}(\Heis))$, where $\operatorname{Kar}$ denotes the Karoubi envelope (idempotent completion). As in the untwisted case, this map is conjecturally surjective.

\begin{remark}Cautis and Sussan define their version of $\Heis$ to be idempotent complete; since the center of a category remains invariant under passage to the idempotent completion, we use the non-idempotent complete version here. All results that hold for the center of $\Heis$ also hold for the center of its idempotent completion. \end{remark}

The objects of $\Heis$ are monoidally generated by $P$ and $Q$, so that a generic object in $\Heis$ is a direct sum of sequences of $P$'s and $Q$'s. We denote the empty sequence, which is the unit object of $\Heis$, by $\UnitModule$. The morphisms of $\mathcal{H}_{tw}$ are generated by oriented planar diagrams up to boundary fixing isotopies,
 with generators 
\begin{equation}
 \hspace{0.9cm}
\begin{tikzpicture}
\draw[thick,->](-2,-0.25) to (-2,0.75);
\draw (-1.99,0.25) circle [radius=2pt];
\node at (-1.6,-.25) {,};
\draw[thick,<-](-1,-0.25) to (-1,0.75);
\draw (-1,0.25) circle [radius=2pt];
\node at (-.5,-.25) {,};
\draw[thick,->] (0,-0.25) to (1,0.75);
\draw[thick,->] (1,-0.25) to (0,0.75);
\node at (1.25,-.25) {,};
\draw[thick,->] (1.5,0.5) arc (180:360:5mm);
\node at (2.6,-.25) {,};
\draw[thick,<-] (3,0.5) arc (180:360:5mm);
\node at (4.2,-.25) {,};
\draw[thick,->] (4.5,-.2) arc (180:0:5mm);
\node at (5.8,-.25) {,};
\draw[thick,<-] (6.2,-.2) arc (180:0:5mm);
\end{tikzpicture}
\end{equation}
where the first diagram corresponds to a map $P\rightarrow P\{1\}$ and the second diagram corresponds to a map $Q\rightarrow Q\{1\}$, where $\{1\}$ denotes the $\mathbb{Z}/2\mathbb{Z}$-grading shift. These generators satisfy the following relations:

\begin{equation} \label{up down double crossings}
\begin{tikzpicture}[baseline=(current bounding box.center),scale=0.6]
\draw[thick,->] (0,0) .. controls (1,1) .. (0,2);
\draw[thick,->] (1,2) .. controls (0,1) .. (1,0);
\node at (1.75,1) {$=$};
\draw[thick,->] (2.5,0) -- (2.5,2);
\draw[thick,->] (3.5,2) -- (3.5,0);
\node at (4,0) {,};
\draw[->,thick] (5.25,0) .. controls (6.25,1) .. (5.25,2);
\draw[->,thick] (6.25,0) .. controls (5.25,1) .. (6.25,2);
\node at (7,1) {$=$};
\draw[->,thick] (7.75,0) -- (7.75,2);
\draw[->,thick] (8.75,0) -- (8.75,2);
\node at (9.25,0) {,};
\draw[->,thick] (10.25,0) -- (12.25,2);
\draw[->,thick] (12.25,0) -- (10.25,2);
\draw[->,thick] (11.25,0) .. controls (12.25,1) .. (11.25,2);
\node at (13,1) {$=$};
\draw[->,thick] (13.75,0) -- (15.75,2);
\draw[->,thick] (15.75,0) -- (13.75,2);
\draw[->,thick] (14.75,0) .. controls (13.75,1) .. (14.75,2);	
\node at (16.1,0) {,};
\end{tikzpicture}
\end{equation}

\begin{equation} \label{eqn-symmetric-group-relations}
\begin{tikzpicture}[baseline=(current bounding box.center),scale=0.6]]
\draw[thick,->] (6,2) .. controls (7,1) .. (6,0);
\draw[thick,->] (7,0) .. controls (6,1) .. (7,2);
\node at (8,1) {$=$};
\draw[thick,->] (9,2) -- (9,0);
\draw[thick,->] (10,0) -- (10,2);
\node at (11,1) {$-$};        	
\draw[thick,->] (12,2) arc (180:360:0.75);
\draw[thick,->] (13.5,0) arc (0:180:0.75);
\node at (14.3,1) {$-$};
\draw[thick,->] (15,2) arc (180:360:0.75);
\draw[thick,->] (16.5,0) arc (0:180:0.75);
\draw (15.12,1.6) circle [radius = 3pt];
\draw (16.4,.4) circle [radius = 3pt];
\node at (17,0) {,};
\end{tikzpicture}
\hspace{1.5cm}
\begin{tikzpicture}[baseline=(current bounding box.center),scale=0.6]
\draw[thick,->] (3,2) arc (-180:180:5mm);
\end{tikzpicture}\hspace{6pt}
=1,
\hspace{1.5cm}
\begin{tikzpicture}[baseline=(current bounding box.center),scale=1.5]
\draw[thick](1,0) to [out=90, in=-75](0.95,0.5);
\draw[thick](0.95,0.5) arc (5:355:2mm);
\draw[thick,->] (0.95,0.46) to [out=75, in=270] (1,1);
\end{tikzpicture}\hspace{6pt}
=0,
\end{equation}


\begin{equation}\label{caps}
    \begin{tikzpicture}[baseline=(current bounding box.center)]
    \draw[thick,<-] (3,2) arc (180:0:5mm);
    \draw (3.1,2.3) circle [radius=1.8pt];
    \end{tikzpicture}\hspace{6pt}
    =
    \hspace{6pt}-\hspace{4pt}
    \begin{tikzpicture}[baseline=(current bounding box.center)]
    \draw[thick,<-] (3,2) arc (180:0:5mm);
    \draw (3.95,2.2) circle [radius=1.8pt];
    \node at (4.2,2) {,};
    \end{tikzpicture}
    \hspace{1cm}
    \begin{tikzpicture}[baseline=(current bounding box.center)]
    \draw[thick,<-] (3,2) arc (180:360:5mm);
    \draw (3.07,1.75) circle [radius=1.8pt];
    \end{tikzpicture}\hspace{6pt}
    =
    \hspace{6pt}-\hspace{4pt}
    \begin{tikzpicture}[baseline=(current bounding box.center)]
    \draw[thick,<-] (3,2) arc (180:360:5mm);
    \draw (3.95,1.8) circle [radius=1.8pt];
    \end{tikzpicture}
\end{equation}

\begin{equation}\label{d01=0}
    \begin{tikzpicture}[baseline=(current bounding box.center),scale=0.6]
    \draw[thick,->] (0,0) to (0,2);
    \end{tikzpicture}\hspace{6pt}
    =
    \hspace{6pt}
    \begin{tikzpicture}[baseline=(current bounding box.center),scale=0.6]
    \draw[thick,->] (0,0) to (0,2);
    \draw (0,0.6) circle [radius=3pt];
    \draw (0,1.2) circle [radius=3pt];
    \node at (.5,0) {,};
    \end{tikzpicture}
    \hspace{1cm}
    \begin{tikzpicture}[baseline=(current bounding box.center),scale=0.6]
    \draw[thick,<-] (0,0) to (0,2);
    \end{tikzpicture}\hspace{6pt}
    =
    \hspace{6pt}-\hspace{1mm}
    \begin{tikzpicture}[baseline=(current bounding box.center),scale=0.6]
    \draw[thick,<-] (0,0) to (0,2);
    \draw (0,0.6) circle [radius=3pt];
    \draw (0,1.2) circle [radius=3pt];
    \node at (.5,0) {,};
    \end{tikzpicture}
\hspace{1cm}
    \begin{tikzpicture}[baseline=(current bounding box.center),scale=0.6]
    \draw[thick,->] (3,2) arc (-180:180:5mm);
    \draw (3.95,2.2) circle [radius=3pt];
    \end{tikzpicture}\hspace{6pt}
    =\hspace{6pt}0,
\hspace{1.5cm}
    \begin{tikzpicture}[baseline=(current bounding box.center),scale=0.6]
    \draw[thick,->] (0,0) to (0,2);
    \draw (0,1.6) circle [radius=3pt];
    \draw[fill=] (0.2,1) circle [radius=0.3pt];
    \draw[fill=] (0.4,1) circle [radius=0.3pt];
    \draw[fill=] (0.6,1) circle [radius=0.3pt];
    \draw[thick,->] (0.8,0) to (0.8,2);
    \draw (0.8,0.3) circle [radius=3pt];
    \end{tikzpicture}\hspace{6pt}
    =\hspace{6pt}-\hspace{6pt}
     \begin{tikzpicture}[baseline=(current bounding box.center),scale=0.6]
    \draw[thick,->] (0,0) to (0,2);
    \draw (0,0.3) circle [radius=3pt];
    \draw[fill=] (0.2,1) circle [radius=0.3pt];
    \draw[fill=] (0.4,1) circle [radius=0.3pt];
    \draw[fill=] (0.6,1) circle [radius=0.3pt];
    \draw[thick,->] (0.8,0) to (0.8,2);
    \draw (0.8,1.6) circle [radius=3pt];
    \node at (1,.2) {.};
    \end{tikzpicture}
\end{equation}

Generators commute in all other situations (for instance, hollow dots commute with crossings).

If we denote a right-twist curl by a dot \quad
    \begin{tikzpicture}[baseline=(current bounding box.center),scale=1.5]
\draw[thick] (1,0) to [out=90, in=85](1.05,0.5);
\draw[thick] (1.05,0.5) arc (-175:175:2mm);
\draw[thick,->] (1.05,0.46) to [out=95, in=270] (1,1);
\end{tikzpicture}\hspace{2pt}
    :=
    \hspace{2pt}
    \begin{tikzpicture}[baseline=(current bounding box.center),scale=0.6]
    \draw[thick,->](0,0) to (0,2);
    \draw[fill](0,1) circle[radius=3pt];
    \end{tikzpicture}
\quad then we have the following relations:

\begin{equation}\label{anticommute}
    \begin{tikzpicture}[baseline=(current bounding box.center),scale=0.6]
    \draw[thick,->](0,0) to (0,2);
    \draw[fill](0,0.6) circle[radius=3pt];
    \draw(0,1.2) circle[radius=3pt];
    \end{tikzpicture}\hspace{6pt}
    =
    \hspace{6pt}-\hspace{4pt}
    \begin{tikzpicture}[baseline=(current bounding box.center),scale=0.6]
    \draw[thick,->](0,0) to (0,2);
    \draw(0,0.6) circle[radius=3pt];
    \draw[fill](0,1.2) circle[radius=3pt];
    \node at (.5,.15) {,};
    \end{tikzpicture}
\hspace{1.5cm}
    \begin{tikzpicture}[baseline=(current bounding box.center),scale=0.6]
    \draw[thick,->] (0,0) to (0,2);
    \draw[fill] (0,1.6) circle [radius=3pt];
    \draw[fill] (0.2,1) circle [radius=0.3pt];
    \draw[fill] (0.4,1) circle [radius=0.3pt];
    \draw[fill] (0.6,1) circle [radius=0.3pt];
    \draw[thick,->] (0.8,0) to (0.8,2);
    \draw[fill] (0.8,0.3) circle [radius=3pt];
    \end{tikzpicture}\hspace{6pt}
    =
    \hspace{6pt}
     \begin{tikzpicture}[baseline=(current bounding box.center),scale=0.6]
    \draw[thick,->] (0,0) to (0,2);
    \draw[fill] (0,0.3) circle [radius=3pt];
    \draw[fill] (0.2,1) circle [radius=0.3pt];
    \draw[fill] (0.4,1) circle [radius=0.3pt];
    \draw[fill] (0.6,1) circle [radius=0.3pt];
    \draw[thick,->] (0.8,0) to (0.8,2);
    \draw[fill] (0.8,1.6) circle [radius=3pt]; 
    \node at (1,.15) {,};  
    \end{tikzpicture}
    \hspace{1.5cm}
    \begin{tikzpicture}[baseline=(current bounding box.center),scale=0.6]
    \draw[thick,->] (0,0) to (0,2);
    \draw (0,1.6) circle [radius=3pt];
    \draw[fill] (0.2,1) circle [radius=0.3pt];
    \draw[fill] (0.4,1) circle [radius=0.3pt];
    \draw[fill] (0.6,1) circle [radius=0.3pt];
    \draw[thick,->] (0.8,0) to (0.8,2);
    \draw[fill] (0.8,0.3) circle [radius=3pt];
    \end{tikzpicture}\hspace{6pt}
    =
    \hspace{6pt}
     \begin{tikzpicture}[baseline=(current bounding box.center),scale=0.6]
    \draw[thick,->] (0,0) to (0,2);
    \draw (0,0.3) circle [radius=3pt];
    \draw[fill] (0.2,1) circle [radius=0.3pt];
    \draw[fill] (0.4,1) circle [radius=0.3pt];
    \draw[fill] (0.6,1) circle [radius=0.3pt];
    \draw[thick,->] (0.8,0) to (0.8,2);
    \draw[fill] (0.8,1.6) circle [radius=3pt];
    \node at (1,.15) {.};  
    \end{tikzpicture}
\end{equation}

From \cite{OR17} we have the following ``dot sliding'' relations,

\begin{equation}\label{dotSlide: bottomLeft}
\begin{tikzpicture}[baseline=(current bounding box.center),scale=0.6]
    \draw[thick,->](0,0) to (1,2);
    \draw[thick,->](1,0) to (0,2);
    \draw[fill](0.25,0.5) circle[radius=3pt];
\end{tikzpicture}\hspace{6pt}
=
\hspace{6pt}
\begin{tikzpicture}[baseline=(current bounding box.center),scale=0.6]
    \draw[thick,->](0,0) to (1,2);
    \draw[thick,->](1,0) to (0,2);
    \draw[fill](0.75,1.5) circle[radius=3pt];
\end{tikzpicture}\hspace{6pt}
+\hspace{2mm}
\begin{tikzpicture}[baseline=(current bounding box.center),scale=0.6]
    \draw[thick,->](0,0) to (0,2);
    \draw[thick,->](0.5,0) to (0.5,2);
\end{tikzpicture}
\hspace{2mm}+\hspace{2mm}
\begin{tikzpicture}[baseline=(current bounding box.center),scale=0.6]
    \draw[thick,->](0,0) to (0,2);
    \draw[thick,->](0.5,0) to (0.5,2);
    \draw(0,1.2) circle[radius=3pt];
    \draw(0.5,0.6) circle[radius=3pt];
    \node at (1,.2) {,};
\end{tikzpicture}
\end{equation}

\begin{equation}\label{dotSlide: topLeft}
\begin{tikzpicture}[baseline=(current bounding box.center),scale=0.6]
    \draw[thick,->](0,0) to (1,2);
    \draw[thick,->](1,0) to (0,2);
    \draw[fill](0.25,1.5) circle[radius=3pt];
\end{tikzpicture}\hspace{6pt}
=
\hspace{6pt}
\begin{tikzpicture}[baseline=(current bounding box.center),scale=0.6]
    \draw[thick,->](0,0) to (1,2);
    \draw[thick,->](1,0) to (0,2);
    \draw[fill](0.75,0.5) circle[radius=3pt];
\end{tikzpicture}\hspace{6pt}
+
\hspace{2mm}
\begin{tikzpicture}[baseline=(current bounding box.center),scale=0.6]
    \draw[thick,->](0,0) to (0,2);
    \draw[thick,->](0.5,0) to (0.5,2);
\end{tikzpicture}\hspace{2mm}
-
\hspace{2mm} 
\begin{tikzpicture}[baseline=(current bounding box.center),scale=0.6]
    \draw[thick,->](0,0) to (0,2);
    \draw[thick,->](0.5,0) to (0.5,2);
    \draw(0,1.2) circle[radius=3pt];
    \draw(0.5,0.6) circle[radius=3pt];
    \node at (1,.2) {.};
\end{tikzpicture}
\end{equation}

We can also move clockwise ``bubbles'' with dots on them through strands.

\begin{lemma} \label{lemma-bubble-sliding-relation}
Let $n \geq 0$, then
\begin{center}
\begin{tikzpicture}[scale = .7]

\draw[thick,->] (0,0) circle [radius = 12pt];
\node at (-.42,0) {\arrowlines};
\draw[fill=black] (-.3,.27) circle [radius = 2pt]; 
\node at (-.7,.4) {\small{$2n$}};

\draw[thick,->] (1,-1)--(1,1);

\node at (2,0) {$=$};

\draw[thick,->] (3,-1)--(3,1);

\draw[thick,->] (4,0) circle [radius = 12pt];
\node at (3.58,0) {\arrowlines};
\draw[fill=black] (3.7,.27) circle [radius = 2pt]; 
\node at (3.3,.4) {\small{$2n$}};

\node at (4.9,0) {$+$};

\node at (5.9,0) {\small{$(4n$$+$$2)$}};
\draw[thick,->] (6.7,-1)--(6.7,1);
\draw[fill = black] (6.7,.4) circle [radius = 2pt];
\node at (7.1,.55) {\small{$2n$}};

\node at (7.7,0) {$-$};

\node at (9.7,0) {$\displaystyle 2 \sum_{a + b = 2n-1} \sum_{k = 1}^b$};

\draw[thick,->] (13,0) circle [radius = 12pt];
\node at (12.58,0) {\arrowlines};
\draw[fill=black] (13.3,.27) circle [radius = 2pt]; 
\draw[fill=black] (11.7,.4) circle [radius = 2pt];
\node at (13.4,.7) {\small{$b$-$k$}};
\node at (12.25,.53) {\small{$a{+}k$}};

\draw[thick,->] (11.7,-1)--(11.7,1);

\node at (13.6,-.8) {.};

\end{tikzpicture}
\end{center}
\end{lemma}

\begin{proof}
This follows from the proof of \cite[Lemma 4.7]{OR17} along with the dot sliding relation \eqref{dotSlide: topLeft}.
\end{proof}

As a consequence of relations \eqref{up down double crossings} and \eqref{eqn-symmetric-group-relations}, there are homomorphisms $\symtoHom{n}: \Sern{n}^{opp} \rightarrow \Hom_{\Heis}(P^n)$ which send
\vspace{1mm}
\begin{center}
\begin{tikzpicture}[scale = .6]

\draw[->,thick] (-.3,0)--(1,0);
\draw[thick] (-.3,.1)--(-.3,-.1);
\node at (.2,.3) {\scriptsize{$\symtoHom{n}$}};
\node at (-1,0) {$s_k$};

\draw[thick,->] (2,-1) -- (2,1);
\draw[thick,->] (3.2,-1) -- (3.2,1);
\draw[thick,->] (3.7,-1) to [out =90,in = 270] (4.2,1);
\draw[thick,->] (4.2,-1) to [out =90,in = 270] (3.7,1);
\draw[thick,->] (4.7,-1) -- (4.7,1);
\draw[thick,->] (5.9,-1) -- (5.9,1);
\draw [decorate,decoration={brace,amplitude=3pt},xshift=-4pt,yshift=0pt]
(3.3,-1.2) -- (2,-1.2) node [black,midway,xshift=0cm,yshift = -.3cm] 
{\scriptsize{ $k$-$1$ strands}};
\draw [decorate,decoration={brace,amplitude=3pt},xshift=-4pt,yshift=0pt]
(6.1,-1.2) -- (4.9,-1.2) node [black,midway,xshift=0cm,yshift = -.3cm] 
{\scriptsize{ $n$-$k$-$1$ strands}};

\node at (2.6,0) {\dots};
\node at (5.3,0) {\dots};

\end{tikzpicture} 
\hspace{1.5 cm}
\begin{tikzpicture}[scale = .6]
\draw[->,thick] (-.3,0)--(1,0);
\draw[thick] (-.3,.1)--(-.3,-.1);
\node at (.2,.3) {\scriptsize{$\symtoHom{n}$}};
\node at (-1,0) {$c_k$};

\draw[thick,->] (2,-1) -- (2,1);
\draw[thick,->] (3.2,-1) -- (3.2,1);
\draw[thick,->] (3.95,-1) to [out =90,in = 270] (3.95,1);
\draw (3.95,0) circle [radius=3.3pt];
\draw[thick,->] (4.7,-1) -- (4.7,1);
\draw[thick,->] (5.9,-1) -- (5.9,1);
\node at (6.2,-.6) {$.$};
\draw [decorate,decoration={brace,amplitude=3pt},xshift=-4pt,yshift=0pt]
(3.3,-1.2) -- (2,-1.2) node [black,midway,xshift=0cm,yshift = -.3cm] 
{\scriptsize{ $k$-$1$ strands}};
\draw [decorate,decoration={brace,amplitude=3pt},xshift=-4pt,yshift=0pt]
(6.1,-1.2) -- (4.9,-1.2) node [black,midway,xshift=0cm,yshift = -.3cm] 
{\scriptsize{ $n$-$k$ strands}};

\node at (2.6,0) {\dots};
\node at (5.3,0) {\dots};

\end{tikzpicture}
\end{center}

In order to simplify our diagrams we write the image of $x \in \Sern{n}$ under $\symtoHom{n}$ as

\vspace{2mm}
\begin{equation} \label{picture-sym-group-emedding}
\begin{tikzpicture}[scale = .8]

\node at (-4.4,0) {$\symtoHom{n}(x) \;\; =: \;\;$};

\draw[thick,->] (-1.7,-1)--(-1.7,1);
\draw[thick,->] (-.8,-1)--(-.8,1);
\draw[thick,->] (-2,-1)--(-2,1);


\node at (-1.35,.6) {$\cdot$};
\node at (-1.2,.6) {$\cdot$};
\node at (-1.05,.6) {$\cdot$};

\draw[fill=white,thick] (-2.3,-.3) rectangle (-.5,.3);

\node at (-1.3,0) {\scriptsize{$x$}};

\node at (-.3,-.9) {.};

\draw [decorate,decoration={brace,amplitude=3pt},xshift=-4pt,yshift=0pt]
(-.6,-1.2) -- (-2,-1.2) node [black,midway,xshift=0cm,yshift = -.3cm] 
{\scriptsize{$n$ strands}};

\end{tikzpicture}
\end{equation}


 \subsection{The center of $\Heis$} \label{Section-H-Center}
 
The center of a $\Bbbk$-linear monoidal category $\MC{C}$ is defined to be the endomorphism algebra of the monoidal unit $\UnitModule$ of $\MC{C}$, that is $\End_{\MC{C}}(\UnitModule)$. In a diagrammatic category such as $\Heis$, $\Endid$ is then by definition the commutative algebra of closed diagrams where multiplication of two closed diagrams corresponds to placing them next to each other.

We define the following elements of $\Endid$:
\begin{equation} \label{eqn-bubble-def}
d_{2n} := 
\begin{tikzpicture}[baseline=(current bounding box.center)]
\draw[thick,<-] (3,2) arc (-180:180:5mm);
\fill (3.95,2.2) circle [radius=2pt];
\node at (4.3,2.2) {$2n$};
\end{tikzpicture} 
\quad\quad \text{and}\quad\quad \bar{d}_{2n}:=
\begin{tikzpicture}[baseline=(current bounding box.center)]
\draw[thick,->] (3,2) arc (-180:180:5mm);
\fill (3.95,2.2) circle [radius=2pt];
\node at (4.3,2.2) {$2n$};
\end{tikzpicture}.
\end{equation}
It follows from the defining relations of $\mathcal{H}_{tw}$ that bubbles with an even number of hollow dots are equivalent to bubbles with no hollow dots, and that bubbles with an odd number of hollow dots are zero. Additionally, bubbles with odd numbers of solid dots are zero. Hence it suffices to consider the bubbles with an even number of solid dots and no hollow dots.
\begin{proposition}\cite[Proposition 4.4]{OR17} \label{prop-generating-set}
The elements $\{d_{2n}\}_{n \geq 0}$ are algebraically independent generators of $\Endid$, i.e. there is an isomorphism
\begin{equation*}
\Endid \cong \MB{C}[d_0,d_2,d_4,\dots].
\end{equation*}
\end{proposition} 
 
 The elements $\{d_{2n}\}_{n \geq 0}$ and $\{\bar{d}_{2n}\}_{n \geq 0}$ are related via a recursive relation.
\begin{proposition} \cite[Lemma 4.3]{OR17} \label{prop-recursive}
For $n \geq 1$,
\begin{equation*}
\bar{d}_{2n} = \sum_{2a+2b = 2n-2} \bar{d}_{2a}d_{2b}.
\end{equation*}
\end{proposition}

\begin{corollary}
The elements $\{\bar{d}_{2n}\}_{n \geq 1}$ are another algebraically independent generating set of $\Endid$.
\end{corollary}
 
Another natural set of diagrams in $\Endid$ come from the closure of permutations. We define

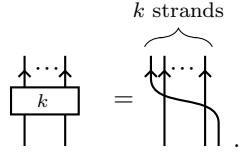
\begin{center}
\begin{tikzpicture}[scale=0.6]

\draw[thick] (-1.7,-1)--(-1.7,1);
\draw[thick] (-.8,-1)--(-.8,1);

\node at (-1.7,.6) {\arrowlines};
\node at (-.8,.6) {\arrowlines};

\node at (-1.4,.7) {$\cdot$};
\node at (-1.2,.7) {$\cdot$};
\node at (-1.0,.7) {$\cdot$};

\draw[fill=white,thick] (-2,-.3) rectangle (-.5,.3);

\node at (-1.3,0) {\scriptsize{$k$}};

\node at (.5,0) {$=$};

\draw [decorate,decoration={brace,amplitude=7pt},xshift=-4pt,yshift=0pt]
(1.1,1.2) -- (2.6,1.2) node [black,midway,xshift=0cm,yshift = .5cm] 
{\footnotesize $k$ strands};

\draw[thick] (2.3,-1)--(2.3,1);
\draw[thick] (1.4,-1)--(1.4,1);
\draw[thick] (2.6,-1) to [out =90,in = 270] (2.6,-.5) to [out = 90, in = 270] (1.1,.5) to [out = 90, in =270] (1.1,1);

\node at (2.3,.6) {\arrowlines};
\node at (1.4,.6) {\arrowlines};
\node at (1.1,.6) {\arrowlines};

\node at (2,.7) {$\cdot$};
\node at (1.8,.7) {$\cdot$};
\node at (1.6,.7) {$\cdot$};

\node at (3,-1) {$.$};

\end{tikzpicture}
\end{center}
For $\nu = (\nu_1,\dots, \nu_{r}) \in \partitionsn{k}$, let

\begin{equation} \label{tikzpicture-def-for-lambda}
\begin{tikzpicture}[scale=0.6]

\draw[thick] (-1.7,-1)--(-1.7,1);
\draw[thick] (-.8,-1)--(-.8,1);

\node at (-1.7,.6) {\arrowlines};
\node at (-.8,.6) {\arrowlines};

\node at (-1.4,.7) {$\cdot$};
\node at (-1.2,.7) {$\cdot$};
\node at (-1.0,.7) {$\cdot$};

\draw[fill=white,thick] (-2,-.3) rectangle (-.5,.3);

\node at (-1.3,0) {\scriptsize{$\nu$}};

\node at (.5,0) {$:=$};

\draw[thick] (2.3,-1)--(2.3,1);
\draw[thick] (1.4,-1)--(1.4,1);

\draw[fill=white,thick] (2.6,-.3) rectangle (1.1,.3);

\node at (2.3,.6) {\arrowlines};
\node at (1.4,.6) {\arrowlines};

\node at (2,.7) {$\cdot$};
\node at (1.8,.7) {$\cdot$};
\node at (1.6,.7) {$\cdot$};

\node at (3.2,0) {$\cdot$};
\node at (3.4,0) {$\cdot$};
\node at (3.6,0) {$\cdot$};

\node at (4.7,.7) {$\cdot$};
\node at (4.9,.7) {$\cdot$};
\node at (5.1,.7) {$\cdot$};

\draw[thick] (4.5,-1)--(4.5,1);
\draw[thick] (5.4,-1)--(5.4,1);

\node at (4.5,.6) {\arrowlines};
\node at (5.4,.6) {\arrowlines};

\draw[fill=white,thick] (4.2,-.3) rectangle (5.7,.3);

\node at (5.05,0) {\scriptsize{$\nu_r$}};
\node at (1.85,0) {\scriptsize{$\nu_1$}};

\end{tikzpicture}
\end{equation}
then we define

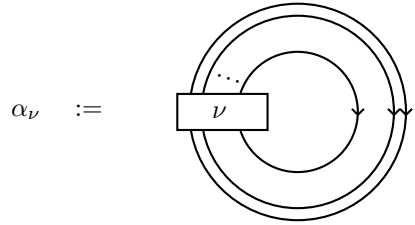
\begin{center}
\begin{tikzpicture}[scale = .8]

\node at (-4,0) {$\alpha_\nu \quad := $};

\draw[thick] (0,0) circle (1.8 cm);
\draw[thick] (0,0) circle (1.6 cm);
\draw[thick] (0,0) circle (1 cm);

\draw[fill=white,thick] (-2,-.3) rectangle (-.5,.3);

\node at (-1,.5) {$\cdot$};
\node at (-1.15,.55) {$\cdot$};
\node at (-1.3,.6) {$\cdot$};

\node at (-1.3,0) {$\nu$};

\node[rotate = 180] at (1.8,0) {\arrowlines};
\node[rotate = 180] at (1.6,0) {\arrowlines};
\node[rotate = 180] at (1,0) {\arrowlines};

\end{tikzpicture}
\end{center}

\omitt{We write
\vspace{3mm}
\begin{center}
\begin{tikzpicture}

\node at (-4,0) {$\alpha_\rho \quad := $};

\draw[thick] (0,0) circle (1.8 cm);
\draw[thick] (0,0) circle (1.6 cm);
\draw[thick] (0,0) circle (1 cm);

\draw[fill=white,thick] (-2,-.3) rectangle (-.5,.3);

\node at (-1,.5) {$\cdot$};
\node at (-1.15,.55) {$\cdot$};
\node at (-1.3,.6) {$\cdot$};

\node at (-1.3,0) {$\sigma_{\rho}$};

\node[rotate = 180] at (1.8,0) {\arrowlines};
\node[rotate = 180] at (1.6,0) {\arrowlines};
\node[rotate = 180] at (1,0) {\arrowlines};

\node at (2,-1.2) {.};

\end{tikzpicture}
\end{center}
\vspace{3mm}
where $\rho$ is a partition and $\sigma_{\rho}$ is a permutation with cycle type $\rho$.}
\vspace{3mm}
We set $\alpha_k := \alpha_{(k)}$. 

\begin{remark} \label{remark-closure-don't-care-about-conj-class-rep}
By an argument given in \cite{KLM16}, this notation is consistent with \eqref{picture-sym-group-emedding} in the sense that the closures of the diagrams $\symtoHom{n}(\omega_1)$ and $\symtoHom{n}(\omega_2)$, for $\omega_1, \omega_2 \in \Sy{n}$ and $\omega_1$ and $\omega_2$ having the same cycle type $\nu$, are both equal to $\alpha_\nu$.

\end{remark}

One can impose a grading on $\Endid$ by setting:
\begin{equation} \label{eqn-grading-def}
\deg(d_{0}) = 0 \quad\quad \text{and} \quad\quad \deg(d_{2k}) = 2k+1.
\end{equation}

\begin{lemma} \label{lemma-degree-of-alpha}
In terms of the grading defined by \eqref{eqn-grading-def}, 
\begin{equation*}
\alpha_{2k+1} = d_{2k} + \loworderterms 
\end{equation*}
\end{lemma}

\begin{proof}
We can reduce the diagram $\alpha_{2k+1}$ to a polynomial in $d_0, d_2, d_4, \dots$ via repeated application of the dot sliding moves \eqref{dotSlide: bottomLeft}-\eqref{dotSlide: topLeft} and clockwise bubble sliding move from Lemma \ref{lemma-bubble-sliding-relation}. Specifically, the innermost crossing in $\alpha_{2k+1}$ forms part of a right twist curl, i.e. a solid dot. Modulo terms with at least $2$ fewer crossings (all of which also have exactly two connected components), this dot can be moved to the outside of the diagram via repeated applications of relations \eqref{dotSlide: bottomLeft}-\eqref{dotSlide: topLeft}. The resulting diagram once again has a right twist curl in its center. Repeating this process $2k$ times yields $d_{2k}$, plus terms which have $2$ fewer crossings (including those in right twists curls written as solid dots) and exactly two connected components. 

It remains to reduce the lower order terms to polynomials in $d_0, d_2, d_4, \dots$. Note that the dot sliding moves \eqref{dotSlide: bottomLeft}-\eqref{dotSlide: topLeft} and bubble sliding relation in Lemma \ref{lemma-bubble-sliding-relation} never increase the number of crossings, and can only increase the number of connected components by at most 1 with a corresponding decrease in crossings by 2 (see the second and third term on the right side of \eqref{dotSlide: bottomLeft}-\eqref{dotSlide: topLeft}). It follows that while reducing each of these additional terms to monomials in bubbles, in each diagram
\begin{equation*}
\# \text{crossings} \;\; + \;\; \# \text{connected components} \leq (2k - 2j) + (j+1) = 2k -j +1
\end{equation*}
where $1 \leq j \leq k$ is half the number of crossings that have been resolved via a dot or bubble slide. Hence the degree of each resulting monomial in bubbles is strictly less than $2k+1$. Since $d_{2k}$ has degree $2k+1$, the result follows.
\end{proof}

\begin{corollary} \label{coro-alpha-properties}
$\Endid$ is generated by $\{\alpha_{2k+1}\}_{k \geq 0}$ and these elements are algebraically independent.
\end{corollary}


\subsection{Diagrams as bimodule homomorphisms} \label{subsect-diagrams-as-bimodule-homs}

An action of $\Heis$ on the category $\FockSpaceCat$ whose objects are compositions of induction and restriction functors between $\mathbb{Z}/2\mathbb{Z}$-graded finite dimensional $\Sern{n}$-supermodules, for all $n\geq0$, is described in \cite[Section 6.3]{CS15}. Because induction and restriction functors for the algebras $\Sern{n}$ can be written as
\begin{equation*}
\ind^{\Sern{n+1}}_{\Sern{n}} ( -) =\; _{\Sern{n+1}} {\Sern{n+1}} \otimes_{\Sern{n}} - \quad\quad \text{and} \quad\quad \res^{\Sern{n+1}}_{\Sern{n}} ( - ) =\; _{\Sern{n}}\Sern{n+1}\otimes_{\Sern{n+1}} -
\end{equation*}
(where above $_{\Sern{n+1}}\Sern{n+1}$ is a ($\Sern{n+1},\Sern{n})$-bimodule and $_{\Sern{n}}\Sern{n+1}$ is a $(\Sern{n},\Sern{n+1})$-bimodule) the objects of $\FockSpaceCat$ can alternatively be described as tensor products of certain $(\Sern{k_1},\Sern{k_2})$-bimodules for all $k_1,k_2 \geq 0$. We will use this interpretation extensively below. Let $k_1, k_2 \leq n$, then we write

\begin{itemize}
	\item $(n)$ for $\Sern{n}$ considered as a $(\Sern{n},\Sern{n})$-bimodule,
	\item $(n)_{k_2}$ for $\Sern{n}$ considered as a $(\Sern{n},\Sern{k_2})$-bimodule,
	\item $_{k_1}(n)$ for $\Sern{n}$ considered as a $(\Sern{k_1},\Sern{n})$-bimodule,
	\item $_{k_1}(n)_{k_2}$ for $\Sern{n}$ considered as a $(\Sern{k_1},\Sern{k_2})$-bimodule.
\end{itemize}

The morphisms in $\FockSpaceCat$ are certain natural transformations of these compositions of induction/restriction functors (or, equivalently, certain bimodule homomorphisms). 
To describe the morphisms of $\FockSpaceCat$, we use the diagrams of $\Heis$, but this time we also label the rightmost region with a non-negative integer $n$.  Since an upward strand denotes the identity endomorphism of induction, it increases the label by one as one reads from right to left. Similarly a downward strand decreases the label by one. 

\begin{center}
		\begin{tikzpicture}
			\draw[->,thick] (1,0) to (1,1);
			\node at (1.5,0.5) {$n$};
			\node at (0,0.5) {$n+1$};
		\end{tikzpicture}\hspace{3cm}
			\begin{tikzpicture}
		\draw[->,thick] (1,1) to (1,0);
		\node at (2,0.5) {$n+1$};
		\node at (0.5,0.5) {$n$};
		\end{tikzpicture}
\end{center}

 We also set any  diagram with a negative label to be zero. Descriptions of other morphisms in $\FockSpaceCat$ are most easily given in terms of bimodules, so we henceforth use this language exclusively.

Below are the generating morphisms and the corresponding bimodule maps: 
 \begin{equation}
\begin{tikzpicture}[baseline=(current bounding box).center, scale=0.75]
\node at (-4,1.6){$ (n+1)_n \rightarrow (n+1)_n$};
\node at (-3.7,.9){$g \mapsto (-1)^{|g|}gc_{n+1}$};
\draw[<->] (-.7,1.3)--(1.5,1.3);
\draw[->,thick] (4,.7) to (4,2);
\draw (4,1.35) circle [radius = 3pt];
\node at (4.7,1.5){$n$};
\node at (3.2,1.5){$n+1$};
\end{tikzpicture}
\end{equation}

 \begin{equation}
\begin{tikzpicture}[baseline=(current bounding box).center, scale=0.75]
\node at (-4,1.6){$_{n-1}(n) \rightarrow _{n-1}(n)$};
\node at (-3.7,.9){$g \mapsto c_{n+1}g$};
\draw[<->] (-.7,1.3)--(1.5,1.3);
\draw[<-,thick] (4,.7) to (4,2);
\draw (4,1.35) circle [radius = 3pt];
\node at (4.7,1.5){$n$};
\node at (3.2,1.5){$n-1$};
\end{tikzpicture}
\end{equation}

 \begin{equation} \label{eqn-left-cup}
\begin{tikzpicture}[baseline=(current bounding box).center, scale=0.75]
\node at (-4,1.6){$q_{n}: (n) \rightarrow (n)_{n-1}(n)$};
\draw[<->] (-.7,1.3)--(1.5,1.3);
\draw[<-,thick] (2.5,1.8) arc(-180:0:9mm); 
\node at (4.7,1.5){$n$};
\node at (3.4,1.6){$n-1$};
\end{tikzpicture}
\end{equation}

\begin{equation}
\begin{tikzpicture}[baseline=(current bounding box).center, scale=0.75]
\node at (-4,1.6){$\pr{n}: \;_n(n+1)_n\rightarrow (n)$};
\draw[<->] (-.7,1.3)--(1.5,1.3);
\draw[<-,thick] (2.5,1) arc(180:0:9mm); 
\node at (4.7,1.5){$n$};
\node at (3.4,1.2){$n+1$};
\end{tikzpicture}
\end{equation}

\begin{equation} \label{eqn-right-cap}
\begin{tikzpicture}[baseline=(current bounding box).center, scale=0.75]
\node at (-4,1.6){$ (n+1)_n(n+1) \rightarrow (n+1)$};
\node at (-3.7,.9){$g_1 \otimes g_2 \mapsto g_1g_2$};
\draw[<->] (-.7,1.3)--(1.5,1.3);
\draw[->,thick] (2.5,1) arc(180:0:9mm); 
\node at (5,1.5){$n+1$};
\node at (3.4,1.2){$n$};
\end{tikzpicture}
\end{equation}

\begin{equation}
\begin{tikzpicture}[baseline=(current bounding box).center, scale=0.75]
\node at (-4,1.6){$i_{n}: (n) \rightarrow _n(n+1)_n$};
\node at (-3.7,.9){$g\mapsto g$};
\draw[<->] (-.7,1.3)--(1.5,1.3);
\draw[->,thick] (2.5,1.8) arc(-180:0:9mm); 
\node at (4.7,1.5){$n$};
\node at (3.4,1.6){$n+1$};
\end{tikzpicture}
\end{equation}

\begin{equation} \label{eqn-crossing}
\begin{tikzpicture}[baseline=(current bounding box).center, scale=0.75]
\node at (-4,1.6){$ (n+2)_{n} \rightarrow (n+2)_{n}$};
\node at (-3.7,.9){$g\mapsto gs_{n+1}$};
\draw[<->] (-.7,1.3)--(1.5,1.3);
\draw[->,thick] (2.5,0)--(4,2);
\draw[->,thick] (4,0)--(2.5,2);
\node at (4,1){$n$};
\node at (3.3,2.25){$n+1$};
\node at (2.4,1){$n+2$};
\end{tikzpicture}
\end{equation}

where $\pr{n}: \;_n(n+1)_n\rightarrow (n)$ is the projection map given by $\pr{n}(g)=g$ if $g\in \Sern{n}$, $\pr{n}(g)=0$ if $g\not\in \Sern{n}$, and $q_{n}: (n) \rightarrow (n)_{n-1}(n)$ is the bimodule map determined by $\ds q_n(1)=\sum_{x \in \LcosSer{n}{n-1}} x \otimes x^{-1}$.

\begin{remark}
The action of $\Heis$ on $\FockSpaceCat$ can be lifted to the idempotent closures of these categories. This then becomes a categorification of the Fock space representation \cite{CS15}. 
\end{remark}

Following Khovanov's approach from \cite{Kho14}, let $\FockSpaceCat[n]$ be the full subcategory of $\FockSpaceCat$ whose objects are $(\Sern{k},\Sern{n})$-bimodules, for all $k\in \mathbb{Z}_{\geq0}$. For every $n \in \mathbb{Z}_{\geq0}$ there is a functor $\FockSpaceFunctor{n}:\Heis \rightarrow \FockSpaceCat[n]$ which is defined on objects of $\Heis$ such that it sends the rightmost $P$ (respectively $Q$) from a $P$,$Q$ sequence to $(n+1)_n$  (resp. $_{n-1}(n)$) and all other $P$ and $Q$ in the sequence are determined by this choice. Hence $\FockSpaceFunctor{n}$ maps all objects of $\Heis$ into $\FockSpaceCat[n]$.  Under $\FockSpaceFunctor{n}$, a morphism (or diagram) is mapped to a morphism in $\FockSpaceCat[n]$ by labeling the rightmost region by $n$ which determines the labelings of all other regions. An upward strand increases the label by $1$, and a downward strand decreases the label by $1$ as we read from right to left.

Note that the image of a closed diagram $D$ under $\FockSpaceFunctor{n}$ will be an $(\Sern{n},\Sern{n})$-bimodule endomorphim of $\Sern{n}$ which we denote as $f: \Sern{n} \rightarrow \Sern{n}$. $f$ is fully determined by the value $f(1)$ since for any $x \in \Sern{n}$, $f(x) = xf(1)$. Furthermore, $f(1)$ is an element of $\evencenter{n}$ because $xf(1) = f(x) = f(1)x$. In this way we can identify the image of $\Endid$ under $\FockSpaceFunctor{n}$ with elements of $\evencenter{n}$.

We next study the image of some of the elements of $\Endid$ from Section \ref{Section-H-Center} under  the functor $\FockSpaceFunctor{n}$.

\begin{lemma} \label{lemma-bimodule-diagrams1}
\begin{enumerate}

\item \label{lemma-bimodule-diagrams1-cups} The diagram
\vspace{2mm}

\begin{center}
\begin{tikzpicture}[scale = .6]

\draw[thick] (.9,-.5) arc (180:360:.6cm);
\draw[thick] (0,-.5) arc (180:360:1.5cm);
\draw[thick] (-.85,-.5) arc (180:360:2.4cm);

\node at (.9,-.5) {\arrowlines};
\node at (0,-.5) {\arrowlines};
\node at (-.85,-.5) {\arrowlines};

\node at (1.5,-.7) {\scriptsize{$n$-$k$}};
\node at (-.3,-1) {\scriptsize{$n$-$1$}};
\node at (-1.1,-1.2) {\scriptsize{$n$}};

\draw[thick,black] (.7,-.8) circle (.2mm);
\draw[thick,black] (.55,-.85) circle (.2mm);
\draw[thick,black] (.4,-.9) circle (.2mm);

\end{tikzpicture}
\end{center}
\vspace{2mm}
corresponds to the $(\Sern{n},\Sern{n})$-bimodule homomorphism $(n) \rightarrow (n)_{n-k}(n)$ which sends
\begin{equation*}
1 \;\; \mapsto \sum_{x \in \LcosSer{n}{n-k}} x \otimes x^{-1}.
\end{equation*}

\item \label{lemma-bimodule-diagrams1-perm} For $\mu \in \oddpartitions[k]$ with $k \leq n$, the diagram
\vspace{2mm}
\begin{center}
\begin{tikzpicture}[scale = .8]

\draw[thick] (2.3,-1)--(2.3,1);
\draw[thick] (1.4,-1)--(1.4,1);

\draw[fill=white,thick] (2.6,-.3) rectangle (1.1,.3);

\node at (2.3,.6) {\arrowlines};
\node at (1.4,.6) {\arrowlines};

\draw[thick,black] (2,.7) circle (.2mm);
\draw[thick,black] (1.8,.7) circle (.2mm);
\draw[thick,black] (1.6,.7) circle (.2mm);

\draw[thick,black] (4.8,.7) circle (.2mm);
\draw[thick,black] (5,.7) circle (.2mm);
\draw[thick,black] (5.2,.7) circle (.2mm);

\draw[thick,black] (3.2,0) circle (.2mm);
\draw[thick,black] (3.4,0) circle (.2mm);
\draw[thick,black] (3.6,0) circle (.2mm);

\draw[thick] (4.5,-1)--(4.5,1);
\draw[thick] (5.4,-1)--(5.4,1);

\node[rotate = 180] at (4.5,.6) {\arrowlines};
\node[rotate = 180] at (5.4,.6) {\arrowlines};

\node at (1.85,0) {$\mu$};

\node at (3.4,.5) {\scriptsize{$n$-$k$}};
\node at (0,.3) {\scriptsize{$n$}};
\node at (6.8,.3) {\scriptsize{$n$}};

\end{tikzpicture}
\end{center}
\vspace{2mm}
corresponds to the $(\Sern{n},\Sern{n})$-bimodule homomorphism $(n)_{n-k}(n) \rightarrow (n)_{n-k}(n)$ which for $x,y \in \Sern{n}$ sends
\begin{equation*}
x \otimes y \;\; \mapsto \;\; x\distinguishedperm{\mu}{n} \otimes y.
\end{equation*}

\end{enumerate}
\end{lemma}

\begin{proof}
Both \ref{lemma-bimodule-diagrams1-cups} and \ref{lemma-bimodule-diagrams1-perm} follow from calculations using the definitions of cup \eqref{eqn-left-cup} and crossing \eqref{eqn-crossing} maps.
\end{proof}

\begin{proposition} \label{prop-perm-diagrams-mapping}
For $\mu \in \oddpartitions[k]$,
\begin{equation*}
\FockSpaceFunctor{n}(\alpha_{\mu}) = \begin{cases}
\Sergeevclasssum{\mu}{n} & \text{if $k \leq n$}\\
0 & \text{otherwise}.
\end{cases}
\end{equation*}
\end{proposition}

\begin{proof}
The diagram for $\alpha_\mu$ can be broken into three components
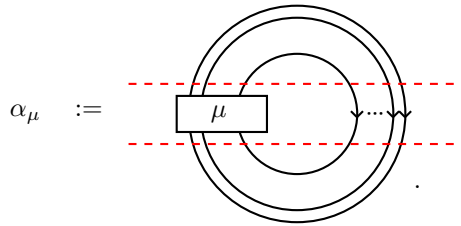
\begin{center}
\begin{tikzpicture}[scale = .8]

\node at (-4,0) {$\alpha_\mu \quad := $};

\draw[thick] (0,0) circle (1.8 cm);
\draw[thick] (0,0) circle (1.6 cm);
\draw[thick] (0,0) circle (1 cm);

\draw[fill=white,thick] (-2,-.3) rectangle (-.5,.3);

\node at (1.4,0) {$\cdot$};
\node at (1.3,0) {$\cdot$};
\node at (1.2,0) {$\cdot$};

\node at (-1.3,0) {$\mu$};

\node[rotate = 180] at (1.8,0) {\arrowlines};
\node[rotate = 180] at (1.6,0) {\arrowlines};
\node[rotate = 180] at (1,0) {\arrowlines};

\node at (2,-1.2) {.};

\draw[red,thick,dashed] (-2.8,.5)--(2.8,.5);
\draw[red,thick,dashed] (-2.8,-.5)--(2.8,-.5);

\end{tikzpicture}
\end{center}
Reading from bottom to top, the first component corresponds to Lemma \ref{lemma-bimodule-diagrams1}.\ref{lemma-bimodule-diagrams1-cups}, and the second corresponds to Lemma \ref{lemma-bimodule-diagrams1}.\ref{lemma-bimodule-diagrams1-perm} The composition of these two maps sends
\begin{equation*}
1 \;\; \mapsto \;\; \sum_{x \in \LcosSer{n}{n-k}} x\distinguishedperm{\mu}{n} \otimes x^{-1}.
\end{equation*}
The top component of $k$ nested caps is the multiplication map which sends
\begin{equation*}
\sum_{x \in \LcosSer{n}{n-k}} x\distinguishedperm{\mu}{n} \otimes x^{-1} \;\; \mapsto \;\; \sum_{x \in \LcosSer{n}{n-k}} x\distinguishedperm{\mu}{n}x^{-1} = \Sergeevclasssum{\mu}{n}. \qedhere
\end{equation*} 
\end{proof}

\begin{lemma} \cite{OR17} \label{lemma-curl-is-JM}
For $n -1 \geq 0$, the right twist curl 
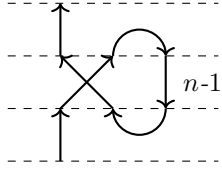
\vspace{2mm}
\begin{center}
    	\begin{tikzpicture}[scale=.7]
    		\draw[->,thick] (0,0) to (0,1);
    		\draw[->,thick] (0,1) to (1,2);
    		\draw[->,thick] (1,1) to (0,2);
    		\draw[<-,thick] (2,1) to (2,2);
    		\draw[->,thick] (2,1) arc(0:-180:0.5);
    		\draw[->,thick] (0,2) to (0,3);
    		\draw[->,thick] (1,2) arc(180:0:0.5);
    		\draw[dashed] (-1,0) to (3,0);
    		\draw[dashed] (-1,1) to (3,1);
    		\draw[dashed] (-1,2) to (3,2);
    		\draw[dashed] (-1,3) to (3,3);
    		\node at (2.7,1.5) {$n$-$1$};
    	\end{tikzpicture}
\end{center}
corresponds to the $(\Sern{n},\Sern{n-1})$-bimodule homomorphism, $(n)_{n-1} \rightarrow (n)_{n-1}$ which multiplies $x \in \Sern{n}$ on the right by the Jucys-Murphy element $\JM{n}$
\begin{equation*}
x \;\; \mapsto \;\; x\JM{n}.
\end{equation*}
\end{lemma} 

\begin{proposition} \label{prop-curl-generators-bimodules}
Let $k \geq 0$ and $n \geq 1$, then
\begin{enumerate}
\item $ \displaystyle \FockSpaceFunctor{n}(\bar{d}_{2k}) = \pr{n}(\JM{n+1}^{2k}), $
\item $\displaystyle \FockSpaceFunctor{n}(d_{2k}) = \sum_{x \in \LcosSer{n}{n-1}} x\JM{n}^{2k}x^{-1}.$
\end{enumerate}
\end{proposition}

\begin{proof}
These follow from direct calculation using the definitions of the cup \eqref{eqn-left-cup} and cap \eqref{eqn-right-cap} maps and Lemma \ref{lemma-curl-is-JM}.
\end{proof}

\section{Main results}\label{Section-main}

\subsection{An isomorphism between $\Endid$ and $\supersym$} 
In this subsection we establish an isomorphism between $\Endid$ and $\supersym$. The key step in the construction of this map will be identifying the elements of $\Endid$ with functions on $\strictpartitions$, i.e. as elements of $\funonYD$. To do this, let $\lambda \in \strictpartitions[n]$ and $x \in \Endid$. Then we evaluate $x$ on $\lambda$ by
\begin{equation*}
x(\lambda) := \normSerChar{\lambda}(\FockSpaceFunctor{n}(x)).
\end{equation*}
Because $\FockSpaceFunctor{n}$ is a homomorphism on $\Endid$ which maps into $\evencenter{n}$ and $\normSerChar{\lambda}$ is a homomorphism when restricted to $\evencenter{n}$, this defines a homomorphism into $\funonYD$. 

\begin{proposition} \label{prop-Fock-value-perm-cycle}
For $\mu \in \oddpartitions[k]$ and $\lambda \in \strictpartitions[n]$ we have 
\begin{equation*}
\alpha_{\mu}(\lambda) = \begin{cases}
2^kn^{\downarrow k}\frac{\SerChar{\lambda}(\mu \cup 1^{n-k})}{\SerChar{\lambda}(1^n)}   & \text{if $k \leq n$}\\
0 & \text{otherwise}
\end{cases}
\end{equation*}
\end{proposition}

\begin{proof}
This follows from Proposition \ref{prop-character-classsum} and Proposition \ref{prop-perm-diagrams-mapping}.
\end{proof}

\begin{theorem} \label{thm-main}
There is an algebra isomorphism $\primaryisom: \Endid \rightarrow \supersym$ which for any $\mu \in \oddpartitions$, sends
\begin{equation*}
\alpha_\mu \mapsto 2^{\length{\mu}} \shiftedpowersum_\mu.
\end{equation*}
\end{theorem}

\begin{proof}
It is clear from Proposition \ref{prop-Fock-value-perm-cycle} and Corollary \ref{cor-value-of-shiftedpower} that $2^{-\length{\mu}}\alpha_\mu$ and $\shiftedpowersum_\mu$ map to the same function in $\funonYD$. Furthermore the collection of functions which are the image of $\{\shiftedpowersum_{2k+1}\}_{k \geq 0}$ are algebraically independent by Proposition \ref{prop-supersym-embed} and Corollary \ref{coro-pfrak-is-gen/basis}. By Proposition \ref{coro-alpha-properties} $\Endid$ is generated by the algebraically independent elements $\{\alpha_{2k+1}\}_{k \geq 0}$. It then follows that the map that sends $\alpha_\mu \mapsto 2^{\length{\mu}}\shiftedpowersum_\mu$ is an isomorphism.
\end{proof}

Let $\mu \in \oddpartitions[n]$. It follows from Lemma \ref{lemma-size-of-orbit}, Remark \ref{remark-closure-don't-care-about-conj-class-rep}, and Theorem \ref{thm-main} that
\begin{equation} \label{eqn-image-of-cong-class}
\begin{tikzpicture}

\node at (0,0) {\begin{tikzpicture}[scale = .8]

\draw[thick] (0,0) circle (1.8 cm);
\draw[thick] (0,0) circle (1.6 cm);
\draw[thick] (0,0) circle (1 cm);

\draw[fill=white,thick] (-2,-.3) rectangle (-.5,.3);

\node at (-1,.5) {$\cdot$};
\node at (-1.15,.55) {$\cdot$};
\node at (-1.3,.6) {$\cdot$};

\node at (-1.3,0) {\small{\text{$C_\mu$}}};

\node[rotate = -30] at (-1.5,1) {\arrowlines};
\node[rotate = -30] at (-1.3,.9) {\arrowlines};
\node[rotate = -30] at (-.8,.6) {\arrowlines};
\end{tikzpicture}};

\node at (2.1,0) {$=$};

\node at (3.35,0) {$\frac{n!}{z_{\mu}}2^{n - \length{\mu}}$};

\node at (5.8,0) {\begin{tikzpicture}[scale = .8]

\draw[thick] (0,0) circle (1.8 cm);
\draw[thick] (0,0) circle (1.6 cm);
\draw[thick] (0,0) circle (1 cm);

\draw[fill=white,thick] (-2,-.3) rectangle (-.5,.3);

\node at (-1,.5) {$\cdot$};
\node at (-1.15,.55) {$\cdot$};
\node at (-1.3,.6) {$\cdot$};

\node at (-1.3,0) {\small{$\mu$}};

\node[rotate = -30] at (-1.5,1) {\arrowlines};
\node[rotate = -30] at (-1.3,.9) {\arrowlines};
\node[rotate = -30] at (-.8,.6) {\arrowlines};
\end{tikzpicture}};

\node at (8.2,.3) {$\primaryiso$};
\draw[thick, ->] (7.6,0)--(8.8,0);
\draw[thick] (7.6,.1)--(7.6,-.1);

\node at (10,0) {$\frac{n!}{z_{\mu}}2^{n}\shiftedpowersum_\mu.$};

\end{tikzpicture}
\end{equation}

\begin{theorem}
Let $\lambda \in \strictpartitions[n]$. Under the isomorphism $\primaryiso: \Endid \rightarrow \supersym$, the closure of the central idempotent $\centralidem{\lambda}$ of $\Sern{n}$ maps to $\pathsfromzero{\lambda} Q_\lambda^{*}$.
\end{theorem}

\begin{proof}
Recall from Lemma \ref{lemma-central-idempotent} that
\begin{equation*}
\centralidem{\lambda} = 2^{\frac{-\length{\lambda}-\lengthparity(\lambda) }{2}}\frac{\stanshiftedstrict{\lambda}}{n!} \sum_{\mu \in \oddpartitions[n]} \SerChar{\lambda}(\mu)C_\mu
\end{equation*}
while by \eqref{eqn-P-in-terms-of-frakp}
\begin{equation*}
Q^*_\lambda =2^{\frac{\length{\lambda}-\lengthparity(\lambda)}{2}} \sum_{\mu \in \oddpartitions[n]}  \frac{\SerChar{\lambda}(\mu)}{z_\mu}\shiftedpowersum_\mu.
\end{equation*}
Combining these facts with Theorem \ref{thm-main} and \eqref{eqn-image-of-cong-class} it follows that the closure of $\centralidem{\lambda}$ is equal to $2^{n-\length{\lambda}}\stanshiftedstrict{\lambda}Q^*_\lambda = \pathsfromzero{\lambda}Q^*_\lambda$.
\end{proof}

\begin{remark}
Recall that the Schur $Q$-functions are related to the Schur $P$-functions by $P_\lambda = 2^{-\length{\lambda}}Q_\lambda$. Ivanov also studied factorial Schur $P$-functions $\{P^*_\lambda\}_{\lambda \in \strictpartitions[n]}$ where $P^*_\lambda = 2^{-\length{\lambda}}Q^*_\lambda$ \cite{Iv01}. Then one alternative description of the closure of $\centralidem{\lambda}$ in $\supersym$ is as $2^{n}\stanshiftedstrict{\lambda}P^*_\lambda $.
\end{remark}
   
Moving in the opposite direction, we can also identify the elements of $\supersym$ corresponding to the generators $\{d_{2k}\}_{k \geq 0}$ and $\{\bar{d}_{2k}\}_{k \geq 0}$.

\begin{theorem} \label{thm-identity-of-bubbles}
For $k \geq 0$, 
\begin{enumerate}
	\item $\psi(\bar{d}_{2k}) = \upmoment{k}(\cdot),$
	\item $\psi(d_{2k}) = \downmoment{k+1}(\cdot).$
\end{enumerate}
\end{theorem}

\begin{proof}
This follows from Proposition \ref{proposition-algebraic-interp-moments} and Proposition \ref{prop-curl-generators-bimodules}.
\end{proof}

\begin{remark}
In light of Theorem \ref{thm-identity-of-bubbles}, Proposition \ref{prop-recursive} can be seen as a diagrammatic manifestation of Proposition \ref{prop-recursive-relation-petrov}.
\end{remark}

\begin{table}[htp]\label{dictionary-table}
\begin{center}
{\tabulinesep=1.2mm
 \begin{tabular}{| m{2.4cm} ||  m{3.7cm} | m{3.7cm} | m{2cm} | m{2cm} |}
 \hline
$ \Gamma$ & $\shiftedpowersum_{\mu}$ & $Q^*_\lambda$ & $\upmoment{k}$ & $\downmoment{k+1}$ \\  \hline {\bf Diagram in} $\Endid$ &
  
 \begin{tikzpicture}[scale = .6, baseline=(current bounding box.center)]

\node at (-2.5,0) {};
\node at (0,2) {};
\draw[thick] (0,0) circle (1.85 cm);
\draw[thick] (0,0) circle (1.6 cm);
\draw[thick] (0,0) circle (1 cm);

\draw[fill=white,thick] (-2,-.3) rectangle (-.5,.3);

\node at (-1,.5) {$\cdot$};
\node at (-1.15,.55) {$\cdot$};
\node at (-1.3,.6) {$\cdot$};

\node at (-3,-.4) {$2^{\length{\mu}}$};
\node at (-3,.3) {$1$};
\draw[thick] (-3.8,0) -- (-2.3,0);

\node at (-1.3,-0.1) {$\mu$};

\node[rotate = 180] at (1,0) {\arrowlines};
\node[rotate = 180] at (1.6,0) {\arrowlines};
\node[rotate = 180] at (1.85,0) {\arrowlines};

\node at (-3.6,-2) {};

\end{tikzpicture}  &

  \begin{tikzpicture}[scale = .6, baseline=(current bounding box.center)]

\node at (-2.5,0) {};
\node at (0,2) {};
\draw[thick] (0,0) circle (1.85 cm);
\draw[thick] (0,0) circle (1.6 cm);
\draw[thick] (0,0) circle (1 cm);
\node at (0,-1.7) {};

\draw[fill=white,thick] (-2,-.3) rectangle (-.5,.3);

\node at (-1,.5) {$\cdot$};
\node at (-1.15,.55) {$\cdot$};
\node at (-1.3,.6) {$\cdot$};

\node at (-3,-.4) {${h(\lambda)}$};
\node at (-3,.3) {$1$};

\node at (-1.3,0) {\small{\text{$\centralidem{\lambda}$}}};

\draw[thick] (-3.8,0) -- (-2.3,0);

\node[rotate = 180] at (1,0) {\arrowlines};
\node[rotate = 180] at (1.6,0) {\arrowlines};
\node[rotate = 180] at (1.85,0) {\arrowlines};

\end{tikzpicture}
   &
    
  \cktildepicture  &
  
  \ckpicture  \\
\hline
\end{tabular}}
\end{center}
\caption{\label{i-class-table} A dictionary between $\supersym$ and diagrams in $\Endid$.}
\label{node-classification}
\end{table}


\subsection{An action of $\TrHEv$ on $\supersym$} \label{sect-Wminus-action}\label{W-algebra}

Aside from taking the Grothendieck group or center, another method for decategorifing a category $\MC{C}$ is taking the categorical trace of $\MC{C}$, $\Tr(\MC{C})$ (also known as the zeroth Hochschild homology of $\MC{C}$). See \cite{BGHL14} for a discussion of this method of decategorification. In \cite{OR17}, it is shown that the even part of the trace of $\Heis$, $\TrHEv$, is isomorphic to the vertex algebra $\Wminus$ at level one, a subalgebra of $\Winfty$ defined by Kac, Wang, and Yan \cite{KWY98}. 

In a diagrammatic setting such as this, the trace can be realized as the algebra of closed diagrams on an annulus. There is a natural action of $\Tr(\MC{C})$ on the center of the category $\MC{C}$, $\End_{\MC{C}}(\UnitModule)$, where diagrammatically a closed diagram on an annulus acts on a closed diagram in a disk by plugging the annulus with the disk, resulting in a new diagram in the disk. The results of \cite{OR17} along with Theorem \ref{thm-main} imply that $\Wminus$ acts on $\supersym$. This action is similar to the action of $\Winfty$ on the centers of symmetric group algebras described in \cite{LT01}. In this section we will first review $\Wminus$ and then describe the action of the generators of $\Wminus$ on basis elements of $\supersym$.

We first review the vertex algebra $W^-$, which appears in the trace of $\mathcal{H}_{tw}$. 

Let $\hat{\mathcal{D}}^-$ be the Lie algebra over the vector space spanned by $\{C\}\cup \{ t^{2k-1} g(D+(2k-1)/2); \ g\text{ even}\} \cup \{ t^{2k} f(D+k) ;\ f\text{ odd}\}$ where $k\in\mathbb{Z}$ and even and odd refer to even and odd polynomial functions. The Lie bracket in $\hat{\mathcal{D}}^-$ is defined by

\begin{equation}\label{WForm}
[t^rf(D),t^sg(D)]=t^{r+s}(f(D+s)g(D)-f(D)g(D+r))+\psi(t^rf(D),t^sg(D))C,
\end{equation}
where
\begin{equation}
\psi(t^rf(D),t^sg(D))= \begin{cases} 
      \ds\sum_{-r\leq j\leq -1}f(j)g(j+r) & r=-s\geq 0 \\
      0 & r+s\neq 0 
   \end{cases},
\end{equation}
and $C$ is a central element. Denote by $W^-$ the universal enveloping algebra of $\hat{\mathcal{D}}^-$. The trace of $\mathcal{H}_{tw}$ was shown in \cite{OR17} to be isomorphic to the quotient $W^-/\langle \omega_{0,0}, C-1 \rangle$. 


A generating set for $W^-/\langle \omega_{0,0}, C-1\rangle$ is given by $\omega_{1,0}$, $\omega_{0,3}$, and $\omega_{\pm 2,1}  \pm \omega_{\pm 2,0}$\cite[Lemma 2.2]{OR17}. In order to explicitly write down an action of the algebra $W^-$ on $\supersym$, we will work with different generating set.

\begin{proposition}
The algebra $W^-/\langle \omega_{0,0}, C-1\rangle$ is also generated by $\omega_{1,0}$, $\omega_{-1,0}$ and $\omega_{0,3}$.
\end{proposition}

\begin{proof}
    We will show that we can obtain the aforementioned generators $\omega_{\pm 2,1}  \pm \omega_{\pm 2,0}$ via the elements $\omega_{1,0}$, $\omega_{-1,0}$ and $\omega_{0,3}$. It is a straightforward computation that 
   \begin{equation*}
   \omega_{0,1}=-\frac{1}{20}[[\omega_{0,3}, \omega_{-1,0}], \omega_{1,0}] +\frac{1}{5}\omega_{-1,0}\omega_{1,0}
   \end{equation*}
   and using $\omega_{0,1}$, we can obtain $\omega_{-1,2}-\omega_{-1,1}$ as follows:
   \begin{equation*}
   \omega_{-1,2}-\omega_{-1,1}=\frac{1}{6}[\omega_{0,3}, \omega_{-1,0}] + \frac{1}{3}\omega_{-1,0}\omega_{0,1}.
   \end{equation*}
Then one of the elements we are looking for is given by
   \begin{equation*}
   \omega_{-2,1} - \omega_{-2,0}=\frac{1}{2}[\omega_{-1,2}-\omega_{-1,1}, \omega_{-1,0}].
   \end{equation*}
    To obtain $\omega_{2,1} + \omega_{2,0}$, we follow a very similar computation:
    \begin{equation*}
   \omega_{1,2}+\omega_{1,1}=-\frac{1}{6}[\omega_{0,3}, \omega_{1,0}] + \frac{1}{3}\omega_{0,1}\omega_{1,0},
    \end{equation*}
    and finally
    \begin{equation*}
    \omega_{2,1} + \omega_{2,0}=-\frac{1}{2}[\omega_{1,2}+\omega_{1,1}, \omega_{1,0}]. \qedhere
    \end{equation*}
\end{proof}

The images of these generators under the isomorphism $W^-/\langle \omega_{0,0}, C-1\rangle \rightarrow \TrHEv$ from \cite{OR17} are given below. In the following diagrams, an asterisk denotes a puncture in the plane, forming the annulus around which the diagrams are closed in the trace:

\begin{equation*}
	\sqrt{2}\omega_{-1,0} \longmapsto \trup;
	\end{equation*}
	
	\begin{equation*}
	    \sqrt{2}\omega_{1,0} \longmapsto \trdown; 
	\end{equation*}
	
	\begin{equation*}
				-2\omega_{0,3} \longmapsto 
				\begin{tikzpicture}[baseline=(current bounding box).center]
				\draw[thick] (2,2) circle (.5cm);
				\node at (2,2) {$*$};
				\node at (1.5,2) {\arrowlines};
				\draw[fill=black] (1.62,2.33) circle (.08cm);
				\node at (1.4,2.5) {$\small{2}$};
				\end{tikzpicture}
				=
				\begin{tikzpicture}[baseline=(current bounding box.center), scale=.25]
				\draw[thick] (2,0) .. controls (2,1.25) and (0,.25) .. (0,2);
				\draw[thick] (0,0) .. controls (0,1) and (.8,.8) .. (1,2);
				\draw[thick] (1,0) .. controls (1,1) and (1.8,.8) .. (2,2);
				\node at (0,2){\arrowlines};
				\node at (1,2){\arrowlines};
				\node at (2,2){\arrowlines};
				\node at (4,1){$*$};
				\draw[thick] (2,2) arc (150:-150:2cm);
				\draw[thick] (1,2) arc (160.5:-160.5:3cm);
				\draw[thick] (0,2) arc (165.5:-165.5:4cm);
				\end{tikzpicture}
				+
				\begin{tikzpicture}[baseline=(current bounding box.center), scale=.25]
				\draw[thick] (1,0) to (1,2);
				\draw[thick] (2,0) to (2,2); 
				\node at (1,2){\arrowlines};
				\node at (2,2){\arrowlines};
				\draw[thick] (2,2) arc (150:-150:2cm);
				\draw[thick] (1,2) arc (160.5:-160.5:3cm);
				\node at (4,1){$*$};
				\end{tikzpicture}.
		\end{equation*}
		
		We will also use the elements $\omega_{-(2n+1),0}$ and their images in $\TrHEv$:
		
		\begin{equation}{\label{omega2n+1}}
		\sqrt{2}\omega_{-(2n+1),0} \longmapsto \begin{tikzpicture}[baseline=(current bounding box).center, scale = .8]
			
			\draw[thick] (0,0) circle (1.8 cm);
			\draw[thick] (0,0) circle (1.6 cm);
			\draw[thick] (0,0) circle (1 cm);
			
			\node at (0,0){*};
			
			\draw[fill=white,thick] (-2,-.3) rectangle (-.5,.3);
			
			\node at (-1,.5) {$\cdot$};
			\node at (-1.15,.55) {$\cdot$};
			\node at (-1.3,.6) {$\cdot$};
			
			\node at (-1.3,0) {$\tau$};
			
			\node[rotate = -30] at (-1.5,1) {\arrowlines};
			\node[rotate = -30] at (-1.3,.9) {\arrowlines};
			\node[rotate = -30] at (-.8,.6) {\arrowlines};
			\end{tikzpicture}
		\end{equation}
			where $\tau$ is a $(2n+1)$-cycle.

   We now describe the action of the generating set $\{\omega_{1,0}, \omega_{-1,0}, \omega_{0,3}\}$ of $W^-$ on the vector space basis $\{\shiftedpowersum_\mu\}_{\mu\in \oddpartitions}$ of $\supersym$. We achieve this by describing the action of the corresponding generators of $\TrHEv$ on the basis $\{\alpha_{\mu}\}_{\mu \in \oddpartitions}$ of $\Hcenter$. 

\begin{lemma}{\label{alphamu1}} We have
	$$\alpha_{(\mu,1)}= \alpha_{\mu}\alpha_1 -2|\mu|\alpha_{\mu}.$$
\end{lemma}

\begin{proof}
	  This simply follows from the local bubble sliding relation 
	  
	  \begin{equation*}
	  \begin{tikzpicture}[baseline=(current bounding box.center),scale=.6]
	  \draw[thick] (0,0) circle(1cm);
	  \node at (-1,0){$\arrowlines$};
	  \draw[thick] (1.5,1)--(1.5,-1);
	  \node[rotate=180] at (1.5,0){$\arrowlines$};
	  \end{tikzpicture}
	  =
	  \begin{tikzpicture}[baseline=(current bounding box.center),scale=.6]
	  \draw[thick] (0,0) circle(1cm);
	  \node at (-1,0){$\arrowlines$};
	  \draw[thick] (-1.5,1)--(-1.5,-1);
	  \node[rotate=180] at (-1.5,0){$\arrowlines$};
	  \end{tikzpicture}
	  -2
	  \begin{tikzpicture}[baseline=(current bounding box.center),scale=.6]
	  \draw[thick] (-1.5,1)--(-1.5,-1);
	  \node[rotate=180] at (-1.5,0){$\arrowlines$};
	  \end{tikzpicture}
	  \end{equation*}
	  applied $|\mu|$ times to the diagram $\alpha_{(\mu,1)}$, as we pull the clockwise bubble $\alpha_1$ from within $\alpha_{\mu}$.
\end{proof}

\begin{lemma}{\label{omega10}} We have
	$$\omega_{1,0}\cdot \alpha_\mu\alpha_1 =(\alpha_1+2)\omega_{1,0}\cdot \alpha_\mu .$$
\end{lemma}

\begin{proof}
	We compute: \begin{align*}
	\begin{tikzpicture}[baseline=(current bounding box.center),scale=.6]
	\draw[thick] (0,0) circle (2cm);
	\node[rotate=180] at (-2,0){$\arrowlines$};
	\draw[thick] (-1.5,-.5) rectangle (-.5,.5);
	\node at (-1,0){$\alpha_\mu$};
	\draw[thick] (-1,.5) arc (150:0:.5cm);
	\draw[thick] (-.07,.3) -- (-.07,-.3);
	\draw[thick] (-1,-.5) arc (-150:0:.5cm);
	\node[rotate=-90] at (-.5,.75){$\arrowlines$};
	\draw[thick] (1,0) circle(.7cm);
	\node at (.3,0){$\arrowlines$};
	\end{tikzpicture}
	&=
	\begin{tikzpicture}[baseline=(current bounding box.center),scale=.6]
	\draw[thick] (0,0) circle (2cm);
	\node[rotate=180] at (-2,0){$\arrowlines$};
	\node at (2,0){$\arrowlines$};
	\draw[thick] (-1.5,-.5) rectangle (-.5,.5);
	\node at (-1,0){$\alpha_\mu$};
	\draw[thick] (-1,.5) arc (150:0:.5cm);
	\draw[thick] (-.07,.3) -- (-.07,-.3);
	\draw[thick] (-1,-.5) arc (-150:0:.5cm);
	\node[rotate=-90] at (-.5,.75){$\arrowlines$};
	\draw[thick] (1.7,0) circle(.7cm);
	\node at (1,0){$\arrowlines$};
	\end{tikzpicture}
	+
	\begin{tikzpicture}[baseline=(current bounding box.center),scale=.6]
	\draw[thick] (-2,0) arc (180:20:2cm);
	\draw[thick] (-2,0) arc (-180:-20:2cm);
	\node[rotate=180] at (-2,0){$\arrowlines$};
	\draw[thick] (-1.5,-.5) rectangle (-.5,.5);
	\node at (-1,0){$\alpha_\mu$};
	\draw[thick] (-1,.5) arc (150:0:.5cm);
	\draw[thick] (-.07,.3) -- (-.07,-.3);
	\draw[thick] (-1,-.5) arc (-150:0:.5cm);
	\node[rotate=-90] at (-.5,.75){$\arrowlines$};
	\node at (.33,0){$\arrowlines$};
	\draw[thick,red] (1.87,.7) to[out=-60,in=-60] (1.45,.65);
	\draw[thick] (1.45,.65) arc (60:300:.75cm);
	\draw[thick,red] (1.45,-.65) to[out=60, in=60] (1.87,-.7);
	\end{tikzpicture}
	+
	\begin{tikzpicture}[baseline=(current bounding box.center),scale=.6]
	\draw[thick] (-2,0) arc (180:20:2cm);
	\draw[thick] (-2,0) arc (-180:-20:2cm);
	\node[rotate=180] at (-2,0){$\arrowlines$};
	\draw[thick] (-1.5,-.5) rectangle (-.5,.5);
	\node at (-1,0){$\alpha_\mu$};
	\draw[thick] (-1,.5) arc (150:0:.5cm);
	\draw[thick] (-.07,.3) -- (-.07,-.3);
	\draw[thick] (-1,-.5) arc (-150:0:.5cm);
	\node[rotate=-90] at (-.5,.75){$\arrowlines$};
	\node at (.33,0){$\arrowlines$};
	\draw[thick,red] (1.87,.7) to[out=-60,in=-60] (1.45,.65);
	\draw[thick] (1.45,.65) arc (60:300:.75cm);
	\draw[thick,red] (1.45,-.65) to[out=60, in=60] (1.87,-.7);
	\draw[thick,red] (1.6,.57) circle (4pt);
	\draw[thick,red] (1.8,-.57) circle (4pt);
	\end{tikzpicture}\\
	&=
	\begin{tikzpicture}[baseline=(current bounding box.center),scale=.6]
	\draw[thick] (0,0) circle (2cm);
	\node[rotate=180] at (-2,0){$\arrowlines$};
	\draw[thick] (-.7,-.5) rectangle (0.3,.5);
	\node at (-.2,0){$\alpha_\mu$};
	\draw[thick] (-.2,.5) arc (150:0:.5cm);
	\draw[thick] (.73,.3) -- (.73,-.3);
	\draw[thick] (-.2,-.5) arc (-150:0:.5cm);
	\node[rotate=-90] at (.3,.75){$\arrowlines$};
	\draw[thick] (3,0) circle(.7cm);
	\node at (2.3,0){$\arrowlines$};
	\end{tikzpicture}
	+2
	\begin{tikzpicture}[baseline=(current bounding box.center),scale=.6]
	\draw[thick] (0,0) circle (2cm);
	\node[rotate=180] at (-2,0){$\arrowlines$};
	\draw[thick] (-.7,-.5) rectangle (0.3,.5);
	\node at (-.2,0){$\alpha_\mu$};
	\draw[thick] (-.2,.5) arc (150:0:.5cm);
	\draw[thick] (.73,.3) -- (.73,-.3);
	\draw[thick] (-.2,-.5) arc (-150:0:.5cm);
	\node[rotate=-90] at (.3,.75){$\arrowlines$};
	\end{tikzpicture},
	\end{align*}
	as desired.
\end{proof}

\begin{theorem}
The generators $\TrHEv$ act on the basis elements $\{ \shiftedpowersum_\mu\}_{\mu \in \oddpartitions}$ of $\supersym$ as follows:
\begin{enumerate} 
\item $\omega_{-1,0}\cdot \shiftedpowersum_\mu = \sqrt{2}\shiftedpowersum_{(\mu,1)}$ 
 \item $\omega_{1,0}\cdot \shiftedpowersum_\mu = 
\frac{1}{\sqrt{2}} \shiftedpowersum_\mu+ \frac{k}{\sqrt{2}} \shiftedpowersum_{\hat{\mu}}$
\item $\omega_{0,3}\cdot \shiftedpowersum_\mu = -\shiftedpowersum_3 \shiftedpowersum_\mu -2 \shiftedpowersum_{(1,1)} \shiftedpowersum_\mu.$
\end{enumerate}
where $k$ is the number of parts of size 1 of $\mu$ and $\hat{\mu}$ stands for the partition obtained by removing one part of size 1 from $\mu$ if this is possible. When $\mu = (1)$ then $\shiftedpowersum_{\wh{(1)}}=1$.
\end{theorem}

\begin{proof}
	For the action of $\omega_{-1,0}$, note that the action of $\trup$ on $\alpha_{\mu}$ is diagrammatically just enclosing the diagram of $\alpha_{\mu}$ by a clockwise oriented strand:
	\begin{equation} \label{eqn-enclosing-bubble}
	\trup \cdot \alpha_{\mu} = 
	\begin{tikzpicture}[baseline=(current bounding box).center]
	\draw[thick] (2,2) circle (.5cm);
	\node at (1.5,2) {\arrowlines};
	\node at (2,2) {$\alpha_{\mu}$};
	\end{tikzpicture}
	\end{equation}
	and the resulting diagram is the diagram of $\alpha_{(\mu,1)}$. Replacing $\alpha_\mu$ by $2^{\length{\mu}}\shiftedpowersum_{\mu}$ and the clockwise bubble by $\sqrt{2}\omega_{-1,0}$, we get 
$
\omega_{-1,0}\cdot \shiftedpowersum_\mu = \sqrt{2}\shiftedpowersum_{(\mu,1)}.
$

We also know that $\omega_{-(2n+1),0}\cdot \shiftedpowersum_{\mu}= \sqrt{2}\shiftedpowersum_{(\mu,2n+1)}$ from \eqref{omega2n+1}. Note that we have
$[\omega_{-1,0}, \omega_{1,0}]=-1$ and, for $n\geq 0$, we have
	$[\omega_{-(2n+1),0}, \omega_{1,0}]=0 $.
	\omitt{
	\begin{equation*}
	\trdown \cdot \alpha_{\rho} = 
	\begin{tikzpicture}[baseline=(current bounding box).center]
	\draw[thick] (2,2) circle (.5cm);
	\node[rotate=180] at (1.5,2) {\arrowlines};
	\node at (2,2) {$\alpha_{\rho}$};
	\end{tikzpicture}
	\end{equation*}
	} 
	
	To simplify the notation in the following computations, denote $\op:=\sqrt{2}\omega_{1,0}$.	

		We start by showing that if the partition $\mu$ doesn't contain any parts of size one, then $\op \cdot \alpha_\mu=\alpha_{\mu}$ by induction on $\length{\mu}$. We provide a diagrammatic proof for the base case $\length{\mu}=1$ (i.e. $\alpha_\mu=\alpha_k$ for $k\neq1$ odd).
		
		In the diagram \begin{tikzpicture}[baseline=(current bounding box).center]
		\draw[thick] (2,2) circle (.5cm);
		\node[rotate=180] at (1.5,2) {\arrowlines};
		\node at (2,2) {$\alpha_{k}$};
		\end{tikzpicture}, we claim that we can pass $\alpha_{k}$ through the outer strand for free, meaning that all the resolution terms that appear as a result of relation \ref{eqn-symmetric-group-relations} are zero.
		
		We provide the computation for the case of $\alpha_{k} = \alpha_{5}$, and explain how the arguments generalize to any $\alpha_{k}$. We have
		
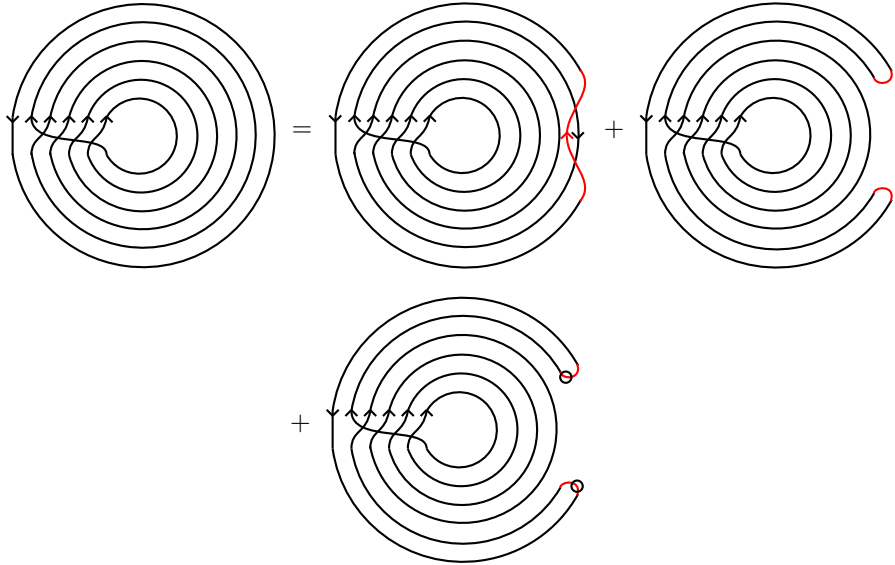
\begin{align*}
		\begin{tikzpicture}[baseline=(current bounding box.center), scale=.25]
		\draw[thick] (-2,0) .. controls (-2,1) and (-1.2,.8) .. (-1,2);
		\draw[thick] (-1,0) .. controls (-1,1) and (-.2,.8) .. (0,2);
		\draw[thick] (0,0) .. controls (0,1) and (.8,.8) .. (1,2);
		\draw[thick] (1,0) .. controls (1,1) and (1.8,.8) .. (2,2);
		\draw[thick] (2,0) .. controls (2,1.25) and (-2,.25) .. (-2,2);
		\draw[thick] (-3,0)--(-3,2);
		\node at (2,1.9){\arrowlines};
		\node at (1,1.9){\arrowlines};
		\node at (0,1.9){\arrowlines};
		\node at (-1,1.9){\arrowlines};
		\node at (-2,1.9){\arrowlines};
		\node[rotate=-180] at (-3,1.9){\arrowlines};
		\draw[thick] (2,2) arc (150:-150:2cm);
		\draw[thick] (1,2) arc (160.5:-160.5:3cm);
		\draw[thick] (0,2) arc (165.5:-165.5:4cm);
		\draw[thick] (-1,2) arc (169:-169.5:5cm);
		\draw[thick] (-2,2) arc (171:-171:6cm);
		\draw[thick] (-3,2) arc (172:-172:7cm);
		\end{tikzpicture}
		&=
		\begin{tikzpicture}[baseline=(current bounding box.center), scale=.25]
		\draw[thick] (-2,0) .. controls (-2,1) and (-1.2,.8) .. (-1,2);
		\draw[thick] (-1,0) .. controls (-1,1) and (-.2,.8) .. (0,2);
		\draw[thick] (0,0) .. controls (0,1) and (.8,.8) .. (1,2);
		\draw[thick] (1,0) .. controls (1,1) and (1.8,.8) .. (2,2);
		\draw[thick] (2,0) .. controls (2,1.25) and (-2,.25) .. (-2,2);
		\draw[thick] (-3,0)--(-3,2);
		\node at (2,1.9){\arrowlines};
		\node at (1,1.9){\arrowlines};
		\node at (0,1.9){\arrowlines};
		\node at (-1,1.9){\arrowlines};
		\node at (-2,1.9){\arrowlines};
		\node[rotate=-180] at (-3,1.9){\arrowlines};
		\node[rotate=-180] at (9.9,1){\arrowlines};
		\node at (9.3,1){\redarrowlines};
		\draw[thick] (2,2) arc (150:-150:2cm);
		\draw[thick] (1,2) arc (160.5:-160.5:3cm);
		\draw[thick] (0,2) arc (165.5:-165.5:4cm);
		\draw[thick] (-1,2) arc (169:-169.5:5cm);
		\draw[thick] (-2,2) arc (171:-171:6cm);
		\draw[thick] (-3,2) arc (172:29:7cm);
		\draw[thick] (-3,0) arc (-172:-29:7cm);
		\draw[thick, red] (10,4.5) to [out=-50, in=90] (9.3,1);
		\draw[thick, red] (10,-2.5) to [out=50, in=-90] (9.3,1);
		\end{tikzpicture}
		+
		\begin{tikzpicture}[baseline=(current bounding box.center), scale=.25]
		\draw[thick] (2,0) .. controls (2,1.25) and (-2,.25) .. (-2,2);
		\draw[thick] (1,0) .. controls (1,1) and (1.8,.8) .. (2,2);
		\draw[thick] (0,0) .. controls (0,1) and (.8,.8) .. (1,2);
		\draw[thick] (-1,0) .. controls (-1,1) and (-.2,.8) .. (0,2);
		\draw[thick] (-2,0) .. controls (-2,1) and (-1.2,.8) .. (-1,2);
		\draw[thick] (-3,0)--(-3,2);
		\node at (2,1.9){\arrowlines};
		\node at (1,1.9){\arrowlines};
		\node at (0,1.9){\arrowlines};
		\node at (-1,1.9){\arrowlines};
		\node at (-2,1.9){\arrowlines};
		\node[rotate=-180] at (-3,1.9){\arrowlines};
		\draw[thick] (2,2) arc (150:-150:2cm);
		\draw[thick] (1,2) arc (160.5:-160.5:3cm);
		\draw[thick] (0,2) arc (165.5:-165.5:4cm);
		\draw[thick] (-1,2) arc (169:-169.5:5cm);
		\draw[thick] (-2,2) arc (171:30:6cm);
		\draw[thick] (-2,0) arc (-171:-30:6cm);
		\draw[thick] (-3,2) arc (172:29:7cm);
		\draw[thick] (-3,0) arc (-172:-29:7cm);
		\draw[thick,red] (10,4.5) .. controls (10.3,3.7) and (9.2,3.6) .. (9.1,4.1);
		\draw[thick,red] (9.1,-2.1) .. controls (9.1,-1.7) and (10.3,-1.6) .. (10,-2.5);
		\end{tikzpicture}\\
		&+
		\begin{tikzpicture}[baseline=(current bounding box.center), scale=.25]
		\draw[thick] (2,0) .. controls (2,1.25) and (-2,.25) .. (-2,2);
		\draw[thick] (1,0) .. controls (1,1) and (1.8,.8) .. (2,2);
		\draw[thick] (0,0) .. controls (0,1) and (.8,.8) .. (1,2);
		\draw[thick] (-1,0) .. controls (-1,1) and (-.2,.8) .. (0,2);
		\draw[thick] (-2,0) .. controls (-2,1) and (-1.2,.8) .. (-1,2);
		\draw[thick] (-3,0)--(-3,2);
		\node at (2,1.9){\arrowlines};
		\node at (1,1.9){\arrowlines};
		\node at (0,1.9){\arrowlines};
		\node at (-1,1.9){\arrowlines};
		\node at (-2,1.9){\arrowlines};
		\node[rotate=-180] at (-3,1.9){\arrowlines};
		\draw[thick] (2,2) arc (150:-150:2cm);
		\draw[thick] (1,2) arc (160.5:-160.5:3cm);
		\draw[thick] (0,2) arc (165.5:-165.5:4cm);
		\draw[thick] (-1,2) arc (169:-169.5:5cm);
		\draw[thick] (-2,2) arc (171:30:6cm);
		\draw[thick] (-2,0) arc (-171:-30:6cm);
		\draw[thick] (-3,2) arc (172:29:7cm);
		\draw[thick] (-3,0) arc (-172:-29:7cm);
		\draw[thick,red] (10,4.5) .. controls (10.3,3.7) and (9.2,3.6) .. (9.1,4.1);
		\draw[thick,red] (9.1,-2.1) .. controls (9.1,-1.7) and (10.3,-1.6) .. (10,-2.5);
		\draw[thick] (9.4,3.8) circle (.3cm);
		\draw[thick] (10,-2) circle (.3cm);
		\end{tikzpicture}
		\end{align*}
		and the two hollow dots appearing in the last term cancel with each other if we slide them along the outermost strand. This observation will hold for the rest of the computation, so we will omit drawing the second resolution term and instead write the first resolution term with coefficient 2. We will show that all resolution terms coming from crossings on the: outermost strand, innermost strand, and intermediate strands are zero. 
		
		For the resolution term coming from the crossing of outermost strands, we have 
		
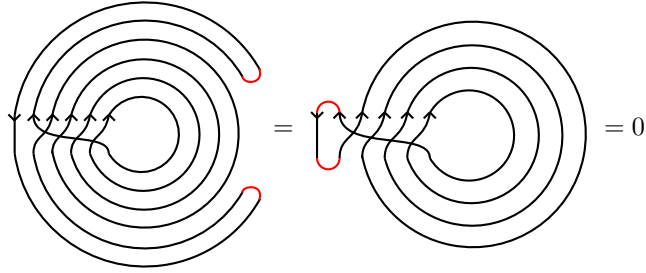
\begin{align*}
		\begin{tikzpicture}[baseline=(current bounding box.center), scale=.25]
		\draw[thick] (2,0) .. controls (2,1.25) and (-2,.25) .. (-2,2);
		\draw[thick] (1,0) .. controls (1,1) and (1.8,.8) .. (2,2);
		\draw[thick] (0,0) .. controls (0,1) and (.8,.8) .. (1,2);
		\draw[thick] (-1,0) .. controls (-1,1) and (-.2,.8) .. (0,2);
		\draw[thick] (-2,0) .. controls (-2,1) and (-1.2,.8) .. (-1,2);
		\draw[thick] (-3,0)--(-3,2);
		\node at (2,1.9){\arrowlines};
		\node at (1,1.9){\arrowlines};
		\node at (0,1.9){\arrowlines};
		\node at (-1,1.9){\arrowlines};
		\node at (-2,1.9){\arrowlines};
		\node[rotate=-180] at (-3,1.9){\arrowlines};
		\draw[thick] (2,2) arc (150:-150:2cm);
		\draw[thick] (1,2) arc (160.5:-160.5:3cm);
		\draw[thick] (0,2) arc (165.5:-165.5:4cm);
		\draw[thick] (-1,2) arc (169:-169.5:5cm);
		\draw[thick] (-2,2) arc (171:30:6cm);
		\draw[thick] (-2,0) arc (-171:-30:6cm);
		\draw[thick] (-3,2) arc (172:29:7cm);
		\draw[thick] (-3,0) arc (-172:-29:7cm);
		\draw[thick,red] (10,4.5) .. controls (10.3,3.7) and (9.2,3.6) .. (9.1,4.1);
		\draw[thick,red] (9.1,-2.1) .. controls (9.1,-1.7) and (10.3,-1.6) .. (10,-2.5);
		\end{tikzpicture}
		=
		\begin{tikzpicture}[baseline=(current bounding box.center), scale=.3]
		\draw[thick] (2,0) .. controls (2,1.25) and (-2,.25) .. (-2,2);
		\draw[thick] (1,0) .. controls (1,1) and (1.8,.8) .. (2,2);
		\draw[thick] (0,0) .. controls (0,1) and (.8,.8) .. (1,2);
		\draw[thick] (-1,0) .. controls (-1,1) and (-.2,.8) .. (0,2);
		\draw[thick] (-2,0) .. controls (-2,1) and (-1.2,.8) .. (-1,2);
		\draw[thick] (-3,0)--(-3,2);
		\node at (2,1.9){\arrowlines};
		\node at (1,1.9){\arrowlines};
		\node at (0,1.9){\arrowlines};
		\node at (-1,1.9){\arrowlines};
		\node at (-2,1.9){\arrowlines};
		\node[rotate=-180] at (-3,1.9){\arrowlines};
		\draw[thick] (2,2) arc (150:-150:2cm);
		\draw[thick] (1,2) arc (160.5:-160.5:3cm);
		\draw[thick] (0,2) arc (165.5:-165.5:4cm);
		\draw[thick] (-1,2) arc (169:-169.5:5cm);
		\draw[thick,red] (-3,2) arc (180:3:.5cm);
		\draw[thick,red] (-3,0) arc (-180:0:.5cm);
		\end{tikzpicture}
		=0
		\end{align*}  
		where the last equality follows from \eqref{eqn-symmetric-group-relations}.
		
		For the resolution term coming from the crossing of intermediate strands, consider a generic intermediate strand. We have

		
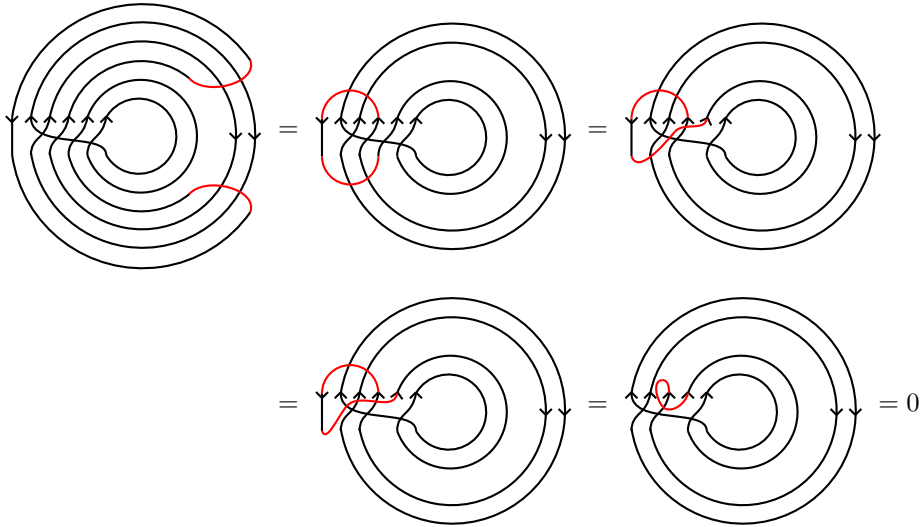
\begin{align*}
		\begin{tikzpicture}[baseline=(current bounding box.center), scale=.25]
		\draw[thick] (-2,0) .. controls (-2,1) and (-1.2,.8) .. (-1,2);
		\draw[thick] (-1,0) .. controls (-1,1) and (-.2,.8) .. (0,2);
		\draw[thick] (0,0) .. controls (0,1) and (.8,.8) .. (1,2);
		\draw[thick] (1,0) .. controls (1,1) and (1.8,.8) .. (2,2);
		\draw[thick] (2,0) .. controls (2,1.25) and (-2,.25) .. (-2,2);
		\draw[thick] (-3,0)--(-3,2);
		\node at (2,1.9){\arrowlines};
		\node at (1,1.9){\arrowlines};
		\node at (0,1.9){\arrowlines};
		\node at (-1,1.9){\arrowlines};
		\node at (-2,1.9){\arrowlines};
		\node[rotate=-180] at (-3,1.9){\arrowlines};
		\node[rotate=-180] at (9.9,1){\arrowlines};
		\node[rotate=-180] at (8.9,1){\arrowlines};
		\draw[thick] (2,2) arc (150:-150:2cm);
		\draw[thick] (1,2) arc (160.5:-160.5:3cm);
		\draw[thick] (0,2) arc (165.5:50.5:4cm);
		\draw[thick] (0,0) arc (-165.5:-50:4cm);
		\draw[thick] (-1,2) arc (169:-169.5:5cm);
		\draw[thick] (-2,2) arc (171:-171:6cm);
		\draw[thick] (-3,2) arc (172:35:7cm);
		\draw[thick] (-3,0) arc (-172:-35:7cm);
		\draw[thick, red] (9.66,5.05) to [out=-70, in=-50] (6.4,4.1);
		\draw[thick, red] (9.66,-3.05) to [out=70, in=50] (6.4,-2.1);
		\end{tikzpicture}
		&=
		\begin{tikzpicture}[baseline=(current bounding box.center), scale=.25]
		\draw[thick] (-2,0) .. controls (-2,1) and (-1.2,.8) .. (-1,2);
		\draw[thick] (-1,0) .. controls (-1,1) and (-.2,.8) .. (0,2);
		\draw[thick] (0,0) .. controls (0,1) and (.8,.8) .. (1,2);
		\draw[thick] (1,0) .. controls (1,1) and (1.8,.8) .. (2,2);
		\draw[thick] (2,0) .. controls (2,1.25) and (-2,.25) .. (-2,2);
		\draw[thick] (-3,0)--(-3,2);
		\node at (2,1.9){\arrowlines};
		\node at (1,1.9){\arrowlines};
		\node at (0,1.9){\arrowlines};
		\node at (-1,1.9){\arrowlines};
		\node at (-2,1.9){\arrowlines};
		\node[rotate=-180] at (-3,1.9){\arrowlines};
		\node[rotate=-180] at (9.9,1){\arrowlines};
		\node[rotate=-180] at (8.9,1){\arrowlines};
		\draw[thick] (2,2) arc (150:-150:2cm);
		\draw[thick] (1,2) arc (160.5:-160.5:3cm);
		\draw[thick] (-1,2) arc (169:-169.5:5cm);
		\draw[thick] (-2,2) arc (171:-171:6cm);
		\draw[thick,red] (-3,2) arc (180:2:1.5cm);
		\draw[thick,red] (-3,0) arc (-180:0:1.5cm);
		\end{tikzpicture}
		=
		\begin{tikzpicture}[baseline=(current bounding box.center), scale=.25]
		\draw[thick] (-2,0) .. controls (-2,1) and (-1.2,.8) .. (-1,2);
		\draw[thick] (-1,0) .. controls (-1,1) and (-.2,.8) .. (0,2);
		\draw[thick] (1,0) .. controls (1,1) and (1.8,.8) .. (2,2);
		\draw[thick] (2,0) .. controls (2,1.25) and (-2,.25) .. (-2,2);
		\draw[thick] (-3,0)--(-3,2);
		\node at (2,1.9){\arrowlines};
		\node[rotate=-20] at (1,1.9){\arrowlines};
		\node at (0,1.9){\arrowlines};
		\node at (-1,1.9){\arrowlines};
		\node at (-2,1.9){\arrowlines};
		\node[rotate=-180] at (-3,1.9){\arrowlines};
		\node[rotate=-180] at (9.9,1){\arrowlines};
		\node[rotate=-180] at (8.9,1){\arrowlines};
		\draw[thick] (2,2) arc (150:-150:2cm);
		\draw[thick] (1,2) arc (160.5:-160.5:3cm);
		\draw[thick] (-1,2) arc (169:-169.5:5cm);
		\draw[thick] (-2,2) arc (171:-171:6cm);
		\draw[thick,red] (-3,2) arc (180:2:1.5cm);
		\draw[thick,red] (-3,0) to[out=-70,in=-130] (-.7,1.3);
		\draw[thick,red] (-.7,1.3) to[out=50, in=-100] (1,2);
		\end{tikzpicture}\\
		&=\begin{tikzpicture}[baseline=(current bounding box.center), scale=.25]
		\draw[thick] (-2,0) .. controls (-2,1) and (-1.2,.8) .. (-1,2);
		\draw[thick] (-1,0) .. controls (-1,1) and (-.2,.8) .. (0,2);
		\draw[thick] (1,0) .. controls (1,1) and (1.8,.8) .. (2,2);
		\draw[thick] (2,0) .. controls (2,1.25) and (-2,.25) .. (-2,2);
		\draw[thick] (-3,0)--(-3,2);
		\node at (2,1.9){\arrowlines};
		\node[rotate=-10] at (1,1.9){\arrowlines};
		\node at (0,1.9){\arrowlines};
		\node at (-1,1.9){\arrowlines};
		\node at (-2,1.9){\arrowlines};
		\node[rotate=-180] at (-3,1.9){\arrowlines};
		\node[rotate=-180] at (9.9,1){\arrowlines};
		\node[rotate=-180] at (8.9,1){\arrowlines};
		\draw[thick] (2,2) arc (150:-150:2cm);
		\draw[thick] (1,2) arc (160.5:-160.5:3cm);
		\draw[thick] (-1,2) arc (169:-169.5:5cm);
		\draw[thick] (-2,2) arc (171:-171:6cm);
		\draw[thick,red] (-3,2) arc (180:2:1.5cm);
		\draw[thick,red] (-3,0) to[out=-70,in=-130] (-1.7,1.3);
		\draw[thick,red] (-1.7,1.3) to[out=50, in=-100] (1,2);
		\end{tikzpicture}
		=
		\begin{tikzpicture}[baseline=(current bounding box.center), scale=.25]
		\draw[thick] (-2,0) .. controls (-2,1) and (-1.2,.8) .. (-1,2);
		\draw[thick] (-1,0) .. controls (-1,1) and (-.2,.8) .. (0,2);
		\draw[thick] (1,0) .. controls (1,1) and (1.8,.8) .. (2,2);
		\draw[thick] (2,0) .. controls (2,1.25) and (-2,.25) .. (-2,2);
		\node at (2,1.9){\arrowlines};
		\node at (1,1.9){\arrowlines};
		\node at (0,1.9){\arrowlines};
		\node at (-1,1.9){\arrowlines};
		\node at (-2,1.9){\arrowlines};
		\node[rotate=-180] at (9.9,1){\arrowlines};
		\node[rotate=-180] at (8.9,1){\arrowlines};
		\draw[thick] (2,2) arc (150:-150:2cm);
		\draw[thick] (1,2) arc (160.5:-160.5:3cm);
		\draw[thick] (-1,2) arc (169:-169.5:5cm);
		\draw[thick] (-2,2) arc (171:-171:6cm);
		\draw[thick,red] (0,2) to[out=90, in=0] (-.35,2.7);
		\draw[thick,red] (-.35,2.7) to[out=180, in=130] (-.35,1.3);
		\draw[thick, red] (-.35,1.3) to[out=-50, in=-100] (1,2);
		\end{tikzpicture}	
		= 0
		\end{align*}
where the second and third equalities follow from a Reidemeister III move, and the fourth is a result of relation \eqref{up down double crossings}. Hence these resolution terms are zero as well. In general, for a resolution term coming from a crossing of intermediate strands, we can first pull the red string above the permutation using Reidemeister III moves, and then pull the red string into the permutation using relation (\ref{up down double crossings}) to get a left twist curl.
		
		Finally, for the resolution term coming from the crossing of intermediate strands the situation is simpler:

		
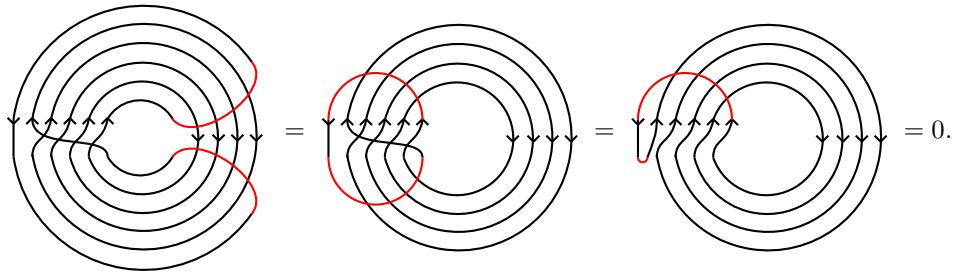
\begin{align*}
		\begin{tikzpicture}[baseline=(current bounding box.center), scale=.25]
		\draw[thick] (-2,0) .. controls (-2,1) and (-1.2,.8) .. (-1,2);
		\draw[thick] (-1,0) .. controls (-1,1) and (-.2,.8) .. (0,2);
		\draw[thick] (0,0) .. controls (0,1) and (.8,.8) .. (1,2);
		\draw[thick] (1,0) .. controls (1,1) and (1.8,.8) .. (2,2);
		\draw[thick] (2,0) .. controls (2,1.25) and (-2,.25) .. (-2,2);
		\draw[thick] (-3,0)--(-3,2);
		\node at (2,1.9){\arrowlines};
		\node at (1,1.9){\arrowlines};
		\node at (0,1.9){\arrowlines};
		\node at (-1,1.9){\arrowlines};
		\node at (-2,1.9){\arrowlines};
		\node[rotate=-180] at (-3,1.9){\arrowlines};
		\node[rotate=-180] at (9.9,1){\arrowlines};
		\node[rotate=-180] at (8.9,1){\arrowlines};
		\node[rotate=-180] at (7.9,1){\arrowlines};
		\node[rotate=-180] at (6.8,1){\arrowlines};
		\draw[thick] (2,2) arc (150:30:2cm);
		\draw[thick] (2,0) arc (-150:-30:2cm);
		\draw[thick] (1,2) arc (160.5:-160.5:3cm);
		\draw[thick] (0,2) arc (165.5:-165.5:4cm);
		\draw[thick] (-1,2) arc (169:-169.5:5cm);
		\draw[thick] (-2,2) arc (171:-171:6cm);
		\draw[thick] (-3,2) arc (172:35:7cm);
		\draw[thick] (-3,0) arc (-172:-35:7cm);
		\draw[thick, red] (9.66,5.05) to [out=-50, in=-50] (5.45,2);
		\draw[thick, red] (9.66,-3.05) to [out=50, in=50] (5.45,0);
		\end{tikzpicture}
		=
		\begin{tikzpicture}[baseline=(current bounding box.center), scale=.25]
		\draw[thick] (-2,0) .. controls (-2,1) and (-1.2,.8) .. (-1,2);
		\draw[thick] (-1,0) .. controls (-1,1) and (-.2,.8) .. (0,2);
		\draw[thick] (0,0) .. controls (0,1) and (.8,.8) .. (1,2);
		\draw[thick] (1,0) .. controls (1,1) and (1.8,.8) .. (2,2);
		\draw[thick] (2,0) .. controls (2,1.25) and (-2,.25) .. (-2,2);
		\draw[thick] (-3,0)--(-3,2);
		\node at (2,1.9){\arrowlines};
		\node at (1,1.9){\arrowlines};
		\node at (0,1.9){\arrowlines};
		\node at (-1,1.9){\arrowlines};
		\node at (-2,1.9){\arrowlines};
		\node[rotate=-180] at (-3,1.9){\arrowlines};
		\node[rotate=-180] at (9.9,1){\arrowlines};
		\node[rotate=-180] at (8.9,1){\arrowlines};
		\node[rotate=-180] at (7.9,1){\arrowlines};
		\node[rotate=-180] at (6.8,1){\arrowlines};
		\draw[thick] (1,2) arc (160.5:-160.5:3cm);
		\draw[thick] (0,2) arc (165.5:-165.5:4cm);
		\draw[thick] (-1,2) arc (169:-169.5:5cm);
		\draw[thick] (-2,2) arc (171:-171:6cm);
		\draw[thick,red] (-3,2) arc (180:0:2.5cm);
		\draw[thick,red] (-3,0) arc (-180:0:2.5cm);
		\end{tikzpicture}
		=
		\begin{tikzpicture}[baseline=(current bounding box.center), scale=.25]
		\draw[thick] (-2,0) .. controls (-2,1) and (-1.2,.8) .. (-1,2);
		\draw[thick] (-1,0) .. controls (-1,1) and (-.2,.8) .. (0,2);
		\draw[thick] (0,0) .. controls (0,1) and (.8,.8) .. (1,2);
		\draw[thick] (1,0) .. controls (1,1) and (1.8,.8) .. (2,2);
		\draw[thick] (-2.5,0) to[out=70, in=-90] (-2,2);
		\draw[thick] (-3,0)--(-3,2);
		\node at (2,1.9){\arrowlines};
		\node at (1,1.9){\arrowlines};
		\node at (0,1.9){\arrowlines};
		\node at (-1,1.9){\arrowlines};
		\node at (-2,1.9){\arrowlines};
		\node[rotate=-180] at (-3,1.9){\arrowlines};
		\node[rotate=-180] at (9.9,1){\arrowlines};
		\node[rotate=-180] at (8.9,1){\arrowlines};
		\node[rotate=-180] at (7.9,1){\arrowlines};
		\node[rotate=-180] at (6.8,1){\arrowlines};
		\draw[thick] (1,2) arc (160.5:-160.5:3cm);
		\draw[thick] (0,2) arc (165.5:-165.5:4cm);
		\draw[thick] (-1,2) arc (169:-169.5:5cm);
		\draw[thick] (-2,2) arc (171:-171:6cm);
		\draw[thick,red] (-3,2) arc (180:0:2.5cm);
		\draw[thick,red] (-3,0) arc (-180:0:0.25cm);
		\end{tikzpicture}
		=0.
		\end{align*}	

		Hence all the resolution terms are zero. This leaves us with

		
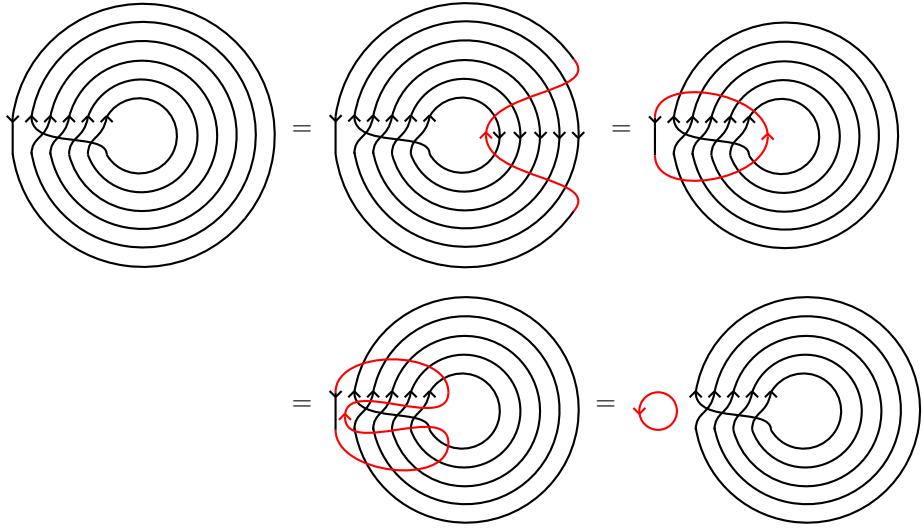
\begin{align*}
		\begin{tikzpicture}[baseline=(current bounding box.center), scale=.25]
		\draw[thick] (-2,0) .. controls (-2,1) and (-1.2,.8) .. (-1,2);
		\draw[thick] (-1,0) .. controls (-1,1) and (-.2,.8) .. (0,2);
		\draw[thick] (0,0) .. controls (0,1) and (.8,.8) .. (1,2);
		\draw[thick] (1,0) .. controls (1,1) and (1.8,.8) .. (2,2);
		\draw[thick] (2,0) .. controls (2,1.25) and (-2,.25) .. (-2,2);
		\draw[thick] (-3,0)--(-3,2);
		\node at (2,1.9){\arrowlines};
		\node at (1,1.9){\arrowlines};
		\node at (0,1.9){\arrowlines};
		\node at (-1,1.9){\arrowlines};
		\node at (-2,1.9){\arrowlines};
		\node[rotate=-180] at (-3,1.9){\arrowlines};
		\draw[thick] (2,2) arc (150:-150:2cm);
		\draw[thick] (1,2) arc (160.5:-160.5:3cm);
		\draw[thick] (0,2) arc (165.5:-165.5:4cm);
		\draw[thick] (-1,2) arc (169:-169.5:5cm);
		\draw[thick] (-2,2) arc (171:-171:6cm);
		\draw[thick] (-3,2) arc (172:-172:7cm);
		\end{tikzpicture}
		&=
		\begin{tikzpicture}[baseline=(current bounding box.center), scale=.25]
		\draw[thick] (-2,0) .. controls (-2,1) and (-1.2,.8) .. (-1,2);
		\draw[thick] (-1,0) .. controls (-1,1) and (-.2,.8) .. (0,2);
		\draw[thick] (0,0) .. controls (0,1) and (.8,.8) .. (1,2);
		\draw[thick] (1,0) .. controls (1,1) and (1.8,.8) .. (2,2);
		\draw[thick] (2,0) .. controls (2,1.25) and (-2,.25) .. (-2,2);
		\draw[thick] (-3,0)--(-3,2);
		\node at (2,1.9){\arrowlines};
		\node at (1,1.9){\arrowlines};
		\node at (0,1.9){\arrowlines};
		\node at (-1,1.9){\arrowlines};
		\node at (-2,1.9){\arrowlines};
		\node[rotate=-180] at (-3,1.9){\arrowlines};
		\node[rotate=-180] at (9.9,1){\arrowlines};
		\node[rotate=-180] at (8.9,1){\arrowlines};
		\node[rotate=-180] at (7.9,1){\arrowlines};
		\node[rotate=-180] at (6.8,1){\arrowlines};
		\node[rotate=-180] at (5.7,1){\arrowlines};
		\node at (5,1){\redarrowlines};
		\draw[thick] (2,2) arc (150:-150:2cm);
		\draw[thick] (1,2) arc (160.5:-160.5:3cm);
		\draw[thick] (0,2) arc (165.5:-165.5:4cm);
		\draw[thick] (-1,2) arc (169:-169.5:5cm);
		\draw[thick] (-2,2) arc (171:-171:6cm);
		\draw[thick] (-3,2) arc (172:35:7cm);
		\draw[thick] (-3,0) arc (-172:-35:7cm);
		\draw[thick, red] (9.66,5.05) to [out=-50, in=90] (5,1);
		\draw[thick, red] (9.66,-3.05) to [out=50, in=-90] (5,1);
		\end{tikzpicture}
		=
		\begin{tikzpicture}[baseline=(current bounding box.center), scale=.25]
		\draw[thick] (-2,0) .. controls (-2,1) and (-1.2,.8) .. (-1,2);
		\draw[thick] (-1,0) .. controls (-1,1) and (-.2,.8) .. (0,2);
		\draw[thick] (0,0) .. controls (0,1) and (.8,.8) .. (1,2);
		\draw[thick] (1,0) .. controls (1,1) and (1.8,.8) .. (2,2);
		\draw[thick] (2,0) .. controls (2,1.25) and (-2,.25) .. (-2,2);
		\draw[thick] (-3,0)--(-3,2);
		\node at (2,1.9){\arrowlines};
		\node at (1,1.9){\arrowlines};
		\node at (0,1.9){\arrowlines};
		\node at (-1,1.9){\arrowlines};
		\node at (-2,1.9){\arrowlines};
		\node[rotate=-180] at (-3,1.9){\arrowlines};
		\node at (3,1){\redarrowlines};
		\draw[thick] (2,2) arc (150:-150:2cm);
		\draw[thick] (1,2) arc (160.5:-160.5:3cm);
		\draw[thick] (0,2) arc (165.5:-165.5:4cm);
		\draw[thick] (-1,2) arc (169:-169.5:5cm);
		\draw[thick] (-2,2) arc (171:-171:6cm);
		\draw[thick, red] (-3,2) to [out=90, in=90] (3,1);
		\draw[thick, red] (-3,0) to [out=-90, in=-90] (3,1);
		\end{tikzpicture}\\
		&=
		\begin{tikzpicture}[baseline=(current bounding box.center), scale=.25]
		\draw[thick] (-2,0) .. controls (-2,1) and (-1.2,.8) .. (-1,2);
		\draw[thick] (-1,0) .. controls (-1,1) and (-.2,.8) .. (0,2);
		\draw[thick] (0,0) .. controls (0,1) and (.8,.8) .. (1,2);
		\draw[thick] (1,0) .. controls (1,1) and (1.8,.8) .. (2,2);
		\draw[thick] (2,0) .. controls (2,1.25) and (-2,.25) .. (-2,2);
		\draw[thick] (-3,0)--(-3,2);
		\node at (2,1.9){\arrowlines};
		\node at (1,1.9){\arrowlines};
		\node at (0,1.9){\arrowlines};
		\node at (-1,1.9){\arrowlines};
		\node at (-2,1.9){\arrowlines};
		\node[rotate=-180] at (-3,1.9){\arrowlines};
		\node at (-2.5,.7){\redarrowlines};
		\draw[thick] (2,2) arc (150:-150:2cm);
		\draw[thick] (1,2) arc (160.5:-160.5:3cm);
		\draw[thick] (0,2) arc (165.5:-165.5:4cm);
		\draw[thick] (-1,2) arc (169:-169.5:5cm);
		\draw[thick] (-2,2) arc (171:-171:6cm);
		\draw[thick, red] (-3,2) to [out=90, in=90] (3,2);
		\draw[thick, red] (-3,0) to [out=-90, in=-90] (3,-.7);
		\draw[thick,red] (3,2) to[out=-90, in=90] (-2.5,.7);
		\draw[thick,red] (-2.5,.7) to[out=-90, in=90] (3,-.7);
		\end{tikzpicture}
		=
		\begin{tikzpicture}[baseline=(current bounding box.center), scale=.25]
		\draw[thick] (-2,0) .. controls (-2,1) and (-1.2,.8) .. (-1,2);
		\draw[thick] (-1,0) .. controls (-1,1) and (-.2,.8) .. (0,2);
		\draw[thick] (0,0) .. controls (0,1) and (.8,.8) .. (1,2);
		\draw[thick] (1,0) .. controls (1,1) and (1.8,.8) .. (2,2);
		\draw[thick] (2,0) .. controls (2,1.25) and (-2,.25) .. (-2,2);
		\node at (2,1.9){\arrowlines};
		\node at (1,1.9){\arrowlines};
		\node at (0,1.9){\arrowlines};
		\node at (-1,1.9){\arrowlines};
		\node at (-2,1.9){\arrowlines};
		\node[rotate=180] at (-5,1){\redarrowlines};
		\draw[thick] (2,2) arc (150:-150:2cm);
		\draw[thick] (1,2) arc (160.5:-160.5:3cm);
		\draw[thick] (0,2) arc (165.5:-165.5:4cm);
		\draw[thick] (-1,2) arc (169:-169.5:5cm);
		\draw[thick] (-2,2) arc (171:-171:6cm);
		\draw[thick,red] (-4,1) circle (1cm);
		\end{tikzpicture}
		\end{align*}
		and a counter-clockwise oriented bubble is equal to $1$ by the defining relation (\ref{eqn-symmetric-group-relations}). These diagrammatic arguments clearly hold for arbitrary $k > 1$. Hence the action of $\omega_{(1,0)}$ on $\alpha_k$ for $k\neq1$ is trivial.

		This concludes the proof of the base case $\length{\mu}=1$. Now suppose $\op \cdot \alpha_\mu=\alpha_\mu$ for some $\mu\in\oddpartitions$ such that $\length{\mu}=m-1$, and let $n$ be a positive integer. Then
		\begin{align*}
		0&=[\sqrt{2}\omega_{-(2n+1),0}, \op]\cdot \alpha_\mu = \sqrt{2}\omega_{-(2n+1),0} \cdot(\op \cdot \alpha_\mu)- \op \cdot (\sqrt{2}\omega_{-(2n+1),0}\cdot \alpha_\mu)\\
		& = \sqrt{2}\omega_{-(2n+1),0} \cdot \alpha_\mu - \op \cdot (\sqrt{2}\omega_{-(2n+1),0}\cdot \alpha_\mu) \\
		&= \alpha_{(\mu,2n+1)} - \op \cdot \alpha_{(\mu,2n+1)},	
		\end{align*}
and the result follows by induction.
				
Hence if $\mu$ doesn't contain any parts of size $1$, then 
$\op \cdot \alpha_\mu=\alpha_{\mu}.$

Now suppose $\gamma$ is an odd partition without parts of size $1$. We will prove that  
\begin{equation*}
\op\cdot \alpha_{(\gamma,1^k)} = \alpha_{(\gamma,1^k)} + 2k\alpha_{(\gamma,1^{k-1})}
\end{equation*}
by induction on $k$. The base case $k=0$ was proved above. Suppose the formula holds for $\alpha_{(\gamma,1^k)}$. 
	\begin{align*}
	\op\cdot \alpha_{(\gamma,1^{k+1})}&=\op\cdot \big(\alpha_{(\gamma,1^k)}\alpha_1-2|(\gamma,1^k)|\alpha_{(\gamma,1^k)}\big)\hspace{3mm} \text{  by Lemma \ref{alphamu1}} \\
	&= \alpha_1\op \cdot\alpha_{(\gamma,1^k)} +2\op \cdot \alpha_{(\gamma,1^k)} -2|(\gamma,1^k)|\op \cdot \alpha_{(\gamma,1^k)}\\
	&= (\alpha_1+2-2|(\gamma,1^k)|) \op \cdot \alpha_{(\gamma,1^k)}\\
	&= (\alpha_1+2-2|(\gamma,1^k)|) (\alpha_{(\gamma,1^k)}+2k \alpha_{(\gamma,1^{k-1})} ) \hspace{3mm}\text{   by the inductive hypothesis}\\
	&= (\alpha_1+2-2|(\gamma,1^k)|)\alpha_{(\gamma,1^k)} +2k(\alpha_1+2-2|(\gamma,1^k)|)\alpha_{(\gamma,1^{k-1})}\\
	&= \alpha_{(\gamma,1^{k+1})} +4 \alpha_{(\gamma,1^k)} + 2k \alpha_{(\gamma,1^k)} \hspace{3mm} \text{ by Lemma \ref{alphamu1}}\\
	&= \alpha_{(\gamma,1^{k+1})}+ 2(k+1) \alpha_{(\gamma,1^k)},
	\end{align*}
and the result follows after the identification $\alpha_\mu \rightarrow 2^{\length{\mu}}\shiftedpowersum_{\mu}$.

For the action of $\omega_{0,3}$, note that this element acts on the center as multiplication by itself. Therefore $$\omega_{0,3}\cdot \alpha_{\mu}=\alpha_{3}\alpha_{\mu}+\alpha_{(1,1)}\alpha_{\mu},$$
		$$-2\omega_{0,3}\cdot 2^{\length{\mu}}\shiftedpowersum_\mu= 2^{\length{\mu}+1}\shiftedpowersum_3\shiftedpowersum_{\mu}+2^{\length{\mu}+2}\shiftedpowersum_{(1,1)}\shiftedpowersum_\mu \qedhere.$$
\end{proof}

\bibliographystyle{amsalpha}
\bibliography{spinHeisenCenter}

\end{document}